\documentclass[12pt,reqno]{amsart}
\usepackage[dvips]{graphicx}
\usepackage{verbatim}
\usepackage{hyperref}
\usepackage{color}
\usepackage{amssymb}
\usepackage{latexsym}
\usepackage{amsmath}
\usepackage{bbm}
\usepackage[mathscr]{euscript}
\usepackage{amsthm}
\usepackage{amsrefs-siam}
\renewcommand{\leq}{\leqslant}
\renewcommand{\geq}{\geqslant}
\newcommand{\R}{\mathbbm{R}} 
\newcommand{\N}{\mathbbm{N}}
\newcommand{\Z}{\mathbbm{Z}}
\providecommand{\ab}[1]{\vert #1\vert} 
\providecommand{\abs}[1]{\biggl\vert #1 \biggr\vert}
\providecommand{\norma}[1]{\Vert #1 \Vert}
\providecommand{\Norma}[1]{\Bigl\Vert #1 \Bigr\Vert}
\newcommand{\Sh}{\mathcal{S}}
\newcommand{\la}{\lambda}
\newcommand{\La}{\Lambda}
\newcommand{\eps}{\varepsilon}
\newcommand{\Sph}{\mathbb{S}}
\newcommand{\supp}{\operatorname{supp}}
\newcommand{\cp}{\mathcal{C}}
\newcommand{\eval}{\Big\vert}
\newcommand{\te}{\theta}
\newcommand{\vphi}{\varphi}
\newcommand{\ds}{\displaystyle}
\providecommand{\suma}[2]{\sum\limits_{#1}^{#2}}
\newcommand{\esssup}{\operatornamewithlimits{ess\,sup}}
\providecommand{\NOrma}[1]{\biggl\Vert #1 \biggr\Vert} 
\newcommand{\one}{\mathbbm{1}}
\newcommand{\pvector}[1]{% column vector
	\begin{pmatrix}
		#1
	\end{pmatrix}}
\newcommand{\ddirac}[1]{\,\boldsymbol{\delta}\!\pvector{#1}\!} % Dirac's delta
\renewcommand{\d}{\,{\rm d}} % differential
\newcommand{\sphcp}{\mathscr{C}}
\newcommand{\Hyp}{{\overline{\mathcal{H}}}^3}
\newcommand{\hyp}{\mathcal H^3}

\numberwithin{equation}{section}
\theoremstyle{plain}
\newtheorem{teo}{Theorem}[section]
\newtheorem{cor}[teo]{Corollary}
\newtheorem{lemma}[teo]{Lemma}
\newtheorem{prop}[teo]{Proposition}

\theoremstyle{definition}
\newtheorem{define}[teo]{Definition}

\newtheorem{remark}[teo]{Remark}

\makeatletter
\def\@tocline#1#2#3#4#5#6#7{\relax
	\ifnum #1>\c@tocdepth % then omit
	\else
	\par \addpenalty\@secpenalty\addvspace{#2}%
	\begingroup \hyphenpenalty\@M
	\@ifempty{#4}{%
		\@tempdima\csname r@tocindent\number#1\endcsname\relax
	}{%
	\@tempdima#4\relax
}%
\parindent\z@ \leftskip#3\relax \advance\leftskip\@tempdima\relax
\rightskip\@pnumwidth plus4em \parfillskip-\@pnumwidth
#5\leavevmode\hskip-\@tempdima
\ifcase #1
\or\or \hskip 1em \or \hskip 2em \else \hskip 3em \fi%
#6\nobreak\relax
\dotfill\hbox to\@pnumwidth{\@tocpagenum{#7}}\par
\nobreak
\endgroup
\fi}
\makeatother

\begin{document}
\title[Existence of extremals for the one-sheeted hyperboloid]{Existence of extremals for a Fourier 
restriction inequality on the one--sheeted hyperboloid}
\author{Ren\'e Quilodr\'an}
\address{Ren\'e Quilodr\'an.}
\email{rquilodr@dim.uchile.cl}
\date{July, 2022}

\begin{abstract}
We prove the existence of functions that extremize the endpoint $L^2$ to $L^4$ adjoint Fourier restriction 
inequality on the one-sheeted
hyperboloid in Euclidean space $\R^4$ and that, taking symmetries into consideration, any extremizing sequence has a subsequence that converges to an extremizer.
\end{abstract}
\maketitle
\setcounter{tocdepth}{1}
\begin{center}
	\begin{minipage}[t]{0.85\linewidth}
		\tableofcontents
	\end{minipage}
\end{center}

%\newpage
%\tableofcontents
\section{Introduction}
In seminal paper \cite{Str} R. Strichartz addressed the adjoint restriction problem of the Fourier transform to $d-1$ dimensional quadric submanifolds of Euclidean space $\R^d$, establishing the necessary and sufficient conditions on $p$ such that an $L^2\to L^p$ estimate holds. Recently, there has been interest in studying the existence of extremizers and the sharp $L^2\to L^p$ estimates for adjoint Fourier restriction operators and progress has been made in the case of quadric curves and surfaces:  the paraboloid and parabola \cites{Fo,HZ}, the cone \cites{Ca,Fo,Ra}, the sphere and circle \cites{CFOT,CaOeS,CS,Fo2,FLS,OQ19a,Scircle}, the two-sheeted hyperboloid and hyperbola \cites{COS,COeSS18,RQ2} and the saddle \cites{COS0,DMPS}. The study of such sharp $L^2$ to $L^p$ estimates 
is intimately related to 
the study of extremizers and sharp constants for Strichartz estimates for classical partial 
differential equations, such as the Schr\"odinger, hyperbolic Schr\"odinger, wave and Klein--Gordon equations. In this note we address the case of the one-sheeted hyperboloid in $\R^4$, which is related to the so called Klein--Gordon equation with imaginary mass.

Let $\mathcal H^3$ denote the upper half of the three dimensional one-sheeted hyperboloid in $\R^4$,
\[\mathcal H^3=\bigl\{(x,\sqrt{\ab{x}^2-1})\colon x\in\R^3,\, \ab{x}\geq 1\bigr\},\]
equipped with the measure $\mu$ with density $\one_{\{\ab{x}> 1\}}(x)\ddirac{t-\sqrt{\ab{x}^2-1}}\frac{\d t\d 
x}{\sqrt{\ab{x}^2-1}}$, defined by duality as
\begin{equation}\label{def-measure-mu}
\int_{\mathcal H^3} g(x,t)\d\mu(x,t)=\int_{\{y\in\R^3:\ab{y}> 
	1\}}g(y,\sqrt{\ab{y}^2-1})\frac{\d y}{\sqrt{\ab{y}^2-1}}
\end{equation}
for all $g\in \Sh(\R^4)$.

A function $f:\mathcal H^3\to\R$ can be identified with a function from $\R^3$ to $\R$, using the orthogonal projection from\footnote{Strictly speaking, it is identified with a function with domain $\{x\in\R^4\colon \ab{x}\geq 1\}$ but we will ignore this minor point and, whenever necessary, it will be understood that $f$ is extended to be equal to zero inside the unit ball. We could have chosen to write our operator as acting on a weighted $L^2$ space of Euclidean space, but we will take this geometric point of view instead.} $\hyp$ to $\R^3\times\{0\}$,  and in what 
follows we do so. We
denote the $L^p(\mathcal H^3,\mu)$ norm of a function $f$ on $\mathcal H^3$ by 
$\norma{f}_{L^p(\mathcal H^3)}$, $\norma{f}_{L^p(\mu)}$ or simply $\norma{f}_{L^p},\,\norma{f}_{p}$ if it is clear 
from context.

The Fourier extension operator on the hyperboloid $\mathcal H^3$, also known as the adjoint Fourier restriction operator, is given by 
\begin{equation}
 \label{fourier-extension-operator}
 Tf(x,t)=\int_{\{y\in\R^3:\ab{y}> 1\}} e^{ix\cdot y} e^{it\sqrt{\ab{y}^2-1}} 
 f(y,\sqrt{\ab{y}^2-1})\frac{\d y}{\sqrt{\ab{y}^2-1}},
\end{equation}
where $(x,t)\in\R^3\times \R$ and $f\in\mathcal S(\R^4)$. Note that 
$Tf(x,t)=\widehat{f\mu}(-x,-t)$, with the Fourier
transform in $\R^{4}$ defined by $\hat g(x,t)=\int_{\R^4}e^{-i(x\cdot y+ts)}g(y,s)\d y\d s$.

Strichartz proved in \cite{Str} that for all $\frac{10}{3}\leq p\leq 4$ there exists 
$C_p<\infty$ such that for all
$f\in L^2(\mathcal H^3)$ the
following estimate for $Tf$ holds
\begin{equation}
 \label{adjoint-restriction-one-sheeted}
 \norma{Tf}_{L^p(\R^4)}\leq \mathbf{H}_{p}\norma{f}_{L^2(\mathcal H^3)},
\end{equation}
where $\mathbf{H}_{p}<\infty$ denotes the best constant in \eqref{adjoint-restriction-one-sheeted},
\begin{equation}\label{D-sharp-constant}
\mathbf{H}_{p}=\sup_{0\neq f\in L^2(\mathcal 
H^3)}\frac{\norma{Tf}_{L^p(\R^4)}}{\norma{f}_{L^2(\mathcal H^3)}}.
\end{equation}
The (full) one-sheeted hyperboloid is defined by
\[\overline{\mathcal H}^3:=\{(x,t)\in\R^3\times \R\colon t^2=\ab{x}^2-1,\, \ab{x}\geq 
1\},\]
and we endow it with the Lorentz invariant measure $\bar
\mu$ with density
\begin{align*}
\d\bar
\mu(x,t)&=\one_{\{\ab{x}>1 \}}\ddirac{t-\sqrt{\ab{x}^2-1}}\frac{\d t\d x}{\sqrt{\ab{x}^2-1}}
+\one_{\{\ab{x}>1 \}}\ddirac{t+\sqrt{\ab{x}^2-1}}\frac{\d t\d x}{\sqrt{\ab{x}^2-1}}\\
&=:\d\mu_+(x,t)+\d\mu_-(x,t).
\end{align*} 
Here $\mu_+$ equals $\mu$ as in \eqref{def-measure-mu}, and $\mu_-$ equals the reflection of $\mu$ via the reflection map $(x,t)\mapsto(-x,-t)$. The adjoint Fourier restriction operator on $\overline{\mathcal H}^3$ is  

\begin{align} \label{full-fourier-extension-operator}
\begin{split}
 \overline{T}f(x,t)=\widehat{f\bar\mu}(-x,-t)&=\int_{\Hyp} e^{i(x\cdot y+ts)}
f(y,s)\d\bar{\mu}(y,s)\\ 
&=\int_{\{y\in\R^3:\ab{y}> 1\}} e^{ix\cdot y} 
 e^{it\sqrt{\ab{y}^2-1}}
 f_{+}(y)\frac{\d y}{\sqrt{\ab{y}^2-1}}\\
 &\quad+\int_{\{y\in\R^3:\ab{y}> 1\}} e^{ix\cdot y} e^{-it\sqrt{\ab{y}^2-1}}
 f_{-}(y)\frac{\d y}{\sqrt{\ab{y}^2-1}},
\end{split}
\end{align}
where $f=f_++f_-$, the function $f_+$ is supported on the upper half of the one-sheeted hyperboloid, 
$\mathcal H^3$, and the function $f_-$, on the
lower half, $-\mathcal H^3$, and we have identified their domains with $\R^3$ via the orthogonal projection as stated before. We see that $\overline{T}f(x,t)=Tf_+(x,t)+Tf_-(x,-t)$.

The triangle inequality and \eqref{adjoint-restriction-one-sheeted} imply that for 
$\frac{10}{3}\leq 
p\leq 4$ the following estimate holds
\begin{equation}\label{sharp-double-sheeted}
\norma{\overline{T}f}_{L^p(\R^4)}\leq \overline{\mathbf 
	H}_{p}\norma{f}_{L^2(\overline{\mathcal{H}}^3)},
\end{equation}
where $\overline{\mathbf{H}}_{p}<\infty$ is the sharp constant
\begin{equation}\label{eq:endpoint_nonendpoint_full_hyp}
\overline{\mathbf{H}}_{p}=\sup_{0\neq f\in L^2(\overline{\mathcal 
	H}^3)}\frac{\norma{\overline{T}f}_{L^p(\R^4)}}{\norma{f}_{L^2(\overline{\mathcal H}^3)}}.
\end{equation}

The Lorentz group on $\R^4$, denoted $\mathcal L$, preserves $\overline{\mathcal{H}}^3$, $\bar{\mu}$, and 
acts on functions on $\overline{\mathcal{H}}^3$ by composition: $L^\ast f(x,t):=f(L(x,t))$, 
$L\in\mathcal L$ (see Section \ref{sec:lorentz_invariance} for more details). In particular, we have $\norma{f}_{L^q(\overline{\mathcal H}^3)}=\norma{L^\ast f}_{L^q(\overline{\mathcal H}^3)}$ and $\norma{Tf}_{L^p(\R^4)}=\norma{T(L^\ast f)}_{L^p(\R^4)}$, for all $p,q\in[1,\infty]$.

\begin{define}\label{def:extremizer}
An \textit{extremizer} (or \textit{maximizer}) for \eqref{adjoint-restriction-one-sheeted} is a function $0\neq f\in L^2(\hyp)$ that satisfies $\norma{Tf}_{L^p(\R^4)}=\mathbf{H}_p\norma{f}_{L^2(\hyp)}$. An $L^2$-normalized \textit{extremizing sequence} for \eqref{adjoint-restriction-one-sheeted} $\{f_n\}_n\subset L^2(\hyp)$ is such that $\norma{f_n}_{L^2(\hyp)}=1$ and $\norma{Tf_n}_{L^p(\R^4)}\to \mathbf{H}_p$, as $n\to\infty$. A corresponding definition holds for extremizers and extremizing sequences for \eqref{sharp-double-sheeted}.
\end{define}

This paper is devoted to the study of the sharp instances of \eqref{adjoint-restriction-one-sheeted} and \eqref{sharp-double-sheeted} in the endpoint case $p=4$, that is, the inequalities
\begin{eqnarray}
\label{sharp_L4_hyp}
\norma{{T}f}_{L^4(\R^4)}&\leq {\mathbf 
	H}_{4}\norma{f}_{L^2({\mathcal{H}}^3)},\\
\label{sharp_L4_double_hyp}
\norma{\overline{T}g}_{L^4(\R^4)}&\leq \overline{\mathbf 
	H}_{4}\norma{g}_{L^2(\overline{\mathcal{H}}^3)},
\end{eqnarray}
and our main results concern the existence of extremizers as well as the precompactness of extremizing sequences.  The statements of the main results of this paper are as follows.

\begin{teo}
 \label{thm:main-theorem}
 There exists an extremizer in $L^2(\mathcal H^3)$ for inequality 
 \eqref{sharp_L4_hyp}. Moreover,
for every $L^2$-normalized complex valued extremizing sequence $\{f_n\}_{n}$ for \break\eqref{sharp_L4_hyp}, there exist a subsequence $\{f_{n_k}\}_{k}$ and a sequence 
$\{(x_k,t_k)\}_{k}\subset
\R^4$ such that $\{e^{ix_k\cdot y}e^{it_k\sqrt{\ab{y}^2-1}}f_{n_k}\}_k$ is 
convergent in $L^2(\mathcal H^3)$.
\end{teo}
\begin{teo}
 \label{thm:main-theorem-2}
 There exists an extremizer in $L^2(\Hyp)$ for inequality 
 \eqref{sharp_L4_double_hyp}. Moreover,
for every $L^2$-normalized complex valued extremizing sequence $\{f_n\}_{n}$ for \eqref{sharp_L4_double_hyp}, there exist a subsequence $\{f_{n_k}\}_{k}$ and 
sequences $\{\xi_k\}_{k}\subset
\R^4$ and $\{L_k\}_k\subset \mathcal L$ such that $\{e^{i\xi_k\cdot \xi}L_k^\ast f_{n_k}\}_k$ is 
convergent in $L^2(\overline{\mathcal H}^3)$.
\end{teo}

In the statement of the theorems we are writing $e^{ix_k\cdot y}e^{it_k\sqrt{\ab{y}^2-1}}f_{n_k}$ for the function $y\mapsto e^{ix_k\cdot y}e^{it_k\sqrt{\ab{y}^2-1}}f_{n_k}(y)$ and $e^{i\xi_k\cdot \xi}L_k^\ast f_{n_k}$ for the function $\xi\mapsto e^{i\xi_k\cdot \xi}f_{n_k}(L_k\xi)$.

\begin{remark}
	Note the qualitative difference regarding existence of extremizers between the one-sheeted 
	hyperboloid and the two-sheeted hyperboloid (or its upper sheet) equipped with its Lorentz invariant measure, which are defined respectively by 
	\[\{(x,t)\in\R^4\times\R\colon t^2=\ab{x}^2+1\},\, \bigr(\delta(t-\sqrt{\ab{x}^2+1})+\delta(t+\sqrt{\ab{x}^2+1})\bigr)\frac{\d t\d x}{\sqrt{\ab{x}^2+1}},\]
	both of which can be considered as "perturbations" of the cone. It was shown in \cite{RQ2} that for the $L^2$ to 
	$L^4(\R^4)$ 
	inequality on the two-sheeted hyperboloid and on its upper sheet, extremizers do not exist and the best constant was calculated explicitly. On the other hand, for the $L^2$ to 	$L^4(\R^4)$ inequality on the cone, extremizers exist, their exact form was obtained and the best constant was calculated (see \cite{Ca}). 

	We note that the results in \cite{FVV2} do not apply to the case of the hyperboloid due to the lack of scale invariance, but information can be obtained from the arguments therein, although we will not go that route. See the discussion in \cite{RQ2}*{Section 2} for some details in the related context of the two-sheeted hyperboloid.

	We take this opportunity to indicate a \textbf{correction} to \cite{RQ2}*{Theorem 1.2 \& Proposition 7.5}, where the best constant for the two-sheeted hyperboloid in $\R^2$ in the $L^2\to L^6$ adjoint Fourier restriction inequality, there denoted $\bar{\mathbf{H}}_{2,6}$, is incorrect. Details can be found in version 3 of \cite{RQ2} available at \url{www.arxiv.org}.
\end{remark}

The convolution form of inequalities \eqref{sharp_L4_hyp} and \eqref{sharp_L4_double_hyp}, obtained via Plancherel's theorem, tells us that in both cases,  $\mathcal H^3$ and 
$\overline{\mathcal H}^3$, there exist nonnegative real valued extremizers, 
and the symmetrization method 
used in \cite{Fo2}, or the one in \cite{OeSQ18}, can be adapted to show that if a function $f$ is a nonnegative real valued extremizer for $\overline{T}$  
on $\overline{\mathcal{H}}^3$ then $f$ is 
necessarily an even function: $f(x,t)=f(-x,-t)$, for $\bar{\mu}$-a.e. $(x,t)\in\Hyp$. We discuss the details in Section \ref{sec:lorentz_invariance}.

It would be of interest to treat the endpoint $p=\frac{10}{3}$ as well, and more generally to study the existence of extremizers at the endpoint and non-endpoint cases for all\footnote{When  $d=1$ the one-sheeted hyperboloid coincides with the two-sheeted hyperboloid after a $90^\circ$ rotation, and the later has been studied in \cite{COS}. They consider only one of the two branches but it is not difficult to see that the existence argument in the non-endpoint case carries through to the two branches. On the other hand, an argument is needed to settle the endpoint $p=6$ for two branches (this is also the case when $d=2$ and $p=6$ as clarified in the {correction} to \cite{RQ2} alluded to before).} $d\geq 2$, as was recently done for non-endpoint cases of the two-sheeted hyperboloid in \cites{COS,COeSS18}. Our analysis here extends the known results on sharp Fourier extension inequalities for quadric manifolds as studied in Strichartz paper \cite{Str}.

\subsection{Organization of the paper and outline of the proofs of the main theorems}

From now on, references to the sharp inequalities \eqref{adjoint-restriction-one-sheeted} and \eqref{sharp-double-sheeted} are understood with $p=4$, unless it is explicitly said otherwise. 

An important tool in our proofs is \cite{FVV}*{Proposition 1.1} which we include next for the convenience of the reader.

\begin{prop}
\label{prop:key-prop-from-fvv}
Let $\mathbbm H$ be a Hilbert space and $S:\mathbbm{H}\to L^p(\R^d)$ be a continuous linear operator, for some $p\in (2,\infty)$. Let $\{f_n\}_{n}\subset \mathbbm{H}$ be such that:
\begin{itemize}
 \item [(i)]$\norma{f_n}_{\mathbbm{H}}=1$,
 \item [(ii)]$\lim\limits_{n\to\infty}\norma{Sf_n}_{L^p(\R^d)}=\norma{S}_{\mathbbm H\to L^p(\R^d)}$,
 \item [(iii)]$f_n\rightharpoonup f$ and $f\neq 0$,
 \item [(iv)]$Sf_n\to Sf$ a.e. in $\R^d$.
\end{itemize}
Then $f_n\to f$ in $\mathbbm H$. In particular, 
$\norma{f}_{\mathbbm H}=1$ and $\norma{Sf}_{L^p(\R^d)}=\norma{S}_{\mathbbm H\to L^p(\R^d)}.$
\end{prop}

To prove Theorem \ref{thm:main-theorem} we apply Proposition \ref{prop:key-prop-from-fvv} with $p=d=4$, $\mathbbm{H}$ equals to $L^2(\hyp)$ and $S$ equals $T$. We need 
to verify (iii) and (iv), as (i) and (ii) are immediate for a normalized extremizing sequence. We
 handle (iv) as in \cite{RQ1}*{Prop. 8.3}, \cite{FVV2}. To prove (iii) we will see that the only way it can fail, the failure being that a weak limiting function equals zero, is that the extremizing sequence
concentrates at infinity, which is defined as follows for $\hyp$, with an analogous definition for $\Hyp$.

\begin{define}\label{def:conv-inf}
	We say that the sequence $\{f_n\}_{n}\subset L^2(\mathcal H^3)$ concentrates at infinity if
	$\inf\limits_n\norma{f_n}_{L^2(\mathcal H^3)}>0$ and for every $\eps,R>0$ there
	exists $N\in\N$
	such  that for all $n\geq N$ 
	\[\norma{f_n\one_{\{\ab{y}\leq R\}}}_{L^2(\hyp)}<\eps,\]
	where, as mentioned before, we are identifying a function on $\mathcal H^3$ with a function on $\{y\in\R^3:\ab{y}\geq 1\}$.
\end{define}

\noindent Finally, to discard the possibility of concentration at infinity we will use a comparison argument with the cone.

In the case of the full one-sheeted hyperboloid $\Hyp$ there is the addition of Lorentz invariance, and our proof will require additional steps when compared to the case of the upper half, $\hyp$. Because of this, in addition to the use of Proposition \ref{prop:key-prop-from-fvv} and a comparison to the double cone, we will use a concentration-compactness argument to be able to discard concentration at infinity.

More in detail, the organization of the paper is as follows. In Section \ref{sec:calculation_convolution} we explicitly calculate the double convolution $\mu\ast\mu$ which we use in Section \ref{sec:compare_cone} to prove the \textit{strict} lower bounds
\begin{equation}\label{a_priori_bounds}
\mathbf{H}_4>(2\pi)^{5/4},\qquad \overline{\mathbf{H}}_4>\Bigl(\frac{3}{2}\Bigr)^{1/4}(2\pi)^{5/4}, 
\end{equation}
which tell us that the best constant for the adjoint Fourier restriction operator on the (resp. full) one-sheeted hyperboloid is strictly greater than that for the (resp. double) cone.

In Section \ref{sec:upper-half} we prove Theorem \ref{thm:main-theorem} by collecting the necessary ingredients to use Proposition \ref{prop:key-prop-from-fvv}. Here the first inequality in \eqref{a_priori_bounds} is used to show that the $L^2$ mass of an extremizing sequence can not tend to infinity (i.e. there is no concentration at infinity).

From Section \ref{sec:full_hyperboloid} onward we focus on the full one-sheeted hyperboloid $\Hyp$. The existence of Lorentz invariance adds complexity to the proof of Theorem \ref{thm:main-theorem-2}, compared to the much simpler proof of Theorem \ref{thm:main-theorem}. We will use a concentration-compactness type argument that we discuss in Section  \ref{sec:concentration_compactness}. In short, given an $L^2$ normalized extremizing sequence $\{f_n\}_n$ for $\overline{T}$, three possibilities hold (possibly after passing to a further subsequence): \textit{compactness}, \textit{vanishing} or \textit{dichotomy}. In Section \ref{sec:no-dichotomy} we prove bilinear estimates at dyadic scales and show that they imply that \textit{dichotomy} can not occur. In Section \ref{sec:no-vanishing} we obtain a dyadic refinement of \eqref{sharp-double-sheeted} and used it to show that \textit{vanishing} can not occur.

To treat the \textit{compactness} case, it will be necessary to study so called "cap bounds" or refinements of the $L^2\to L^4$ estimate for the adjoint Fourier restriction operators $T$ and $\overline{T}$ and this we achieve in Section \ref{sec:lifting_the_sphere} by "lifting" to the hyperboloid the results for the sphere in $\R^3$, as proved in \cite{CS}, and recalled in Section \ref{sec:Tomas_Stein_sphere}. 

In Section \ref{sec:concluding_compactness} we show that if an extremizing sequence satisfies \textit{compactness} and does not concentrate at infinity then it is precompact in $L^2(\Hyp)$, modulo multiplication by characters and composition with Lorentz transformations.

Finally, in Section \ref{sec:conv-to-cone} we study some limiting relationships between the hyperboloid and the cone. The results of this section together with the second strict inequality in \eqref{a_priori_bounds} are used to show that, in the case of $\Hyp$, the $L^2$ mass of an extremizing sequence satisfying \textit{compactness} does not tend to infinity. When this is done, the proof of Theorem \ref{thm:main-theorem-2} is complete.

\subsection{Notation and some definitions.} The set of natural numbers is \break$\N=\{0,1,2,\dotsc\}$ and 
$\N^\ast=\{1,2,3,\dotsc\}$.

 For $s>0$, we let $\mathcal H^3_s:=\{(x,t)\colon x\in\R^3,\, 
t=\sqrt{\ab{x}^2-s^2}\}$, equipped with the measure
\begin{equation}\label{measure-mu-s}
\d\mu_s(x,t)=\one_{\{\ab{x}>s \}}\ddirac{t-\sqrt{\ab{x}^2-s^2}}\frac{\d x\d t}{\sqrt{\ab{x}^2-s^2}},
\end{equation}
and adjoint Fourier restriction operator $T_s$, $T_sf(x,t)=\widehat{f\mu_s}(-x,-t)$. There are corresponding definitions of $\Hyp_s,\,\bar{\mu}_s$ and $\overline{T}_s$ in analogy with the case $s=1$.
 
The cone in $\R^3$ is denoted $\Gamma^3:=\{(y,\ab{y}):y\in\R^3\}$ which comes with its Lorentz and scale invariant measure 
$\sigma_c$,
\[ \int_{\Gamma^3}f\d\sigma_c=\int_{\R^3}f(y,\ab{y})\frac{\d y}{\ab{y}}. \]
The adjoint Fourier restriction operator on the cone, $T_c$, is given by the expression
\begin{equation}
\label{fourier-extension-operator-cone}
T_cf(x,t)=\int_{\R^3} e^{ix\cdot y} e^{it\ab{y}} f(y)\frac{\d y}{\ab{y}}.
\end{equation}
which acts, a priori, on functions $f\in\Sh(\R^3)$.
The adjoint Fourier restriction operator on the double cone, 
$\overline{\Gamma}^3:=\Gamma^3\cup-\Gamma^3$, denoted by $\overline{T}_c$, 
is given by the expression
\begin{equation}
\label{eq:fourier-extension-operator-dcone}
\overline{T}_cf(x,t)=\int_{\R^3} e^{ix\cdot y} e^{it\ab{y}} f(y,\ab{y})\frac{\d y}{\ab{y}}
+\int_{\R^3} e^{ix\cdot y} e^{-it\ab{y}} f(y,-\ab{y})\frac{\d y}{\ab{y}},
\end{equation}
$f\in \Sh(\R^4)$. We let $\mathbf{C}_{4},\overline{\mathbf{C}}_{4}<\infty$ denote the best constants in the inequalities
\[ \norma{T_c f}_{L^4(\R^4)}\leq \mathbf{C}_{4}\norma{f}_{L^2(\Gamma^3)},\quad  \norma{\overline{T}_c f}_{L^4(\R^4)}\leq \overline{\mathbf{C}}_{4}\norma{f}_{L^2(\overline{\Gamma}^3)},\]
respectively. We sometimes use the alternative notation $\norma{T}=\mathbf{H}_4$, $\norma{\overline{T}}=\overline{\mathbf{H}}_4$, $\norma{T_c}=\mathbf{C}_{4}$ and $\norma{\overline{T}_c}=\overline{\mathbf{C}}_{4}$.

 The sphere of radius $r>0$ on $\R^3$ is $\Sph_r^2:=\{y\in\R^3:\ab{y}=r\}$. The sphere of radius $1$ 
is denoted simply $\Sph^2$. On $\Sph_r^2$ we consider the measure 
$\sigma_r$,
\begin{equation}\label{eq:sigma_r}
\int_{\Sph_r^2}f\d\sigma_r=\int_{\Sph^2}f(r\omega)r\d\sigma(\omega),
\end{equation}
where $\sigma$ is the surface measure on $\Sph^2$. With this choice, 
$\sigma_r(\Sph_r^2)=r\sigma(\Sph^2)$, for all $r>0$. For $r>0$ and a function $f:\R^3\to\mathbb C$ we set $f_r:\Sph^2\to\R$ by $f_r(\cdot)=f(r\,\cdot)$.

We let $\mathbf{S}$ denote the best constant 
in the convolution form of the Tomas--Stein inequality for the sphere $\Sph^2$,
\begin{equation*}
%\label{eq:inequality-sphere}
\norma{f\sigma\ast f\sigma}_{L^2(\R^3)}\leq \mathbf{S}^2\norma{f}_{L^2(\Sph^2)}^2.
\end{equation*}

We also use the following convention. For $f:\overline{\mathcal H}^3\to\R$ we 
write $f=f_++f_-$, where $f_+$ is supported on $\hyp$ and $f_-$ on the reflection of $\hyp$ with respect to the origin, $-\hyp$, and we further identify their domains with $\R^3$ via the orthogonal projection. For $A\subseteq 
\R^3$
we define
\[ \int_A f\d\mu:=\int_{\{(x,t)\in\hyp\colon x\in A\}}f\d\mu, \]
$f\in L^1(\hyp)$, while for $\Hyp$,
\[ \int_{A} f d\bar\mu:=\int_{\{(x,t)\in\Hyp\colon x\in A\}} f\d\bar\mu, \]
$f\in L^1(\Hyp)$, so that in both cases the integral over $A\subset\R^3$ equals to the integral over the "lift" of $A$ to $\hyp$ or $\Hyp$, as it corresponds.

An element $R\in SO(4)$ that preserves the $t$-axis, $R(0,0,0,1)=(0,0,0,1)$, is canonically identified with an element of $SO(3)$, and as such we will just write $R\in SO(3)$.

We let 
$\psi_s(r)=\sqrt{r^2-s^2}\one_{\{r\geq s\}}$, $\phi_s(t)=\psi_s^{-1}(t)=\sqrt{t^2+s^2}\one_{\{t\geq 0\}}$. The (closed) ball of radius $r>0$ centered at $y\in\R^3$ is $B(y,r)$. For a set $A$, $\one_A$ denotes the characteristic function of $A$ and $A^\complement$, the complement of $A$ with respect to a set containing $A$ that will be understood from context, usually $\R^3$, $\hyp$ or $\Hyp$. We sometimes slightly abuse notation and use $\ab{A}$ to denote the measure of a set $A$, where the measure used must be understood from context, for instance, if $A$ is an interval it refers to the Lebesgue measure, if $A\subseteq \Sph^2$, it refers to the surface measure, etc. The support of a function $f$ is denoted $\supp(f)$.

We will use the usual asymptotic notation $X\lesssim Y,\,Y\gtrsim X$ if there exists a constant $C$ (independent of $X,Y$) such that $\ab{X}\leq CY$; we use $X\asymp Y$ if $X\lesssim Y$ and $Y\lesssim X$; when such constants depend on parameters of the problem that we want to make explicit, such as $\alpha,\dotsc$ etc., we write $\lesssim_{\alpha,\dotsc},\gtrsim_{\alpha,\dotsc}$ and $\asymp_{\alpha,\dotsc}$. At times we will use the common asymptotic notation $o(\cdot)$ and $O(\cdot)$. Thus, $g_n=o(f_n)$ if $g_n/f_n\to0$ as $n\to\infty$, while $g_n=O(f_n)$ if $\ab{g_n}\leq C\ab{f_n}$ for all $n$. If there is more than one parameter, say $n\in\N$ and $s>0$, then $g_n(s)=o_n(f_s(s))$ means the limit of $g_n/f_n\to0$ is taken with respect to $n$ and is uniform in $s$, that is $\sup_s\ab{g_n(s)}/\ab{f_n(s)}\to0$ as $n\to\infty$.

\section{Lorentz invariance, symmetrization and caps}\label{sec:lorentz_invariance}
\subsection{Lorentz invariance}
Recall that the Lorentz group on $\R^4$, denoted $\mathcal L$, is defined as the group of 
invertible linear 
transformations in $\R^4$ that preserve the bilinear form
\[ B(x,y)=x_{4}y_{4}-x_{3}y_{3}-x_2y_2-x_1y_1, \]
for $x=(x_1,x_2,x_3,x_4)\in\R^4$ and $y=(y_1,y_2,y_3,y_4)\in\R^4$. If $L\in\mathcal L$ then 
$\ab{\det L}=1$. Given that we can write $\Hyp=\{(x,t)\in\R^{3+1}\colon B((x,t),(x,t))=-1\}$ it is 
direct that $\mathcal L$ preserves the hyperboloid: $L(\Hyp)=\Hyp$, for every 
$L\in\mathcal L$. 
Moreover, $\mathcal L$ preserves the measure $\bar{\mu}$, in the sense that for every $f\in 
L^1(\Hyp)$ and $L\in\mathcal{L}$
\begin{equation}\label{eq:lorentz-invariance}
\int_{\Hyp} f(x,t)\d\bar{\mu}(x,t)=\int_{\Hyp} f(L(x,t))\d\bar{\mu}(x,t).
\end{equation}
To see this note that a simple calculation shows that we can write 
\[ \bar{\mu}(x,t)=\ddirac{t^2-\ab{x}^2+1}\d x\d t \]
so that 
\[ \int_{\R^4} f(x,t) \d\bar\mu(x,t)= \int_{\R^4} f(x,t)\ddirac{t^2-\ab{x}^2+1}\d t\d x.\]
Then, if $L$ is a Lorentz transformation  
and $f\in L^1(\overline{\mathcal{H}}^3)$, \eqref{eq:lorentz-invariance} can be seen to hold by the change of 
variable formula. 

For $t\in (-1,1)$ the Lorentz boost $L^t\in\mathcal L$ is the linear map
\begin{equation}\label{eq:lorentz_boost}
L^t(\xi_1,\xi_2,\xi_3,\tau)=\biggl(\frac{\xi_1+t\tau}{\sqrt{1-t^2}},\xi_2,\xi_3,\frac{t\xi_1+\tau}{\sqrt{1-t^2}}\biggr),
\end{equation}
while $L_t$ denotes the rescaling $L_t:=(1-t^2)^{1/2}L^t$, so that $(L_t)^{-1}=(1-t^2)^{-1/2}L^{-t}$.

\subsection{Convolution form} With the Fourier transform in $\R^d$ normalized as\break $\widehat{F}(x)=\int_{\R^d}e^{-ix\cdot y}F(y)\d y$ we have the identities
\[ \widehat{F\ast G}=\widehat{F}\,\widehat{G},\quad \norma{\widehat{F}}_{L^2(\R^d)}=(2\pi)^{d/2}\norma{F}_{L^2(\R^d)}, \]
so that using $Tf(x,t)=\widehat{f\mu}(-x,-t)$ and $\overline{T}g(x,t)=\widehat{g\bar{\mu}}(-x,-t)$ we find the equalities
\begin{equation}\label{eq:convolution-form}
\norma{Tf}_{L^4(\R^4)}=2\pi\norma{f\mu\ast f\mu}_{L^2(\R^4)}^{1/2},\quad \norma{\overline{T}g}_{L^4(\R^4)}=2\pi\norma{g\bar\mu\ast g\bar\mu}_{L^2(\R^4)}^{1/2}.
\end{equation}
Using this \textit{convolution form} of the $L^4$ norm and the triangle inequality we see that $\norma{Tf}_{L^4(\R^4)}\leq \norma{T\ab {f}}_{L^4(\R^4)}$ and $\norma{\overline{T}g}_{L^4(\R^4)}\leq \norma{\overline{T}\ab {g}}_{L^4(\R^4)}$, so that if $f$ is an extremizer for \eqref{adjoint-restriction-one-sheeted} (resp. $g$ for \eqref{sharp-double-sheeted}), then so is $\ab{f}$ (resp. $\ab{g}$), showing that if extremizers exist then there are nonnegative real valued extremizers.

\subsection{Symmetrization} Let $f\in L^2(\Hyp)$ be a complex valued 
function. Denote the reflection of $f$ by $\widetilde 
f(x,t)=f(-x,-t)$ and the nonnegative $L^2$-symmetrization of $f$ by
\[f_\sharp(x,t)=\biggl(\frac{\ab{f(x,t)}^2+\ab{f(-x,-t)}^2}{2}\biggr)^{1/2}.\]
Regarding the relationship between $f$ and $f_\sharp$ we have the following lemma.
\begin{lemma}\label{lem:symmetrization-lemma}
	Let $f\in L^2(\Hyp)$ be a complex valued function. Then
\begin{equation}\label{eq:ineq_even}
\norma{f\bar{\mu}\ast f\bar{\mu}}_{L^2(\R^4)}\leq \norma{f_\sharp\bar{\mu}\ast 
		f_\sharp\bar{\mu}}_{L^2(\R^4)}.
\end{equation}
\end{lemma}
\begin{proof}
	As in \cite{Fo2}*{Proof of Prop. 3.2} we write
	\begin{align*}
	f\bar{\mu}\ast \widetilde f\bar{\mu}(\xi,\tau)
	&=\int 
	f(y,s)f(-x,-t)\ddirac{(\xi,\tau)-(y,s)-(x,t)}\d\bar\mu(y,s)\d\bar\mu(x,t)\\
	&=\frac{1}{2}\int 
	(f(y,s)f(-x,-t)+f(-y,-s)f(x,t))\\
	&\qquad\qquad\cdot \ddirac{(\xi,\tau)-(y,s)-(x,t)}\d\bar\mu(y,s)\d\bar\mu(x,t),
	\end{align*}
	and apply the Cauchy-Schwarz inequality
	\[ \ab{f(y,s)f(-x,-t)+f(-y,-s)f(x,t)}\leq 2f_\sharp(y,s)f_\sharp(x,t) \]
	to obtain that for all $(\xi,\tau)\in\R^4$
	\[ \ab{f\bar{\mu}\ast \widetilde f\bar{\mu}(\xi,\tau)}\leq f_\sharp\bar{\mu}\ast f_\sharp\bar{\mu}(\xi,\tau). \]
	Then
	\[ \norma{f\bar{\mu}\ast f\bar{\mu}}_{L^4(\R^4)}=\norma{f\bar{\mu}\ast \widetilde 
		f\bar{\mu}}_{L^4(\R^4)}\leq \norma{f_\sharp\bar{\mu}\ast f_\sharp\bar{\mu}}_{L^4(\R^4)}. \]
\end{proof}

\noindent Since we also have
\begin{equation*}
%\label{eq:equal-L2-norm}
\norma{f}_{L^2(\bar{\mu})}=\norma{f_\sharp}_{L^2(\bar{\mu})},
\end{equation*}
it follows that there exist real valued extremizers for $\overline{T}$ which are nonnegative even functions on 
$\overline{\mathcal{H}}^3$. Moreover, any nonnegative real valued extremizer is \textit{necessarily} even. This can be explained by studying the cases of equality in \eqref{eq:ineq_even} by following the proof of the inequality (see \cite{CaOeS} for a detailed discussion in the case of the sphere) or, alternatively, by using the same method as in the proof of \cite{OeSQ18}*{Lemma 6.1} where a different approach to symmetrization is used and the cases of equality were studied. Therefore, we have the following result.
\begin{prop}\label{prop:symmetric-complex}
	Let $f\in L^2(\Hyp)$ be a nonnegative real valued extremizer for 
	\eqref{sharp-double-sheeted}, then 
	$f(x,t)=f(-x,-t)$ for $\bar\mu$-a.e. $(x,t)\in \Hyp$.
\end{prop}

There are some interesting problems that we do not address in this article:
\begin{itemize}
\item[(i)] the nonnegativity of \textit{all} real valued extremizers,
\item[(ii)] the relationship between complex and real valued extremizers,
\item[(iii)] the smoothness of extremizers.
\end{itemize}
We provide the following comments in the context of the $L^2(\Sph^{d-1})\to L^p(\R^d)$ adjoint Fourier restriction inequality on the sphere. Christ and Shao \cite{CS2} showed that for the case of the the sphere $\Sph^2$ in $\R^3$ and $p=4$ each complex valued extremizer is of the form $ce^{ix\cdot \xi}F(x)$, for some $\xi\in\R^3$, some $c\in\mathbb{C}$ and some nonnegative extremizer $F$, and that extremizers are of class $C^{\infty}$; this results were later expanded to all dimensions $d\geq 2$ and even integers $p$ in \cite{OQ19a}*{Lemma 2.2 and Theorem 1.2} and \cites{OQ19b}. Note that the answer obtained for (ii) resolves (i). By using the outline in \cites{CS2,OQ19a,OQ19b}, the Euler--Lagrange equation, which can be obtained as in \cite{CQ14}, and the results in \cite{Cha} we expect similar relationships for the case of $\hyp$ and $\Hyp$, but have not investigated the extent to which the arguments would need to be changed. 

A related question is that of the rate of decay at infinity of an extremizer for which the argument in \cite{HS12} gives a possible route; see also \cite{OeSQ18}.\\

We remark that Theorems \ref{thm:main-theorem} and \ref{thm:main-theorem-2} are stated for general (possibly complex valued) extremizing sequences, that is, we do not assume nonnegativity and/or symmetry. 

\subsection{Caps}
A (closed) spherical cap $\sphcp\subseteq \Sph^2$ is a set of the form 
$\sphcp=\{x\in\Sph^2\colon \ab{x-x_0}\leq t\}$ for some $x_0\in\Sph^2$ and $t>0$. If we want to be explicit about the dependence on $x_0$ and $t$ we write $\sphcp(x_0,t)$.

A cap $\cp$ of $\mathcal H_s^3$ is a set of the form
\begin{equation}\label{eq:definition-cap}
\cp=\{(r\omega,\sqrt{r^2-s^2}):r\in [a,b],\,\omega\in \sphcp \},
\end{equation}
where $s\leq a<b\leq \infty$ and $\sphcp\subseteq \Sph^2$ is a spherical cap. 
When $a=s\,2^k$ and 
$b=s\,2^{k+1}$ 
for some $k\in\Z$ we say that $\cp$ is a dyadic cap. We identify a cap $\cp$ as before with its 
orthogonal projection to $\R^3\times \{0\}$, 
and moreover we use spherical coordinates and write the cap in \eqref{eq:definition-cap} as 
$\cp=[a,b]\times \sphcp$, where the hyperboloid it belongs to will be understood from context. A 
cap 
$\cp$ of $\overline{\mathcal H}_s^3$ is such that either $\cp\subseteq\hyp_s$ or its reflection with respect to the origin $(-\cp)\subseteq \hyp_s$ is a cap on $\mathcal H_s^3$.

The $\mu_s$-measure of a cap is easily calculated
\begin{equation}\label{eq:measure_cap}
\mu_s(\cp)=\int_{\cp}\d\mu_s=\sigma(\sphcp)\int_{a}^b \frac{r^2}{\sqrt{r^2-s^2}} \d r
=\frac{\sigma(\sphcp)}{2}\Bigl({s}^{2}\ln\bigl(r+\sqrt{{r}^{2}-{s}^{2}}\bigr) 
+r\sqrt{{r}^{2}-{s}^{2}}\Bigr)\eval_a^b.
\end{equation}

\noindent For a cap $\cp=[a,b]\times \sphcp$ in 
$\mathcal H^3_s$
and $t>0$ we define the rescaled cap $t\cp=[ta,tb]\times \sphcp$ as the cap in $\hyp_{ts}$ given by
\[ t\cp=\{(r\omega,\sqrt{r^2-(ts)^2})\colon r\in[ta,tb],\,\omega\in\sphcp\}, \] 
and note that
\begin{equation}\label{eq:rescaling_Measure}
\mu_{ts}(t\cp)=t^2\mu_s(\cp).
\end{equation}
We also note that for such a cap $\cp\subset\hyp_s$ there exist $R\in SO(3)$ and $\eps\in[0,\pi]$ such that
\begin{multline}\label{eq:rotatedHypCap}
R^{-1}(\cp)=\{(r\omega,\sqrt{r^2-s^2})\colon a\leq r\leq b,\\
\omega=(\cos\vphi,\cos\te\sin\vphi,\sin\te\sin\vphi),\, \te\in[0,2\pi],\,\vphi\in[0,\eps] \}.
\end{multline}
Keeping this notation in mind for the rest of the section we study the use of Lorentz transformations and scaling in the regimes when $\mu(\cp)$ is large and small. 
The following two lemmas will be useful in later sections. 
\begin{lemma}\label{scaling-cap-hyp}
	Let $s\leq \frac{1}{2}$, $\sphcp\subseteq\Sph^2$ be a spherical cap and $\cp= [1,2]\times \sphcp$ be a cap in the hyperboloid 
	$\mathcal H^3_s$. Let $R$ and $\eps$ be as in \eqref{eq:rotatedHypCap} and suppose that $\eps\in[0,\frac{\pi}{2}]$ and $s^{-2}\sin^2\eps\geq 8$. Then there 
	exist $0\leq 
	t<1$ such that $L_t^{-1}R^{-1}(\cp)\subset \mathcal H^3_{\frac{s}{\sqrt{1-t^2}}}$ satisfies
	\[ \mu_{\frac{s}{\sqrt{1-t^2}}}(L_t^{-1}R^{-1}(\cp))\geq \tfrac{\pi}{2}\text{ and } L_t^{-1}R^{-1}(\cp)\subseteq 
	[\tfrac{7}{16},\tfrac{33}{16}]\times\Sph^2. \]
	Moreover, if $\eps\in[0,\frac{\pi}{3}]$, we can take $t=\cos\eps$, while if $\eps\in(\frac{\pi}{3},\frac{\pi}{2}]$ we can take $t=0$.
\end{lemma}
\begin{proof}
	With $R\in SO(3)$ and $\eps\in[0,\frac{\pi}{2}]$ satisfying \eqref{eq:rotatedHypCap},
	note that $L_t^{-1}R^{-1}(\cp)=(1-t^2)^{-1/2}L^{-t}R^{-1}(\cp)\subseteq \mathcal H^3_{s(1-t^2)^{-1/2}}$, for every $t\in(-1,1)$. 
	According to \eqref{eq:measure_cap}, the $\mu_s$-measure of $\cp$ satisfies
	\begin{align*}
	 \mu_s(\cp)&=2\pi(1-\cos\eps)\Bigl(\frac{{s}^{2}}{2}\ln\left( \sqrt{{r}^{2}-{s}^{2}}+r\right) 
	 +\frac{r}{2}\sqrt{{r}^{2}-{s}^{2}}\Bigr)\eval_{1}^2\\
	 &\geq \pi(1-\cos\eps)(\sqrt{4-s^2}-\sqrt{1-s^2})\geq \pi(1-\cos\eps),
	\end{align*}
	so that in what follows we can assume $\cos\eps\geq 
	1/2$, otherwise we are done by taking $t=0$. From \eqref{eq:rescaling_Measure}, for $t\in(0,1)$,
	\[ \mu_{\frac{s}{\sqrt{1-t^2}}}(L_t^{-1}R^{-1}(\cp))=(1-t^2)^{-1}\mu_s(\cp), \] 
	so that choosing $t=\cos\eps$ gives 
	$\mu_{s(1-t^2)^{-1/2}}(L_t^{-1}R^{-1}(\cp))\geq \frac{\pi}{1+\cos\eps}\geq \frac{\pi}{2}$.
		On the other hand, we have
	\begin{multline*}
	L_t^{-1}R^{-1}(\cp)=\Bigl\{(1-t^2)^{-1/2}\Bigl(\frac{r\cos\vphi-t\sqrt{r^2-s^2}}{(1-t^2)^{1/2}},\\
	r\cos\te\sin\vphi
	,r\sin\te\sin\vphi,
	\frac{\sqrt{r^2-s^2}-tr\cos\vphi}{(1-t^2)^{1/2}}\Bigr)\colon r\in [1,2],\te\in[0,2\pi],\vphi\in[0,\eps]\Bigr\},
	\end{multline*}
	and since $\cos\vphi\geq 
	\cos\eps$ and $1\leq r\leq 2$ we obtain that the fourth coordinate of any point in $L_t^{-1}R^{-1}(\cp)$ is bounded as follows
	\begin{align*}
	\frac{\sqrt{r^2-s^2}-tr\cos\vphi}{1-t^2}=\frac{r(\sqrt{1-(s/r)^2}-t\cos\vphi)}{1-t^2}\leq 2\frac{1-\cos^{2}\eps}{1-\cos^2\eps}= 2
	\end{align*}
	and
	\begin{align*}
	\frac{r(\sqrt{1-(s/r)^2}-t\cos\vphi)}{1-t^2}&= \frac{r}{\sin^2\eps}(\sqrt{1-(s/r)^2}-\cos\eps\cos\vphi)\\
	&\geq \frac{r}{\sin^2\eps}(\sqrt{1-(s/r)^2}-\cos\eps)\\
	&=\frac{r}{\sqrt{1-(s/r)^2}+\cos\eps}\Bigl(1-\frac{1}{r^2s^{-2}\sin^2\eps}\Bigr)\\
	&\geq \frac{r}{2}\Bigl(1-\frac{1}{8r^2}\Bigr)\geq \frac{7}{16}.
	\end{align*}
	Therefore 
	\[ L_t^{-1}R^{-1}(\cp)\subseteq [\phi_{\frac{s}{\sqrt{1-t^2}}}(\tfrac{7}{16}),\phi_{\frac{s}{\sqrt{1-t^2}}}(2)]\times\Sph^2. \]
	Now, from the definition of $t$ and the assumption that $s^{-2}\sin^2\eps\geq 8$ we obtain
	\begin{equation*}
	\frac{s}{\sqrt{1-t^2}}=\frac{s}{\sin\eps}\leq\frac{\sqrt{2}}{4},
	\end{equation*}
	so that the following inequalities hold
	\[r\leq \phi_{\frac{s}{\sqrt{1-t^2}}}(r)=\sqrt{r^2+s^2(1-t^2)^{-1}}\leq \sqrt{r^2+1/8},\]
	from where $\phi_{\frac{s}{\sqrt{1-t^2}}}(\tfrac{7}{16})\geq \frac{7}{16}$ and $\phi_{\frac{s}{\sqrt{1-t^2}}}(2)\leq \frac{33}{16}$
	and then we find
	$L_t^{-1}R^{-1}(\cp)\subseteq [\tfrac{7}{16},\tfrac{33}{16}]\times\Sph^2$.
\end{proof}
As noted in \cite{COS}*{Lemma 4} for the two-sheeted hyperboloid, a Lorentz transformation can map  caps of uniformly bounded measure into a bounded ball. This we record in the next lemma.
\begin{lemma}\label{lem:measure_bounded_cap}
	Let $s>0$, $k\in\N$ and $\cp_k\subset\Hyp_s$ be a dyadic cap of the form $\cp_k=[s2^k,s2^{k+1}]\times\sphcp_k$, for some spherical cap $\sphcp_k\subseteq\Sph^2$. Let $R$ and $\eps$ be associated to $\cp_k$ as in \eqref{eq:rotatedHypCap}, then
	\begin{enumerate}
		\item[(i)] The $\bar\mu_s$-measure of $\cp_k$ satisfies
		\begin{equation}\label{eq:measure_asymp}
		\begin{split}
		\bar\mu_s(\cp_k)&=3\pi s^2 (1+o_k(1))2^{2k}(1-\cos\eps)\\
		&=\frac{3\pi s^2}{1+\cos\eps}(1+o_k(1))2^{2k}\sin^2\eps.
		\end{split}
		\end{equation}
		\item[(ii)] Suppose $\eps\in[0,\frac{\pi}{2}]$. Then, there exists $t\in [0,1)$ such that the orthogonal projection of $L^{-t}R^{-1}(\cp_k)\subset\Hyp_s $ to $\R^3$ is contained in a ball of $\R^3$ of radius comparable to $s+s^{-1}\bar\mu_s(\cp_k)+\bar\mu_s(\cp_k)^{1/2}$.
	\end{enumerate}
\end{lemma}
\begin{proof}
	Without loss of generality, we may assume that $\cp_k$ is contained in the upper half $\hyp_s$. For part (i), \eqref{eq:measure_cap} implies that the $\bar\mu_s$-measure of $\cp_k$ is given by the expression
\[ \bar\mu_s(\cp_k)=\pi s^2(1-\cos\eps)\Bigl(\ln\Bigl(\frac{2^{k+1}+\sqrt{2^{2(k+1)}-1}}{2^k+\sqrt{2^{2k}-1}}\Bigr) 
		+2^{k+1}\sqrt{2^{2(k+1)}-1}-2^k\sqrt{2^{2k}-1}\Bigr). \]
	The expression involving the logarithm converges to $\ln(2)$ as $k\to\infty$, while
	\begin{align*}
	2^{k+1}\sqrt{2^{2(k+1)}-1}-2^k\sqrt{2^{2k}-1}&=\frac{2^{2(k+1)}(2^{2(k+1)}-1)-2^{2k}(2^{2k}-1)}{2^{k+1}\sqrt{2^{2(k+1)}-1}+2^k\sqrt{2^{2k}-1}}\\
	&=2^{2k}\frac{15-3\cdot 2^{-2k}}{4\sqrt{1-2^{-2(k+1)}}+\sqrt{1-2^{-2k}}}\\
	&=3\cdot 2^{2k}(1+o_k(1)).
	\end{align*}
	
	For part (ii), let $R\in SO(3)$ and $\eps\in[0,\frac{\pi}{2}]$ be such that \eqref{eq:rotatedHypCap} holds. 
	The image of $R^{-1}(\cp_k)$ under the Lorentz boost $L^{-t}$ is
	\begin{multline}\label{eq:LorentzRotatedCap}
	L^{-t}R^{-1}(\cp_k)=\Bigl\{\Bigl(\frac{r\cos\vphi-t\sqrt{r^2-s^2}}{(1-t^2)^{1/2}},r\cos\te\sin\vphi
	,r\sin\te\sin\vphi,\frac{\sqrt{r^2-s^2}-tr\cos\vphi}{(1-t^2)^{1/2}}\Bigr)\colon\\
	r\in[s2^k,s2^{k+1}],\theta\in[0,2\pi],\vphi\in[0,\eps]\Bigr\}.
	\end{multline}
	Let $t=\sqrt{1-2^{-2(k+1)}}$, so that the first coordinate of a point in the set on the right hand side of \eqref{eq:LorentzRotatedCap} is bounded as follows
	\begin{align*} \abs{\frac{r\cos\vphi-t\sqrt{r^2-s^2}}{(1-t^2)^{1/2}}}&=2^{k+1}r\ab{\cos\vphi-\sqrt{1-2^{-2(k+1)}}\sqrt{1-(s/r)^2}}\\
	&\leq 2^{2(k+1)}s(1-\cos\vphi)+2^{2(k+1)}s(1-(1-2^{-2(k+1)}))\\
	&= 2^{2(k+1)}s(1-\cos\eps)+s\\
	&\lesssim \frac{\bar\mu_s(\cp_k)}{s}+s,
	\end{align*}
	where in the last line we used \eqref{eq:measure_asymp}. The second and third coordinates are bounded as follows
	\[ \ab{r\cos\te\sin\vphi},\,\ab{r\sin\te\sin\vphi}\leq 2^{k+1}s\sin\eps\lesssim \sqrt{\bar\mu(\cp_k)}. \]
	Then $L^{-t}R^{-1}(\cp_k)$ is contained in the set
	\[ \Bigl\{(x,t)\in\Hyp_s:\ab{x}\leq C\Bigl(\sqrt{\bar\mu_s(\cp_k)}+\frac{\bar\mu_s(\cp_k)}{s}+s\Bigr) \Bigr\}, \]
	for some constant $C$ independent of $k$ and $s$.
\end{proof}

\section{Calculation of a double convolution}\label{sec:calculation_convolution}

In previous studies of quadric surfaces and curves and their perturbations it has become clear 
the importance of the 
double or triple, and more generally the $n$-th fold, convolution of the underlying measure. Its properties may determine existence 
or nonexistence of extremizers and in some cases it can be used to find their explicit 
form and/or the value of the best constant in the corresponding adjoint Fourier 
restriction inequality. In the case of the one-sheeted hyperboloid and its upper half, the double 
convolution will be used to prove that extremizing sequences do not concentrate at infinity.

Let $\mu_s\ast\mu_s$ denote the double convolution of $\mu_s$ with itself, defined by duality
\[ \langle \mu_s\ast\mu_s,f\rangle=\int_{(\R^4)^2} f(x+x',t+t') \d\mu_s(x,t)\d\mu_s(x',t'), \]
for all $f\in \Sh(\R^4)$.
It is not difficult to see that $\mu_s\ast\mu_s$ is absolutely continuous with respect to the 
Lebesgue measure in $\R^4$, indeed this follows from \eqref{adjoint-restriction-one-sheeted} since 
$e^{-\tau}(\mu_s\ast\mu_s)\in L^2(\R^4)$, it being the (inverse) Fourier transform of the 
$L^2(\R^4)$ 
function $(\widehat{e^{-\tau}\mu_s})^2$ (see also \cite{OSQ}*{Proposition 2.1}). 
In what follows we identify $\mu_s\ast\mu_s$ with its Radon--Nicodym 
derivative with respect to the Lebesgue measure in $\R^4$.
\begin{prop}
\label{prop:formula-double-convolution}
 Let $\mu_s$ be the measure on $\mathcal H^3_s$ defined in \eqref{measure-mu-s}. Then
 \begin{itemize}
 	\item[(i)] The support of the convolution measure $\mu_s\ast\mu_s$ is
 	\[ \supp(\mu_s\ast\mu_s)=\{(\xi,\tau)\in\R^4\colon \tau\geq 0,\,\ab{\xi}\leq 
 	\sqrt{\tau^2+s^2}+s\}. \]
 	\item[(ii)] For every $(\xi,\tau)\in \R^4$ with $\tau\geq 0$ we have the formula
 	\begin{equation}\label{eq:alternative-conv-formula}
 	\begin{split}
 	\mu_s*\mu_s(\xi,\tau)&=\frac{2\pi}{\ab{\xi}}\biggl(
 	\ab{\xi}\Bigl(1+\frac{4s^2}{\tau^2-\ab{\xi}^2}\Bigr)^{\frac{1}{2}}\one_{\{\ab{\xi}<\sqrt{\tau^2+s^2}-s\}}+\tau
 	\one_{\{\sqrt{\tau^2+s^2}-s\leq \ab{\xi}\leq \sqrt{\tau^2+(2s)^2}\}}\\
 	&\quad\quad\quad+\Bigl(\tau-\ab{\xi}\Bigl(1+\frac{4s^2}{\tau^2-\ab{\xi}^2}\Bigr)^{\frac{1}{2}}\Bigr)
 	\one_{\{\sqrt{\tau^2+(2s)^2}< \ab{\xi}\leq \sqrt{\tau^2+s^2}+s\}} \biggr).
 	\end{split}
 	\end{equation}
 \end{itemize}
\end{prop}
When $\xi=0$ and $\tau> 0$ we understand that in \eqref{eq:alternative-conv-formula}
$\mu_s\ast\mu_s(0,\tau)=2\pi(1+\frac{4s^2}{\tau^2})^{1/2}$.

\noindent We postpone the proof of Proposition \ref{prop:formula-double-convolution} and study the behavior of 
$\mu_s*\mu_s(\xi,\tau)$ for large $\tau$.
\begin{lemma}\label{lem:behavior-infinity}
	For all $\tau>0$,
	\[ 2\pi\Bigl(1+\frac{4s^2}{\tau^2}\Bigr)^{1/2}\leq \sup_{\xi\in\R^3}\mu_s\ast\mu_s(\xi,\tau)\leq  
			2\pi\Bigl(1+\frac{2s}{\tau}\Bigr). \]
	In particular
	\[ \lim_{\tau\to\infty} \sup_{\xi\in\R^3}\mu_s\ast\mu_s(\xi,\tau)=2\pi.\]
\end{lemma}
\begin{proof}
	We start by noting that 
	\[ \mu_s\ast \mu_s(s\xi,s\tau)=\mu\ast\mu(\xi,\tau),\]
	hence it is enough to consider the case $s=1$. We analyze the different cases in formula 
	\eqref{eq:alternative-conv-formula}.
	
	\textbf{Case 1.} $\ab{\xi}<\sqrt{\tau^2+1}-1$. Then
	\[ \Bigl(1+\frac{4}{\tau^2}\Bigr)^{1/2}\leq \Bigl(1+\frac{4}{\tau^2-\ab{\xi}^2}\Bigr)^{1/2}\leq 
	\Bigl(\frac{\sqrt{\tau^2+1}+1}{\sqrt{\tau^2+1}-1}\Bigr)^{1/2}=\frac{\tau}{\sqrt{\tau^2+1}-1}.\]
	
	\textbf{Case 2.} $\sqrt{\tau^2+1}-1\leq \ab{\xi}\leq \sqrt{\tau^2+4}$. Then
	\[ \frac{\tau}{\sqrt{\tau^2+4}}\leq 
	\frac{\tau}{\ab{\xi}}\leq\frac{\tau}{\sqrt{\tau^2+1}-1}. \]
	
	\textbf{Case 3.} $\sqrt{\tau^2+4}< \ab{\xi}\leq \sqrt{\tau^2+1}+1$. Then $\ab{\xi}^2-\tau^2>4$ and
	\[ \ab{\xi}\mapsto\frac{\tau}{\ab{\xi}}-\Bigl(1+\frac{4}{\tau^2-\ab{\xi}^2}\Bigr)^{1/2} \]
	is a decreasing function of $\ab{\xi}$. Then
	\[ \frac{\tau}{\ab{\xi}}-\Bigl(1+\frac{4}{\tau^2-\ab{\xi}^2}\Bigr)^{1/2}\leq 
	\frac{\tau}{\sqrt{\tau^2+4}},\]
	and 
	\[ \frac{\tau}{\ab{\xi}}-\Bigl(1+\frac{4}{\tau^2-\ab{\xi}^2}\Bigr)^{1/2}\geq
	\frac{\tau}{\sqrt{\tau^2+1}+1}- 
	\Bigl(1-\frac{2}{\sqrt{\tau^2+1}+1}\Bigr)^{1/2}=0. \]
	As a conclusion, for all $\tau>0$ and $x\in\R^3$
	\[ \mu\ast\mu(\xi,\tau)\leq \frac{2\pi 
	\tau}{\sqrt{\tau^2+1}-1}=2\pi\Bigl(\Bigl(1+\frac{1}{\tau^2}\Bigr)^{1/2}+\frac{1}{\tau}\Bigr)
	\leq 2\pi\Bigl(1+\frac{2}{\tau}\Bigr),\]
	and for $\tau>0$
	\[ \sup_{\xi\in\R^3}\mu\ast\mu(\xi,\tau)\geq 2\pi\Bigl(1+\frac{4}{\tau^2}\Bigr)^{1/2}.\]
\end{proof}

\noindent We now turn to the proof of Proposition \ref{prop:formula-double-convolution}.

\begin{proof}[Proof of Proposition \ref{prop:formula-double-convolution}]
Part (i) is a simple calculation and is left to the reader. For part 
(ii) we start by discussing a change of coordinates that was used in the proof of \cite{Fo}*{Lemma 5.1} in the
arxiv's 
second version of 
\cite{Fo}; see also Appendix 3 on the arxiv's version of 
\cite{RQ2} 
where an outline of
the computation of the double convolution of the Lorentz invariant measure on the two-sheeted hyperboloid was given using the same technique. 

For each fixed $\xi\neq 0$ we 
consider a spherical coordinate system with axis $\xi$, that is, each 
$\eta\in\R^3$ is described as $\eta=(\rho\cos\te\sin\vphi,\rho\sin\te\sin\vphi,\rho\cos\vphi)$, where 
$\rho=\ab{\eta}\geq 0$,  
$\vphi\in[0,\pi]$ is the angle between $\xi$ and $\eta$ and $\te\in[0,2\pi]$ is a polar coordinate angle on the plane orthogonal to $\xi$. Then 
$\d\eta=\rho^2\sin\vphi\d\rho\d\te\d\vphi$. 

 Define the new variable
$\varsigma=\ab{\xi-\eta}$, which corresponds to the size of the side opposite to the origin, $0$, in the 
triangle 
whose vertices are located at 
$0,\,\xi$ and $\eta$. Then
\[ \varsigma^2=\ab{\xi}^2+\rho^2-2\ab{\xi}\rho\cos\vphi. \]
Changing variables from $\vphi$ to $\varsigma$, gives 
$\varsigma\d\varsigma=\ab{\xi}\rho\sin\vphi\d\vphi$,
so that in the variables $(\rho,\varsigma,\te)$ we have
$\d\eta=\frac{\rho\varsigma}{\ab{\xi}}\d\rho\d\varsigma\d\te$. The range of $\varsigma$ can be 
seen by using that $\varsigma$, $\ab{\xi}$ and $\rho$ are the sizes of the sides of a triangle, so 
$\ab{\rho-\varsigma}\leq \ab{\xi}\leq \rho+\varsigma$, which translates into $\ab{\ab{\xi}-\rho}\leq\varsigma\leq \ab{\xi}+\rho$.

Using delta calculus (see for instance the survey article \cite{FoD}) and the previous change of 
variables we have
\begin{align*}
 \mu_s*\mu_s(\xi,\tau)&=\int_{\substack{\eta\in\R^3\\ \ab{\eta}\geq s\\ \ab{\xi-\eta}\geq 
 s}}\frac{\ddirac{\tau-\sqrt{\ab{\xi-\eta}^2-s^2}-
 \sqrt{\ab{\eta}^2-s^2}}}{\sqrt{\ab{\xi-\eta}^2-s^2}\sqrt{\ab{\eta}^2-s^2}} \d\eta\\
 &=\frac{2\pi}{\ab{\xi}}\int_{\substack{\ab{\rho-\varsigma} \leq \ab{\xi} \\
        \rho + \varsigma \geq \ab{\xi}\\
        \rho\geq s,\,\varsigma\geq s}} \frac{\ddirac{\tau-\sqrt{\varsigma^2-s^2}
        -\sqrt{\rho^2-s^2}}}{\sqrt{\varsigma^2-s^2}\sqrt{\rho^2-s^2}}\rho\varsigma\d\rho 
        \d\varsigma\\
 &=\frac{2\pi}{\ab{\xi}} \int_{R_s} \ddirac{\tau-u-v}\d u\d v,
\end{align*}
where we changed variables $u=\sqrt{\rho^2-s^2},\, v=\sqrt{\varsigma^2-s^2}$ and $R_s=R_s(\xi)$ is the image of the 
region 
$\{(\rho,\varsigma)\colon\ab{\rho-\varsigma} \leq \ab{\xi},\rho
+ \varsigma \geq \ab{\xi},\rho\geq s,\varsigma\geq s\}$ under the transformation $(\rho,\varsigma)\mapsto(u,v)$. Using the change of
variables $a=u-v,\, b=u+v$, so that $2\d u\d v=\d a\d b$, we obtain
\begin{equation}\label{eq:preliminar-formula-conv}
\mu_s*\mu_s(\xi,\tau)=\frac{\pi}{\ab{\xi}} \int_{\widetilde R_{s}} \ddirac{\tau-b}\d a\d 
b=\frac{\pi}{\ab{\xi}}\ab{\widetilde R_s\cap
	\tilde \ell_\tau}=\frac{\pi}{\ab{\xi}}\sqrt{2}\ab{R_s\cap
	\ell_\tau}.
\end{equation}
where $\widetilde R_s=\widetilde R_s(\xi)$ is the image of $R_s(\xi)$ under the map 
$(u,v)\mapsto(a,b)$, $\tilde 
\ell_\tau$ is 
the 
horizontal line
$\{(a,b)\in\R^2:b=\tau\}$, $\ell_\tau$ is the line $\{(u,v)\in\R^2:u+v=\tau\}$ and $\ab{R_s\cap \ell_\tau}$ denotes the measure of $R_s\cap \ell_\tau$ as a subset of $\ell_\tau$ with the induced Lebesgue measure.
In order to calculate $\ab{R_s\cap\ell_\tau}$ we divide the analysis into two cases.

\textbf{Case 1:} $\ab{\xi}\leq 2s$.
The boundary of the region 
$$\{(\rho,\varsigma)\colon\ab{\rho-\varsigma} \leq \ab{\xi},\rho
+ \varsigma \geq \ab{\xi},\rho\geq s,\varsigma\geq s\}$$
consists of two (bounded) line segments 
and two half lines. Its image in the $(u,v)$-plane, $R_s$, is bounded by two line segments and two curves and is 
symmetric with respect to the diagonal $u=v$. The line 
segments have equations 
\begin{align*}
\{(u,v)\colon u=0,\,0\leq v\leq\sqrt{(\ab{\xi}+s)^2-s^2}\},
\{(u,v)\colon 0\leq u\leq\sqrt{(\ab{\xi}+s)^2-s^2},\,v=0\}
\end{align*}
and the curves have equations
\begin{gather}\label{eq:equation-curve-side}
\begin{split}
&\bigl\{(u,v)\colon u\geq 0,\, v=\bigl((\sqrt{u^2+s^2}+\ab{\xi})^2-s^2\bigr)^{1/2}\bigr \},\\
\bigl\{(u,v)\colon u&\geq \bigl((\ab{\xi}+s)^2-s^2\bigr)^{1/2},\, 
v=\bigl((\sqrt{u^2+s^2}-\ab{\xi})^2-s^2\bigr)^{1/2}\bigr \}.
\end{split}
\end{gather}
Then $\ab{R_s\cap \ell_\tau}$ is given by
\[ \ab{R_s\cap \ell_\tau}=\begin{cases}
\sqrt{2}\tau &\text{, if }0\leq \tau\leq \bigl((\ab{\xi}+s)^2-s^2\bigr)^{1/2}\\
\sqrt{2}\ab{u-v} &\text{, if } \tau>\bigl((\ab{\xi}+s)^2-s^2\bigr)^{1/2},
\end{cases} \]
where in the last expression $u$ and $v$ are related to $(\xi,\tau)$ by the equations $u+v=\tau$ and
$v=\bigl((\sqrt{u^2+s^2}+\ab{\xi})^2-s^2\bigr)^{1/2}$. Therefore

\begin{align*}
 \sqrt{2}\ab{R_s\cap \ell_\tau}&=2\tau\one_{\{\tau\leq\sqrt{(\ab{\xi}+s)^2-s^2}\}}\\
&\quad+2((\sqrt{u_1(\xi,\tau)^2+s^2}+\ab{\xi})^2-s^2)^{\frac{1}{2}}
 -u_1(\xi,\tau))\one_{\{\tau>\sqrt{(\ab{\xi}+s)^2-s^2}\}},
\end{align*}
where $u_1(\xi,\tau)$ and $(\xi,\tau)$ are related by the expression 
\begin{equation}\label{eq:u1-tau}
\tau=u_1(\xi,\tau)+\bigl((\sqrt{u_1(\xi,\tau)^2+s^2}+\ab{\xi})^2-s^2\bigr)^{1/2},
\end{equation}
and $0\leq u_1(\xi,\tau)\leq \frac{\tau}{2}$.

\textbf{Case 2:} $\ab{\xi}>2s$. Now the boundary of the region 
$\{(\rho,\varsigma)\colon\ab{\rho-\varsigma} \leq \ab{\xi},\rho
+ \varsigma \geq \ab{\xi},\rho\geq s,\varsigma\geq s\}$ consists of three (bounded) line segments 
and two half lines and the region $R_s$ is now bounded by two line segments and three curves. The 
line 
segments have equations
\begin{gather*}
\{(u,v)\colon u=0,\, \sqrt{(\ab{\xi}-s)^2-s^2}\leq v\leq\sqrt{(\ab{\xi}+s)^2-s^2} \},\\
\{(u,v)\colon \sqrt{(\ab{\xi}-s)^2-s^2}\leq u\leq\sqrt{(\ab{\xi}+s)^2-s^2},\, v=0 \}.
\end{gather*}
The next two curves have equations as in \eqref{eq:equation-curve-side}. The last boundary curve is 
the image of the segment $\{(\rho,\varsigma)\colon \rho+\varsigma=\ab{\xi},\, s\leq\rho\leq 
\ab{\xi}-s \}$. Its equation is
\[ \{(u,v)\colon 0\leq u\leq \bigl((\ab{\xi}-s)^2-s^2\bigr)^{1/2} 
,\,v=\bigl((\ab{\xi}-\sqrt{u^2+s^2})^2-s^2\bigr)^{1/2}  \}, \]
and note that it is the graph of a strictly decreasing and concave function of $u$. It follows that
\[ \ab{R_s\cap \ell_\tau}=\begin{cases}
\sqrt{2}(\tau-\ab{u_2-v_2}) &\text{, if }\sqrt{(\ab{\xi}-s)^2-s^2}\leq \tau\leq 
\sqrt{\ab{\xi}^2-(2s)^2},\\
\sqrt{2}\tau &\text{, if }\sqrt{\ab{\xi}^2-(2s)^2}\leq \tau\leq \sqrt{(\ab{\xi}+s)^2-s^2}\\
\sqrt{2}\ab{u_1-v_1} &\text{, if } \tau\geq \bigl((\ab{\xi}+s)^2-s^2\bigr)^{1/2},
\end{cases}, \]
where $(u_1,v_1),\,(u_2,v_2)$ are the solutions to the equations $u_1+v_1=\tau,\,u_2+v_2=\tau$,  
$v_1=\bigl((\sqrt{u_1^2+s^2}+\ab{\xi})^2-s^2\bigr)^{1/2}$ and 
$v_2=\bigl((\ab{\xi}-\sqrt{u_2^2+s^2})^2-s^2\bigr)^{1/2}$.

Then 
\begin{align*}
  \sqrt{2}\ab{R_s\cap \ell_\tau}&=2\bigl(\tau-\bigl(((\ab{\xi}-\sqrt{u_2(\xi,\tau)^2+s^2})^2-s^2)^{1/2}\\
&\qquad\qquad-u_2(\xi,\tau)\bigr)\bigr)
\one_{\{\sqrt{(\ab{\xi}-s)^2-s^2 } \leq  \tau<\sqrt{\ab{\xi}^2-(2s)^2}\}}\\
  &\quad+2\tau\one_{\{\sqrt{\ab{\xi}^2-(2s)^2}\leq\tau\leq\sqrt{(\ab{\xi}+s)^2-s^2}\}}\\
  &\quad+2\bigl(((\sqrt{u_1(\xi,\tau)^2+s^2}+\ab{\xi})^2-s^2)^{1/2}-u_1(\xi,\tau)\bigr)\one_{\{\tau>\sqrt{(\ab{\xi}+s)^2-s^2}\}},
\end{align*}
where $u_1(\xi,\tau)$ is as in \eqref{eq:u1-tau} and $u_2(\xi,\tau)$ and $(\xi,\tau)$ are related by 
the 
expression
\[\tau=u_2(\xi,\tau)+\bigl((\sqrt{u_2(\xi,\tau)^2+s^2}-\ab{\xi})^2-s^2\bigr)^{1/2},\]
and $0\leq u_2(\xi,\tau)\leq \frac{\tau}{2}$. Algebraic manipulation shows that for $(\xi,\tau)$ in 
their respective domains of definition
\begin{equation}\label{eq:u1andu2}
\tau-2u_i(\xi,\tau)=\ab{\xi}\Bigl(1+\frac{4s^2}{\tau^2-\ab{\xi}^2}\Bigr)^{1/2},\quad i=1,2.
\end{equation}
Collecting all in one expression we have
\begin{align*}
 \sqrt{2}\ab{R_s\cap \ell_\tau}&=2\tau\one_{\{\tau\leq\sqrt{(\ab{\xi}+s)^2-s^2}\}}\one_{\{\ab{\xi}\leq 
 2s\}}\\
 &\qquad+4u_2(\xi,\tau)\one_{\{\sqrt{(\ab{\xi}-s)^2-s^2}\leq 
 \tau<\sqrt{\ab{\xi}^2-(2s)^2}\}}\one_{\{\ab{\xi}>2s\}}\\
 &\qquad+2\tau\one_{\{\sqrt{\ab{\xi}^2-(2s)^2}\leq\tau\leq\sqrt{(\ab{\xi}+s)^2-s^2}\}}\one_{\{\ab{\xi}>2s\}}\\
 &\qquad+2(\tau-2u_1(\xi,\tau))\one_{\{\tau>\sqrt{(\ab{\xi}+s)^2-s^2}\}}.
\end{align*}
Replacing $u_1(\xi,\tau)$ and $u_2(\xi,\tau)$ using \eqref{eq:u1andu2} we 
obtain using \eqref{eq:preliminar-formula-conv}
 	 \begin{equation}
 	 \label{eq:expresion-double-convolution}
 	 \begin{split}
 	 \mu_s\ast\mu_s(\tau,\xi)=\frac{2\pi}{\ab{\xi}}\biggl(&\tau\one_{\{\tau\leq\sqrt{(\ab{\xi}+s)^2-s^2}\}}
 	 \one_{\{\ab{\xi}\leq 2s\}}\\
 	 &+\Bigl(\tau-\ab{\xi}\Bigl(1+\frac{4s^2}{\tau^2-\ab{\xi}^2}\Bigr)^{1/2}\Bigr)\one_{\{\sqrt{(\ab{\xi}-s)^2-s^2}
 	 	\leq \tau<\sqrt{\ab{\xi}^2-(2s)^2}\}}\one_{\{\ab{\xi}>2s\}}\\
 	 &+\tau\one_{\sqrt{\ab{\xi}^2-(2s)^2}\leq\tau\leq\sqrt{(\ab{\xi}+s)^2-s^2}}\one_{\{\ab{\xi}>2s\}}\\
 	 &+\ab{\xi}\Bigl(1+\frac{4s^2}{\tau^2-\ab{\xi}^2}\Bigr)^{1/2}\one_{\{\tau>\sqrt{(\ab{\xi}+s)^2-s^2}\}}\biggr).
 	 \end{split}
 	 \end{equation}
Rearranging \eqref{eq:expresion-double-convolution} we find that $\mu_s*\mu_s$ can be written in the 
equivalent form \eqref{eq:alternative-conv-formula}.
\end{proof}
More generally, the same method used in the proof of Proposition \ref{prop:formula-double-convolution} 
allows 
us to write an explicit formula for $\mu_s\ast \mu_t$, for any $s,\,t\geq 0$. For instance, as it 
will be useful in Section \ref{sec:conv-to-cone}, we have
\begin{equation}\label{eq:formula-pre-mixed-conv}
\mu_s\ast\mu_0(\xi,\tau)=\frac{\pi}{\ab{\xi}}\int_{\widetilde{Q}_s(\xi)}\ddirac{\tau-b}\d a\d b,
\end{equation}
where $\widetilde{Q}_s(\xi)$ is the image of the set $\{(\rho,\varsigma)\colon 
\ab{\rho-\varsigma}\leq\ab{\xi},\,\rho+\varsigma\geq \ab{\xi},\,\rho\geq 0,\,\varsigma\geq s \}$ 
under the transformations $(\rho,\varsigma)\mapsto (u,v)=(\rho,\sqrt{\varsigma^2-s^2})\mapsto 
(a,b)=(u-v,u+v)$. Here $\mu_0$ equals $\sigma_c$, the Lorentz invariant measure on the 
cone. A calculation similar to the one for $\mu_s\ast\mu_s$ gives the following explicit formula
\begin{align}\label{eq:formula-mixed-conv}
\begin{split}
\mu_s\ast 
\sigma_c(\xi,\tau)=\frac{2\pi}{\ab{\xi}}\biggl(\frac{\ab{\xi}(\tau^2-\ab{\xi}^2+s^2)}{\tau^2-\ab{\xi}^2}
&\one_{\{\tau\geq s \}}\one_{\{\ab{\xi}<\tau-s \}}\\
+ \frac{(\tau+\ab{\xi})^2-s^2}{2(\tau+\ab{\xi})}
&\one_{\{\tau\geq 0 \}}\one_{\{\ab{\tau-s}\leq\ab{\xi}< \sqrt{\tau^2+s^2} \}}\\
+\frac{s^2-(\ab{\xi}-\tau)^2}{2(\ab{\xi}-\tau)}
&\one_{\{\tau\geq 0 \}}\one_{\{\sqrt{\tau^2+s^2}\leq\ab{\xi}\leq \tau+s \}}\biggr).
\end{split}
\end{align}
Using \eqref{eq:formula-mixed-conv} we see that for each $\tau\geq 0$
\begin{equation}\label{eq:infinity-norm-mixed-conv}
\sup_{\xi\in\R^3}\mu_s\ast\sigma_c(\xi,\tau)=2\pi\cdot
\begin{cases}
\frac{\tau}{\sqrt{\tau^2+s^2}}&, 0\leq \tau\leq \frac{s}{2}\sqrt{2(\sqrt{5}-1)}\\
1+\frac{(s-\sqrt{s^2-\tau^2})^2}{\tau^2}&,\frac{s}{2}\sqrt{2(\sqrt{5}-1)}\leq\tau\leq s\\
1+\frac{s}{2\tau-s}&,\tau\geq s
\end{cases}
\end{equation}
and $\norma{\mu_s\ast\sigma_c}_{L^\infty(\R^4)}=4\pi$.

\section{Comparison with the cone}\label{sec:compare_cone}

Recall that $\sigma_c$ denotes the scale and Lorentz invariant measure on the cone $\Gamma^3$ and $T_c$ denotes its associated adjoint Fourier restriction operator. From 
\cite{Ca} we know the value of the sharp constant
\begin{equation}\label{eq:sharp-cone-estimate}
\sup_{0\neq f\in L^2(\sigma_c)}\frac{\norma{f\sigma_c\ast 
		f\sigma_c}_{L^4(\R^4)}^2}{\norma{f}_{L^2(\sigma_c)}^4}= 2\pi.
\end{equation}
We had defined the 
numerical constants
\begin{align*}
\mathbf{C}_{4}&=\sup_{0\neq f\in 
	L^2(\sigma_c)}\frac{\norma{T_cf}_{L^4(\R^4)}}{\norma{f}_{L^2(\sigma_c)}}=2\pi\sup_{0\neq 
	f\in L^2(\sigma)}\frac{\norma{f\sigma_c\ast 
		f\sigma_c}_{L^4(\R^4)}^{1/2}}{\norma{f}_{L^2(\sigma_c)}},\\
\mathbf{H}_{4}&=\sup_{0\neq f\in 
	L^2(\mu)}\frac{\norma{Tf}_{L^4(\R^4)}}{\norma{f}_{L^2(\mu)}}=2\pi\sup_{0\neq 
	f\in L^2(\mu)}\frac{\norma{f\mu\ast 
		f\mu}_{L^4(\R^4)}^{1/2}}{\norma{f}_{L^2(\mu)}}.
\end{align*}

The next proposition gives a comparison between $\mathbf{C}_{4}$ and $\mathbf{H}_4$ and its role is the 
analog 
of the comparison of the best constant for the sphere and the paraboloid in $\R^3$ as used in 
\cite{CS} where an strict inequality was needed to rule out concentration at a pair of 
antipodal points. In our present case, an strict inequality will rule out 
concentration at infinity.

\begin{prop} \label{prop:comparison-cone}
	$\qquad\ds \mathbf{H}_{4}>\mathbf{C}_{4}.$
\end{prop}
\begin{proof}
We consider the family of trial functions $f_a(y)=e^{-\frac{a}{2}\sqrt{\ab{y}^2-s^2}}$, $a>0$, and 
claim that
\[\sup_{a>0}\frac{\norma{T_sf_a}_{L^4(\R^4)}}{\norma{f_a}_{L^2(\mathcal H^3_s)}}> \sup_{0\neq f\in
L^2(\sigma_c)}\frac{\norma{T_cf}_{L^4(\R^4)}}{\norma{f}_{L^2(\sigma_c)}}.\]

\noindent Using spherical coordinates, the $L^2(\mathcal H^3_s)$-norm of $f_a$ is given by the 
expression
\begin{align*}
 \norma{f_a}_{L^2(\mathcal H^3_s)}^2&=\int_{\R^3} 
 e^{-a\sqrt{\ab{x}^2-s^2}}\frac{\d x}{\sqrt{\ab{x}^2-s^2}}=4\pi\int_s^\infty
e^{-a\sqrt{r^2-s^2}}\frac{r^2}{\sqrt{r^2-s^2}}\d r\\
 &=4\pi\int_0^\infty e^{-a\tau}\sqrt{\tau^2+s^2}\d \tau.
\end{align*}
It is easier to estimate $\norma{T_sf_a}_{L^4(\R^4)}$ if we use the convolution form \eqref{eq:convolution-form},
$$\norma{T_sf_a}_{L^4(\R^4)}=2\pi\norma{f_a\mu_s*f_a\mu_s}_{L^2(\R^4)}^{1/2}.$$
As in  
\cite{RQ2}*{Appendix 2}, using that $f_a$ is the
restriction to $\mathcal H^3_s$ of the exponential of the linear function in $\R^4$, 
$(\xi,\tau)\mapsto e^{-\frac{a}{2}\tau}$, 
we obtain
\[f_a\mu_s*f_a\mu_s(\xi,\tau)=e^{-\frac{a}{2}\tau}\bigl(\mu_s*\mu_s(\xi,\tau)\bigr).\]
It will be 
enough for our 
purpose to use 
\begin{align*}
 \mu_s*\mu_s(\xi,\tau)\geq &\frac{2\pi}{\ab{\xi}}\biggl(
\ab{\xi}\Bigl(1+\frac{4s^2}{\tau^2-\ab{\xi}^2}\Bigr)^{\frac{1}{2}}\one_{\{\ab{\xi}<\sqrt{\tau^2+s^2}-s\}}+\tau
\one_{\{\sqrt{\tau^2+s^2}-s\leq \ab{\xi}\leq \sqrt{\tau^2+(2s)^2}\}}\biggr),
\end{align*}
as obtained from \eqref{eq:alternative-conv-formula}. In this way
\begin{multline*}
f_a\mu_s*f_a\mu_s(\xi,\tau)\geq \\
\frac{2\pi }{\ab{\xi}}e^{-\frac{a}{2}\tau} \biggl(
\ab{\xi}\Bigl(1+\frac{4s^2}{\tau^2-\ab{\xi}^2}\Bigr)^{\frac{1}{2}}\one_{\{\ab{\xi}<\sqrt{\tau^2+s^2}-s\}}+\tau
\one_{\{\sqrt{\tau^2+s^2}-s\leq \ab{\xi}\leq \sqrt{\tau^2+(2s)^2}\}}\biggr),
\end{multline*}
so that using spherical coordinates we obtain
\begin{align*}
 \norma{f_a\mu_s*f_a\mu_s&}_{L^2(\R^4)}^2
 \geq (2\pi)^2\int_{\R^3\times\R} e^{-a\tau}\biggl(
\ab{\xi}^2\Bigl(1+\frac{4s^2}{\tau^2-\ab{\xi}^2}\Bigr)\one_{\{\ab{\xi}<\sqrt{\tau^2+s^2}-s\}}\\
&\qquad\hspace{4.5cm}+\tau^2\one_{\{\sqrt{\tau^2+s^2}-s\leq \ab{\xi}\leq \sqrt{\tau^2+(2s)^2}\}}
 \biggr)\d\tau\frac{\d\xi}{\ab{\xi}^2}\\
&= 16\pi^3\int_0^\infty
e^{-a\tau}\biggl(\tau^2(\sqrt{\tau^2+(2s)^2}-\sqrt{\tau^2+s^2}+s)+\frac{1}{3}(\sqrt{\tau^2+s^2}-s)^3\\
&\qquad\qquad-4s^2(\sqrt{\tau^2+s^2}-s)+2s^2\tau\,\log\Bigl(\frac{\tau+\sqrt{\tau^2+s^2}}{s}\Bigr)\biggr)\d\tau\\
&=16\pi^3 
\int_0^{\infty}e^{-a\tau}\biggl(\tau^2\sqrt{\tau^2+4s^2}-\frac{2}{3}(\tau^2+4s^2)\sqrt{\tau^2+s^2}
 +\frac{8s^3}{3}\\
 &\qquad\qquad+ 2s^2\tau \log\Bigl(\frac{\tau+\sqrt{\tau^2+s^2}}{s}\Bigr)\biggr)\d\tau.
\end{align*}
Since by scaling it is enough to consider $s=1$ (see Section \ref{sec:scaling}) we let 
\begin{align*}
&I(a)=16\pi^3 
\int_0^{\infty}e^{-a\tau}\Bigl(\tau^2\sqrt{\tau^2+4}-\tfrac{2}{3}(\tau^2+4)\sqrt{\tau^2+1}
+\tfrac{8}{3}\\
&\qquad\qquad\qquad\qquad+ 2\tau \log\bigl(\tau+\sqrt{\tau^2+1}\bigr)\Bigr)\d\tau,\\
&II(a)=16\pi^2\biggl(\int_0^\infty e^{-a\tau}\sqrt{\tau^2+s^2}\d \tau\biggr)^2,
\end{align*}
then
\[ 
\frac{\norma{f_a\mu*f_a\mu}_{L^2(\R^4)}^2}{\norma{f_a}_{L^2(\mu)}^4}
\geq\frac{I(a)}{II(a)}.
 \]
Form Lemma \ref{lem:asymptotics-hyp} in the appendix, we conclude that for all $a>0$ small enough
\begin{equation}\label{eq:good_a_comparison}
\frac{\norma{f_a\mu*f_a\mu}_{L^2(\R^4)}^2}{\norma{f_a}_{L^2(\mu)}^4}>2\pi.
\end{equation}
This finishes the proof in view of \eqref{eq:sharp-cone-estimate}.
\end{proof}

\begin{remark}
	The easy lower bound we can obtain for 
	$\norma{f_a\mu*f_a\mu}_{L^2(\R^4)}^2\norma{f_a}_{L^2(\mu)}^{-4}$ using the analog of
	\cite{OSQ}*{Lemma 6.1} is not good enough in this case to obtain \eqref{eq:good_a_comparison}.
\end{remark}

Let us now move to the full one-sheeted hyperboloid $\Hyp$. Recall that $\overline{T}_c$ denotes the adjoint Fourier 
restriction operator on the double cone 
$\overline{\Gamma}^3$. An argument in 
\cite{Fo} can be used to 
show that
\begin{equation}\label{eq:best-double-cone}
\norma{\overline{T}_c}=\Bigl(\frac{3}{2}\Bigr)^{\frac{1}{4}}\norma{T_c},
\end{equation}
see for instance \cite{RQ2}*{Proposition 7.3}. We now compare the best constant for $\overline{T}$ and 
 $\overline{T}_c$.
\begin{prop}\label{prop:comparison-double-cone}
	For the one-sheeted hyperboloid $\Hyp$ and its adjoint Fourier restriction operator $\overline{T}$ we 
	have
	\[ \norma{\overline{T}}>\norma{\overline{T}_c}. \]
\end{prop}
\begin{proof}
	Let $f_a(y)=e^{-\frac{a}{2}\sqrt{\ab{y}^2-1}}$ be as in the proof of Proposition \ref{prop:comparison-cone} and set $g_a=f_{a,+}+f_{a,-}$, where $f_{a,+}=cf_a$ and $f_{a,-}=cf_a$ (here 
	there are domain identifications through projections to $\R^3$), in other words, 
	$g_a(\xi,\tau)=ce^{-\frac{a}{2}\ab{\tau}}\mathbbm{1}_{\overline{\mathcal H}^3}(\xi,\tau)$, where $c$ is 
	such that $g_a$ is $L^2$ normalized. Expanding and using the positivity of $f_{a,+}$ and $f_{a,-}$ (which for brevity we simply call $f_+$ and $f_-$, respectively) we see that
	\begin{align*}
	\norma{\overline{T}g_a}_{L^4}^4&= 
	\norma{Tf_+}_{L^4(\R^3)}^4+\norma{Tf_-}_{L^4(\R^3)}^4+4\norma{(Tf_+)(Tf_-(\cdot,-\cdot))}_{L^2}^2\\
	&\qquad+4(2\pi)^4\langle f_+\mu\ast f_+\mu,f_+\mu\ast f_-\mu_-\rangle\\
	&\qquad+4(2\pi)^4\langle f_+\mu\ast f_-\mu_-,f_-\mu_-\ast f_-\mu_-\rangle\\
	&\geq \norma{Tf_+}_{L^4(\R^3)}^4+\norma{T_-f_-}_{L^4(\R^3)}^4+4\norma{(Tf_+)(Tf_-(\cdot,-\cdot))}_{L^2}^2.
	\end{align*}
	On the other hand $Tf_-(\cdot,-\cdot)=\overline{Tf_+}$, the complex conjugate, since $f_-(y)=f_+(-y)$. Then 
	$\norma{(Tf_+)(Tf_-(\cdot,-\cdot))}_{L^2}^2=\norma{Tf_+}_{L^4(\R^3)}^4=\norma{Tf_-}_{L^4(\R^3)}^4$ and we 
	obtain
	\[ \norma{\overline{T}g_a}_{L^4}^4\geq 6\norma{Tf_{a,+}}_{L^4(\R^3)}^4. \]
	If $a>0$ is small enough, then from \eqref{eq:good_a_comparison} in the proof of Proposition 
	\ref{prop:comparison-cone} and using $\norma{f_{a,+}}_{L^2(\mu)}=\sqrt{2}/2$, we obtain
	\[ \norma{\overline{T}g_a}_{L^4}^4\geq 
	6\norma{Tf_{a,+}}_{L^4(\R^3)}^4>6\norma{T_c}^4\norma{f_{a,+}}_{L^2(\mu)}^4=\frac{3}{2}\norma{T_c}^4. 
	\]
	The conclusion follows using \eqref{eq:best-double-cone}.
\end{proof}

\section{The upper half of the one-sheeted hyperboloid}\label{sec:upper-half}

In this section we present the proof of Theorem \ref{thm:main-theorem}. The proof of precompactness of extremizing sequences, modulo multiplication by characters, is much 
simpler for the upper half of the one-sheeted hyperboloid as the full Lorentz invariance of 
$\overline{\mathcal H}^3$ is absent for $\mathcal{H}^3$. 

In what follows we collect the necessary results to invoke Proposition \ref{prop:key-prop-from-fvv} and 
the first such step is to show that, with enumeration as in Proposition \ref{prop:key-prop-from-fvv}, (i) and (iii) imply (iv), possibly after passing to a 
subsequence.

\begin{prop}\label{prop:pointwise-convergence}
	Let $\{f_n\}_{n}$ be a sequence in $L^2(\mathcal H^3)$ satisfying
	$\sup_n\norma{f_n}_{L^2(\mathcal H^3)}<\infty$.
	Suppose that there
	exists $f\in L^2(\mathcal H^3)$ such that $f_n\rightharpoonup f$ as $n\to\infty$. Then, there 
	exists a subsequence
	$\{f_{n_k}\}_{k}$
	such that $Tf{_{n_k}}\to Tf$ a.e. in $\R^4$.
\end{prop}

The previous result implies an analogous one for the full two-sheeted hyperboloid $\Hyp$. Recall the Fourier multiplier notation
\begin{equation}\label{eq:multiplier_notation}
e^{it\sqrt{-\Delta-s^2}}u(x)=\frac{1}{(2\pi)^3}\int_{\{y\in\R^3\colon\ab{y}\geq s\}}e^{ix\cdot y}e^{it\sqrt{\ab{y}^2-s^2}}\hat{u}(y)\d y,
\end{equation}
and the homogeneous $\dot H^{1/2}(\R^3)$ Sobolev norm and inner product
\begin{equation}\label{eq:def_homogeneous_Sobolev}
\norma{u}_{\dot H^{1/2}(\R^3)}^2:=\int_{\R^3}\ab{\hat u(y)}^2\ab{y}\d y, \quad\langle u,v \rangle_{\dot H^{1/2}(\R^3)}:=\int_{\R^3}\hat u(y)\overline{\hat v(y)}\ab{y}\d y. 
\end{equation}

\begin{proof}[Proof of Proposition \ref{prop:pointwise-convergence}]
	The proof follows similar lines to the proofs of \cite{FVV2}*{Theorem 1.1} and 
	\cite{RQ1}*{Proposition 8.3}. We start by splitting 
	$f_n=f_n\one_{B(0,2)}+f_n\one_{B(0,2)^\complement}=:f_{n,1}+f_{n,2}$, respectively,
	and $f=f\one_{B(0,2)}+f\one_{B(0,2)^\complement}=:f_{1}+f_{2}$. The conclusion of the proposition will follow if 
	we show 
	that 
	there exists a
	subsequence $\{f_{n_k}\}_{k}$ such that $Tf_{n_k,1}\to Tf_1$ and $Tf_{n_k,2}\to Tf_2$ a.e. 
	in $\R^4$.
	
	Since $f_{n,1}\rightharpoonup f_1$ in $L^2(\mathcal H^3)$ and the support of $f_{n,1}$ is 
	contained on the compact
	set $B(0,2)$, it follows that $Tf_{n,1}(x,t)\to Tf_1(x,t)$ for all $(x,t)\in\R^4$ given that the function
	$(y,s)\mapsto e^{ix\cdot
		y}e^{its}\one_{B(0,2)}(y)$ belongs to
	$L^2(\mathcal H^3)$.
	
	To study the pointwise convergence of
	$Tf_{n,2}$ define $g_n$ and $g$ by their Fourier transforms as follows
	\[\hat g_n(y)=\frac{f_{n,2}(y)}{\sqrt{\ab{y}^2-1}},\quad\quad \hat 
	g(y)=\frac{f_{2}(y)}{\sqrt{\ab{y}^2-1}}.\]
	Because 
	\[\norma{f_{n,2}}_{L^2(\mathcal H^3)}^2=\int_{\{y\in\R^3:\ab{y}\geq 2\}}\ab{f_n(y)}^2\frac{\d 
	y}{\sqrt{\ab{y}^2-1}}\leq 
	\sup_k \norma{f_k}_{L^2(\mathcal H^3)}^2=: c,\]
	we see that the norms of the $g_n$'s in the homogeneous Sobolev space $\dot{H}^{1/2}(\R^3)$ are uniformly bounded
	\[\norma{g_n}_{\dot H^{1/2}(\R^3)}^2=\int_{\R^3}\ab{\hat g_n(y)}^2\ab{y}\d 
	y\leq\frac{2}{\sqrt{3}}\int_{\{y\in\R^3:\ab{y}\geq 2\}}\ab{f_{n,2}(y)}^2\frac{\d 
		y}{\sqrt{\ab{y}^2-1}}\leq \frac{2c}{\sqrt{3}}.\]
	The weak convergence of $\{f_{n,2}\}_n$ to $f_2$ in $L^2(\hyp)$ easily implies $g_n\rightharpoonup g$ as $n\to\infty$ in $\dot{H}^{1/2}(\R^3)$. On the other hand \[(2\pi)^{3}\norma{g_n}_{L^2(\R^3)}^2=\norma{\widehat{g_n}}_{L^2(\R^3)}^2=\int_{\{y\in\R^3:\ab{y}\geq 2\}}\frac{\ab{f_{n,2}(y)}^2}{(\sqrt{\ab{y}^2-1})^2}\frac{\d 
		y}{\sqrt{\ab{y}^2-1}}\leq \frac{c}{\sqrt{3}},\]
	so $\{g_n\}_n$ is uniformly bounded in $L^2(\R^3)$.
	
	The operator $T$ applied to 
	$f_{n,2}$ equals
	$(2\pi)^3e^{it\sqrt{-\Delta-1}}g_n$, where the operator $e^{it\sqrt{-\Delta-1}}$ is understood in the Fourier multiplier sense as in \eqref{eq:multiplier_notation}. Let
	$t\in \R$ be fixed. By the continuity of $e^{it\sqrt{-\Delta-1}}$ 
	in $\dot H^{1/2}(\R^3)$ we obtain 
	\[e^{it\sqrt{-\Delta-1}}g_n\rightharpoonup e^{it\sqrt{-\Delta-1}}g\]
	weakly in $\dot H^{1/2}(\R^3)$, as $n\to\infty$. Then, by the Rellich-Kondrashov Theorem (\cite{DNPV12}*{Theorem 7.1}), for any $R>0$
	\[e^{it\sqrt{-\Delta-1}}g_n\to e^{it\sqrt{-\Delta-1}}g\]
	strongly in $L^2(B(0,R))$, as $n\to\infty$. Denote by
	\[F_n(t):=\int_{\ab{x}<R}\abs{e^{it\sqrt{-\Delta-1}}(g_n-g)}^2\d 
	x=\norma{e^{it\sqrt{-\Delta-1}}(g_n-g)}_{ L^2(B(0,R))}^2.\]
	By H\"older's inequality and Sobolev embedding, \cite{DNPV12}*{Theorem 6.5}, we obtain
	\begin{align*}
	F_n(t)&=\norma{e^{it\sqrt{-\Delta-1}}(g_n-g)}_{L^2(B(0,R))}^2\leq CR\norma{e^{it\sqrt{-\Delta-1}}(g_n-g)}_{L^3(B(0,R))}^2\\
	&\leq CR\norma{e^{it\sqrt{-\Delta-1}}(g_n-g)}_{\dot{H}^{1/2}(\R^3)}^2\leq \frac{8}{\sqrt{3}}cCR,
	\end{align*}
	then, by the Fubini and Dominated Convergence Theorems we have that
	\[\int_{-R}^R 
	F_n(t)dt=\int_{-R}^{R}\int_{\ab{x}<R}\abs{e^{it\sqrt{-\Delta-1}}(g_n-g)}^2\d 
	x\d t\to 0,\]
	as $n\to \infty$. This implies that, up to a subsequence,
	\[e^{it\sqrt{-\Delta-1}}g_n(x)-e^{it\sqrt{-\Delta-1}}g(x)\to 0\quad\text{a.e. }(x,t)\in B(0,R)\times (-R,R).\]
	Repeating the argument on a discrete sequence of radii $R_n$ such that $R_n\to\infty$, as $n\to\infty$ we conclude, by a diagonal
	argument, that there exists a subsequence $\{g_{n_k}\}_k$ of $\{g_n\}_n$ such that
	\[e^{it\sqrt{-\Delta-1}}g_{n_k}(x)-e^{it\sqrt{-\Delta-1}}g(x)\to 0\quad\text{a.e. for 
	}(x,t)\in\R^4,\]
	or equivalently, in terms of the sequence $\{f_{n,2}\}_{n}$ and the operator $T$,
	\[Tf_{n_k,2}(x,t)\to Tf_2(x,t)\quad\text{a.e. }(x,t)\in\R^4.\]
\end{proof}

We now show that the only obstruction to precompactness of extremizing sequences is the possibility of concentration at 
infinity, as in Definition \ref{def:conv-inf}.
\begin{prop}
	\label{prop:convergence-not-concentration-infinity}
	Let $\{f_n\}_{n}\subset L^2(\mathcal H^3)$ be an $L^2$ normalized extremizing sequence for $T$. Suppose 
	that $\{f_n\}_{n}$ does not
	concentrate at infinity. Then there exist a 
	subsequence $\{f_{n_k}\}_k$ and a sequence $\{(x_k,t_k)\}_k\subset\R^4$ such that $\{e^{ix_k\cdot
		y}e^{it_k\sqrt{\ab{y}^2-1}}f_{n_k}\}_{k}$ is convergent in $L^2(\mathcal H^3)$.
\end{prop}

\begin{proof}
	If $\{f_n\}_{n}$
	does not concentrate at infinity, then there exist 
	$\eps,R>0$ with the property that for all 
	$N\in\N$ there exists $n\geq N$ such that
	\[\norma{f_n\one_{B(0,R)}}_{L^2(\hyp)}\geq \eps.\]
	We can generate a subsequence, $\{f_{n_k}\}_{k}$, such that 
	$\norma{f_{n_k}\one_{B(0,R)}}_{L^2(\hyp)}\geq \eps$ for all $k\in\N$. Rename 
	the subsequence as $\{f_n\}_n$, if necessary. Writing 
	$f_n=f_n\one_{B(0,R)}+f_n\one_{B(0,R)^\complement}=:f_{n,1}+f_{n,2}$, respectively, we have
	\begin{align}
	\norma{Tf_{n,1}}_{L^4(\R^4)}&=\norma{T(f_n-f_{n,2})}_{L^4(\R^4)}\geq
	\norma{Tf_n}_{L^4(\R^4)}-\norma{Tf_{n,2}}_{L^4(\R^4)}\nonumber\\
	&\geq \norma{Tf_n}_{L^4(\R^3)}-\mathbf{H}_{4}\norma{f_{n,2}}_{L^2(\mathcal H^3)}\nonumber\\
	&=
	\norma{Tf_n}_{L^4(\R^3)}-\mathbf{H}_{4}(1-\norma{f_{n,1}}_{L^2(\mathcal H^3)}^2)^{1/2}\nonumber\\
	\label{eq:last-inequality}
	&\geq \norma{Tf_n}_{L^4(\R^3)}-\mathbf{H}_{4}\sqrt{1-\eps^2}.
	\end{align}
	As the right hand side in \eqref{eq:last-inequality} converges to 
	$c:=\mathbf{H}_{4}-\mathbf{H}_{4}\sqrt{1-\eps^2}>0$ as $n\to\infty$ we see that
	\begin{equation}
	\label{eq:lower-bound-truncated-sequence}
	\norma{Tf_{n,1}}_{L^4(\R^4)}\geq \frac{c}{2}>0,
	\end{equation}
	for all large $n$.
	
	We may use the argument in the proof of \cite{FVV}*{Theorem 1.1} to construct the sequence $\{(x_n,t_n)\}_n$. In brief, the argument goes as follows. Taking any $\bar p\in[\frac{10}{3},4)$, interpolating the $L^4$ norm of $Tf_{n,1}$ between $L^{\bar p}$ and $L^\infty$ and using \eqref{eq:lower-bound-truncated-sequence} together with the boundedness of $T$ in $L^{\bar{p}}$ imply that $\norma{Tf_{n,1}}_{L^\infty(\R^4)}\gtrsim 1$, so that there exists a sequence $\{(x_n,t_n)\}_n\subset\R^4$ such that $\ab{Tf_{n,1}(x_n,t_n)}\geq C>0$, that is, $\ab{(T(e^{ix_n\cdot y}e^{it_n\sqrt{\ab{y}^2-1}}f_{n,1}))(0,0)}\geq C>0$. The compact support of $f_{n,1}$ implies that $Tf_{n,1}$ belongs to $C^\infty(\R^4)$ and $\norma{Tf_{n,1}}_{L^\infty(\R^4)}\lesssim \norma{f_{n,1}}_{L^1}\lesssim 1$, $\norma{\nabla_{x,t} Tf_{n,1}}_{L^\infty(\R^4)}\lesssim \norma{f_{n,1}}_{L^1}\lesssim 1$. By the Arzel\'a--Ascoli Theorem, it follows that $\{T(e^{ix_n\cdot y}e^{it_n\sqrt{\ab{y}^2-1}}f_{n,1})\}_n$ is precompact in the space of continuous functions on the unit ball of $\R^4$. On the other hand, passing to a subsequence, we may assume $F_n:=e^{ix_n\cdot y}e^{it_n\sqrt{\ab{y}^2-1}}f_{n,1}\rightharpoonup f_1$ weakly in $ L^2(\hyp)$, for some $f_1\in L^2(\hyp)$, and then $T(F_{n})(x,t)\to Tf_1(x,t)$ for all $(x,t)\in\R^4$. Moreover, $T(F_n)\to Tf_1$ uniformly in the unit ball, so that $\ab{Tf_1(0,0)}\geq C>0$, which implies that $f_1\neq 0$.
	
	Compactness of the unit ball in $L^2(\mathcal H^3)$ in the weak topology implies that, after 
	passing 
	to a further subsequence,
	$e^{ix_n\cdot y}e^{it_n\sqrt{\ab{y}^2-1}}f_n\rightharpoonup f$, for some $f\in L^2(\mathcal 
	H^3)$. 
	Since $f_1=f\one_{B(0,R)}$ a.e.
	in $\R^3$ we conclude that $f\neq 0$. Therefore condition (iii) of Proposition 
	\ref{prop:key-prop-from-fvv} is satisfied. Proposition
	\ref{prop:pointwise-convergence} implies that condition (iv) is also satisfied. As (i) and (ii) are 
	immediate, we conclude that
	$e^{ix_n\cdot y}e^{it_n\sqrt{\ab{y}^2-1}}f_n\to f$ in $L^2(\mathcal H^3)$, and we are done.
\end{proof}

To conclude the precompactness of extremizing sequences we need to discard the possibility of concentration at infinity. For this we use a comparison argument with the cone
where the upper bound for $\mu_s*\mu_s$ 
as found in Lemma \ref{lem:behavior-infinity} will be useful.

\begin{lemma}\label{lem:no-concentration-infinity}
	Let $a>1$ and $f\in L^2(\mathcal{H}^3)$ be supported in the region where $\ab{y}\geq a$. Then
	\[ \norma{f\mu\ast f\mu}_{L^2(\R^4)}^2\leq 
	2\pi\Bigl(1+\frac{1}{\sqrt{a^2-1}}\Bigr)\norma{f}_{L^2(\mu)}^4. 
	\]
\end{lemma}

\begin{proof}
	If $f$ is supported where $\ab{y}\geq a$, then the support of $f\mu\ast f\mu$ is contained in 
	the region $\{(\xi,\tau)\in\R^4\colon \tau\geq 2\sqrt{a^2-1}\}$.
	The Cauchy-Schwarz inequality provides the a.e. pointwise bound
	\[ \ab{f\mu\ast f\mu}^2(\xi,\tau)\leq 
	\bigl(\ab{f}^2\mu\ast\ab{f}^2\mu\bigr)(\xi,\tau)\bigl(\mu\ast\mu\bigr)(\xi,\tau), \]
	which together with the upper bound in Lemma \ref{lem:behavior-infinity} imply
	\[ \ab{f\mu\ast f\mu}^2(\xi,\tau)\leq 
	2\pi\Bigl(1+\frac{1}{\sqrt{a^2-1}}\Bigr)\bigl(\ab{f}^2\mu\ast\ab{f}^2\mu\bigr)(\xi,\tau), \]
	for a.e. $(\xi,\tau)\in\R^4$. Integrating over $(\xi,\tau)\in\R^4$ yields the result.
\end{proof}

It is now direct to prove our first main theorem.
\begin{proof}[Proof of Theorem \ref{thm:main-theorem}]
	We start by noting that if an $L^2$-normalized sequence $\{f_n\}_n$ concentrates at infinity, 
	then  for any 
	$a>1$,
	$\norma{f_n\one_{B(0,a)}}_{L^2(\mu)}\to 0$ as $n\to\infty$, therefore, for such a sequence we
	obtain, using Lemma \ref{lem:no-concentration-infinity}, that 
	\[\limsup_{n\to\infty}\frac{\norma{f_n\mu\ast 
	f_n\mu}_{L^2(\R^4)}^2}{\norma{f_n}_{L^2(\mu)}^4}\leq 2\pi.\]
	
	Using Proposition \ref{prop:comparison-cone} we conclude that an 
	extremizing sequence for $T$ does not concentrate at infinity. We apply Proposition 
	\ref{prop:convergence-not-concentration-infinity} to conclude.
\end{proof}

\section{The full one-sheeted hyperboloid}\label{sec:full_hyperboloid}

Our task in the sections to come is to prove Theorem \ref{thm:main-theorem-2}, the existence of extremals for the adjoint Fourier restriction inequality on the one-sheeted 
hyperboloid 
$\Hyp$. In the $L^4$ case, there is an argument
available for the cone $\Gamma^3$ that allows to relate the best constant and extremizers for $\Gamma^3$ with that for the double cone $\overline{\Gamma}^3$. It relies on the observation that
the algebraic sums $\Gamma^3+\Gamma^3$ and $\Gamma^3+(-\Gamma^3)$ 
intersect on a null set of $\R^3$, namely, 
$(\Gamma^3+\Gamma^3)\cap(\Gamma^3+(-\Gamma^3))=\Gamma^3$, so that for any $f_+,g_+,h_+\in L^2(\Gamma^3)$ and $f_-\in L^2(-\Gamma^3)$ one has
\[ \bigl\langle f_+\sigma_c\ast g_+\sigma_c, h_+\sigma_c\ast f_-\widetilde{\sigma}_c \bigr\rangle_{L^2(\R^4)}=0, \]
where $\widetilde{\sigma}_c$ denotes the reflection of $\sigma_c$, supported on $-\Gamma^3$.
An analogous property in the $L^4$ case applies to the two-sheeted hyperboloid in $\R^4$ and allows one to obtain its best constant from that of the upper 
sheet only (see \cite{RQ2}*{Proposition 7.3, Corollary 7.4}).
This approach is not applicable to $\Hyp$ because here $\hyp+\hyp$ and 
$\hyp+(-\hyp)$ intersect on a set of infinite Lebesgue measure.

The argument we use to prove precompactness of extremizing sequences (modulo 
multiplication by characters and Lorentz transformations) is close to that of 
Christ and Shao \cite{CS} and of \cite{RQ1} using a 
concentration-compactness argument, a refined
cap estimate, comparison to the cone and the use of Lorentz invariance. 

In the next section we review a cap refinement for the Tomas--Stein inequality for $\Sph^2$ that will be used in the subsequent section to obtain a corresponding cap refinement for the adjoint Fourier restriction inequality on the hyperboloid via a lifting method. In later sections we consider the concentration-compactness argument.

\section{The Tomas--Stein inequality for \texorpdfstring{$\Sph^2$}{S2} and refinements}\label{sec:Tomas_Stein_sphere}

The sharp convolution form of the Tomas--Stein inequality for $\Sph^2$ states that for all $f\in L^2(\Sph^2)$ we have
\begin{equation}\label{eq:convolution_ts}
\norma{f\sigma\ast f\sigma}_{L^2(\R^3)}\leq \mathbf{S}^2\norma{f}_{L^2(\Sph^2)}^2,
\end{equation}
where $\mathbf{S}=(2\pi)^{1/4}$ is the sharp constant, as obtained in \cite{Fo}.

In this section we review some refinements of \eqref{eq:convolution_ts} that will be useful in the next section. The exposition here follows that of \cite{CS}*{Section 6}.
We start by setting things up to define the $X_p$ spaces, $p\in[1,\infty)$, and the first step is to generate increasingly finer "grids" for the sphere $\Sph^2$. With this in mind, for each integer $k\geq 0$ choose a maximal subset $\{y_k^j\}_j\subset \Sph^2$ satisfying $\ab{y_k^j-y_k^l}\geq 2^{-k}$, for all $j\neq l$. Then, for each $k\geq 0$, the spherical caps $\sphcp_k^j:=\sphcp(y_k^j,2^{-k+1})$ have finite overlap and cover $\Sph^2$, that is, $\cup_j\sphcp_k^j=\Sph^2$, and there exists a constant $C$, independent of $k$, such that any point in $\Sph^2$ belongs to no more than $C$ caps in  $\{\sphcp_k^j\}_j$, for every $k\geq 0$. For $p\in[1,\infty)$, the $X_p$ norm of $f$ is defined by the expression
\begin{equation}\label{eq:xp_norm}
\norma{f}_{X_p}^4=\sum_{k\geq 0}\sum_{j}2^{-4k}\biggl(\frac{1}{\ab{\sphcp_k^j}}\int_{\sphcp_k^j}\ab{f}^p\d\sigma\biggr)^{4/p}.
\end{equation}
Moyua, Vargas and Vega showed in \cite{MVV} that there is a continuous inclusion $L^2(\Sph^2)\subset X_p$ for all $p\in(1,2)$ and that for any $p\geq \frac{12}{7}$ there exists $C<\infty$ such that for all $f\in L^ 2(\Sph^2)$
\begin{equation}\label{eq:MVVrefinement}
\norma{\widehat{f\sigma}}_{L^4(\R^3)}\leq C\norma{f}_{X_p}.
\end{equation}
Let us define
\[ \Lambda_{k,j}(f)=\Bigl(\ab{\sphcp_k^j}^{-1}\int_{\sphcp_k^j}\ab{f}\d\sigma\Bigr)\Bigl(\ab{\sphcp_k^j}^{-1}\int_{\Sph^2}\ab{f}^2\d\sigma\Bigr)^{-1/2}, \]
which by H\"older's inequality satisfies $\La_{k,j}(f)\leq 1$. It was shown in \cite{CS}*{Lemma 6.1} that for any $p\in[1,2)$, there exists $C<\infty$ and $\gamma>0$ such that for any $f\in L^2(\Sph^2)$,
\begin{equation}\label{eq:CSXprefinement}
\norma{f}_{X_p}\leq C\norma{f}_{L^2(\Sph^2)}\sup_{k,j}(\La_{k,j}(f))^\gamma.
\end{equation}
Combining the two results, \eqref{eq:MVVrefinement} and \eqref{eq:CSXprefinement}, by choosing any $p\in[\frac{12}{7},2)$, we obtain the following "cap refinement" of \eqref{eq:convolution_ts}: there exists $C<\infty$ and $\gamma>0$ such that for all $f\in L^2(\Sph^2)$
\begin{equation}\label{eq:cap_refinement_TS}
\norma{\widehat{f\sigma}}_{L^4(\R^3)}\leq C\norma{f}_{L^2(\Sph^2)}\sup_{k,j}(\La_{k,j}(f))^\gamma.
\end{equation}

A $\delta$-\textit{quasi-extremal} for the sphere is a function $f\neq 0$ that satisfies $\norma{f\sigma\ast f\sigma}_{L^4(\R^3)}\geq \delta^2\mathbf{S}^2\norma{f}_{L^2(\Sph^2)}^2$. With the aid of the previous inequality, Christ and Shao proved the following result regarding $\delta$-quasi-extremals.

\begin{lemma}[\cite{CS}*{Lemma 2.9}]
	\label{lem:quasiextremals-sphere-CS}
	For any $\delta>0$ there exists $C_{\delta}>0$ and $\eta_\delta>0$ with the 
	following property. If $f\in
	L^2(\Sph^2)$ satisfies $\norma{f\sigma*f\sigma}_2\geq 
	\delta^2\mathbf{S}^2\norma{f}_2^2$ then 
	there exist a decomposition $f=g+h$ and a spherical cap $\sphcp\subseteq\Sph^2$ satisfying
	\begin{align}
	\label{eq:pointwise}
	&0\leq \ab{g},\ab{h}\leq \ab{f},\\
	\label{eq:disjointness}
	&g,h \text{ have disjoint supports},\\
	\label{eq:upper-bound-cap}
	&\ab{g(x)}\leq C_\delta\norma{f}_2\ab{\sphcp}^{-1/2}\one_\sphcp(x)\quad\text{for all } x,\\
	\label{eq:lower-bound-l2-norm}
	&\norma{g}_2\geq \eta_\delta\norma{f}_2.
	\end{align}
	Moreover \eqref{eq:upper-bound-cap} and \eqref{eq:lower-bound-l2-norm} hold with constants that 
	satisfy
	$C_\delta\asymp \delta^{-1/\gamma}$ and $\eta_\delta\asymp \delta^{1/\gamma}$, where
	$\gamma>0$ is a universal constant\footnote{The power dependence of $C_\delta$ and $\eta_\delta$ on $\delta$ can be found in the proof of the lemma in \cite{CS}*{pp. 277-278}}.
\end{lemma}

It will be our task in the next section to obtain an analogous result for the hyperboloid and for this it will be convenient to briefly discuss the construction of the function $g$ and the cap $\sphcp$ in the conclusion of the previous lemma. Letting $f\in L^2(\Sph^2)$ be a $\delta$-quasi-extremal, inequality \eqref{eq:cap_refinement_TS} implies that there is a constant $c_0\in(0,\infty)$, independent of $f$, such that
\[ \sup_{k,j}\Lambda_{k,j}(f)\geq 2c_0\delta^{1/\gamma}. \]
It follows that there exist $k$ and $j$ such that $\Lambda_{k,j}(f)\geq c_0\delta^{1/\la}$. Let $\sphcp:=\sphcp_k^j$. Then,
\begin{equation}\label{eq:lowerBoundCap}
\int_\sphcp \ab{f}\d\sigma\geq c_0\delta^{1/\gamma}\ab{\sphcp}^{1/2}\norma{f}_{L^2(\Sph^2)}.
\end{equation}
Let $R=(\frac{1}{2}c_0\delta^{1/\gamma}\ab{\sphcp}^{1/2}\norma{f}_{L^2(\Sph^2)})^{-1}$, $A=\{x\in\sphcp\colon\ab{f}\leq R \}$, $g=f\one_A$ and $h=f-f\one_A$. It is now a simple exercise to prove that $g,\,h$ and $\sphcp$ satisfy the conditions stated in the lemma.

\begin{remark}\label{rem:measurability}
Let us consider the following scenario: a measurable set $E\subseteq \R$ and a measurable function $F:E\times\Sph^2\to\mathbb{C}$ that satisfies $F\in L^2(E\times\Sph^2)$, $\norma{F_r\sigma\ast F_r\sigma}_{L^2(\R^3)}\geq \delta^2\mathbf{S}^2\norma{F_r}_{L^2(\Sph^2)}^2>0$ for all $r\in E$, where $F_r(x)=F(r,x)$, $(r,x)\in E\times\Sph^2$. Applying Lemma \ref{lem:quasiextremals-sphere-CS} to $F_r$ for each $r\in E$ generates caps $\sphcp_r\subseteq\Sph^2$ and functions $G_r$ and $H_r$, and in this way functions $G,H:E\times\Sph^2\to\mathbb{C}$, which a priori may not be measurable in the product space $E\times\Sph^2$. This can be overcome if we are careful with the choice of the caps as we now proceed to explain. For a collection of spherical caps $\{\sphcp_r\}_{r\in E}$ satisfying \eqref{eq:lowerBoundCap} with $\sphcp=\sphcp_r$ and $f=F_r$, for all $r\in E$, denote
\begin{align*}
\mathcal{G}_0&=\{(r,x)\colon r\in E,\, x\in\sphcp_r \},\\ \mathcal{G}_1&=\Bigl\{(r,x)\in\mathcal{G}_0\colon\ab{F_r(x)}\leq \big(\tfrac{1}{2}c_0\delta^{1/\gamma}\ab{\sphcp_r}^{1/2}\norma{F_r}_{L^2(\Sph^2)}\big)^{-1} \Bigr\}.
\end{align*}
Then, as explained following \eqref{eq:lowerBoundCap}, we can take $G=F\one_{\mathcal{G}_1}$ and $H=F-F\one_{\mathcal{G}_1}$. We need to argue that we can have $\mathcal G_0$ and $\mathcal{G}_1$ measurable, by choosing the collection $\{\sphcp_r\}_{r\in E}$ appropriately.
When $r\in E$, then $\sup_{k,j}\Lambda_{k,j}(F_r)\geq 2c(\delta)$, for some universal constant $c(\delta)$. The cap 
$\sphcp_r=\sphcp_k^j$ is to be chosen so that
$\Lambda_{k,j}(F_r)\geq c(\delta)$,
that is,
\[ \ab{\sphcp_r}^{-1/2}\int_{\sphcp_r}\ab{F_r}\d\sigma\geq
c(\delta)\norma{F_r}_{L^2(\Sph^2)}. \]

The set of caps $\{\sphcp_k^j\colon k,j\}$
in $\Sph^2$ is parametrized by indices $k$ and $j$ where $k\in\N$ and $j\in\{1,2,\dotsc,J_k\}$, for some $J_k<\infty$. Let $\mathcal{Z}=\{(k,j)\colon k\in\N,\, j\in\{1,\dotsc,J_k\} \}$ and
define the function $\Theta\colon E\times \mathcal{Z}\to \R$ by
\[ \Theta(r,k,j)= \ab{\sphcp_k^j}^{-1/2}\norma{F_r}_{L^2(\Sph^2)}^{-1}\int_{\sphcp_k^j}\ab{F_r}\d\sigma.\]

By Fubini's theorem, for each fixed 
$(k,j)\in E\times \mathcal{Z}$, 
$\Theta(\cdot,k,j)$ is a measurable function. By assumption, for each $r\in E$, 
$\sup_{k,j}\Theta(r,k,j)\geq 
2c(\delta)$. We want to find a 
measurable function $\tau:E\to \mathcal{Z}$ such that 
\[\Theta(r,\tau(r))\geq \sup_{k,j}\Theta(r,k,j)-c(\delta)\geq c(\delta),\]
a so called $c(\delta)$-maximizer.
That this is possible 
is a consequence of measurable selection theorems, see for instance \cite{Ri}*{Theorem 4.1}. 

For such a measurable selection function $\tau$ write $\tau(r)=(k(r),j(r))\in\mathcal Z$, then the function $E\to \Sph^2$, $r\mapsto y_{k(r)}^{j(r)}$, is measurable and we can write $\mathcal G_0=\{(x,r):r\in E, \ab{x-y_{k(r)}^{j(r)}}\leq 2^{-k(r)+1} \}$. We can 
therefore 
assume that the sets $\mathcal G_0$ and $\mathcal G_1$ are measurable sets in $E\times\Sph^2$, so that $G$ and $H$ are measurable functions. In this way, we have the following lemma.

\begin{lemma}\label{lem:measurable_Sel}
Let $E\subseteq \R$ be a measurable set and $F:E\times\Sph^2\to\mathbb{C}$ be a measurable function satisfying $F\in L^2(E\times\Sph^2)$, $\norma{F_r\sigma\ast F_r\sigma}_{L^2(\R^3)}\geq \delta^2\mathbf{S}^2\norma{F_r}_{L^2(\Sph^2)}^2>0$ for all $r\in E$, where $F_r(x)=F(r,x)$, $(r,x)\in E\times\Sph^2$. Then, there are spherical caps $\{\sphcp_r\}_{r\in E}$ and measurable functions $G,H$ satisfying: $F=G+H$, $G$ and $H$ have disjoint supports, $0\leq \ab{G},\ab{H}\leq \ab{F}$, and for all $r\in E$:
\[ \ab{G_r(x)}\leq C_\delta \norma{F_r}_2\ab{\sphcp_r}^{-1/2}\one_{\sphcp_r}(x), \, x\in\Sph^2
\text{ and } \norma{G_r}_2\geq \eta_\delta\norma{F_r}_2. \]
\end{lemma}
\end{remark}
 
We now prove a slight improvement of Lemma \ref{lem:quasiextremals-sphere-CS} that adds one more restriction to the function $g$. It tells us 
that we can replace a $\delta$-quasi-extremal for the sphere for a better 
controlled one at the 
expense of powers of $\delta$.

\begin{lemma}
	\label{lem:quasiextremals-sphere}
	For any $\delta>0$ there exists $C_{\delta}>0$, $\eta_\delta>0$ and $\la_\delta>0$ with the 
	following property. If $f\in
	L^2(\Sph^2)$ satisfies $\norma{f\sigma*f\sigma}_2\geq \delta^2\mathbf{S}^2\norma{f}_2^2$ then 
	there 
	exist a decomposition $f=g+h$
	and a spherical cap $\sphcp$ satisfying \eqref{eq:pointwise}, \eqref{eq:disjointness}, \eqref{eq:upper-bound-cap}, 
	\ref{eq:lower-bound-l2-norm} 
	and 
	\begin{equation}
	\label{eq:lower-bound-functional}
	\norma{g\sigma*g\sigma}_2\geq \la_\delta\mathbf{S}^2\norma{f}_2^2.
	\end{equation}
	Moreover \eqref{eq:upper-bound-cap}, \eqref{eq:lower-bound-l2-norm} and 
	\eqref{eq:lower-bound-functional} hold with constants that satisfy
	$C_\delta\asymp \delta^{-1/\gamma},\,\eta_\delta\asymp \delta^{1+1/\gamma}$ and 
	$\la_\delta\asymp 
	\delta^{6+4/\gamma}$, where
	$\gamma>0$ is a universal constant.
\end{lemma}

\begin{remark}\label{rem:lower_L1}
	It is not difficult to see (e.g. \cite{RQ1}*{Lemma 6.2}) that for a function $g$ satisfying \eqref{eq:upper-bound-cap} and 
	\eqref{eq:lower-bound-l2-norm}
	there is a lower bound for the $L^1$ norm of the form
	\begin{equation}
	\label{eq:lower-bound-l1-norm-sphere}
	\int_{\sphcp}\ab{g}\d\sigma\geq \frac{\eta_\delta^2}{C_\delta}\norma{f}_2\ab{\sphcp}^{1/2}.
	\end{equation}
	Note that the sharp estimate \eqref{eq:convolution_ts} for $\Sph^2$ implies that for $g$ 
	satisfying 
	\eqref{eq:lower-bound-functional} 
	we have
	\[\mathbf{S}\norma{g}_{2}\geq \norma{g\sigma*g\sigma}_2^{1/2}\geq 
	\la_\delta^{1/2}\mathbf{S}\norma{f}_2,\]
	so that
	\begin{equation}\label{eq:lower-bound:9:15}
	\norma{g}_{L^2(\Sph^2)}\geq\la_\delta^{1/2}\norma{f}_{L^2(\Sph^2)}.
	\end{equation}
\end{remark}

\begin{proof}[Proof of Lemma \ref{lem:quasiextremals-sphere}]
	Take $C_\delta$ and $\eta_\delta$ as in the 
	conclusion of Lemma \ref{lem:quasiextremals-sphere-CS}. We claim that the 
	lemma at hand holds with respective constants $C_\delta$, $\delta\eta_\delta/\sqrt{2}$ and
	$\la_\delta=(\delta^3\eta_\delta^2/8)^2$. To see this we employ a 
	decomposition algorithm, 
	reminiscent of that in \cite{CS}*{Section 8, step 6A}, 
	defined in the following inductive way. 
	
	Let $G_0=f$ and $f_0=0$ 
	and suppose that for $N\geq 0$ we have defined $G_{N}$ and 
	$f_k$, for $0\leq k\leq N$, satisfying:
	\begin{align}
	\label{eq:suma}
	&f=G_N+f_0+\dotsb+f_{N},\\
	\label{eq:disjoint-support}
	&\supp(G_{N}),\supp(f_0),\dots,\supp(f_{N})\text{ are pairwise disjoint,}\\
	\label{eq:lower-bound-convolution}
	&\norma{G_N\sigma*G_N\sigma}_2\geq \frac{1}{2}\delta^2\mathbf{S}^2\norma{f}_2^2.
	\end{align}
	The previous conditions are satisfied if $N=0$. We now define the inductive step of 
	the 
	algorithm. If \eqref{eq:suma}, \eqref{eq:disjoint-support} and 
	\eqref{eq:lower-bound-convolution} hold for $N$ we define $G_{N+1}$ and $f_{N+1}$ in the
	following way. 
	
	Given that $\norma{G_N\sigma*G_N\sigma}_2\geq \frac{1}{2}\delta^2\mathbf{S}^2\norma{f}_2^2\geq 
	\frac{1}{2}\delta^2\mathbf{S}^2\norma{G_N}_2^2$ we can
	apply 
	Lemma \ref{lem:quasiextremals-sphere-CS} to $G_N$ to
	obtain a decomposition $G_N=g_N+h_N$ and a cap $\sphcp_{N}$. Define $G_{N+1}=h_N$ and $f_{N+1}=g_N$. 
	The 
	functions $G_{N+1}$ and $f_{N+1}$ therefore have 
	disjoint supports and satisfy
	\begin{gather}
	\label{eq:cond1_f_N_plus_1}
	\ab{f_{N+1}(x)}\leq C_\delta\norma{G_N}_2\ab{\sphcp_{N}}^{-1/2}\one_{\sphcp_{N}}(x)\leq 
	C_\delta\norma{f}_2\ab{\sphcp_{N}}^{-1/2}\one_{\sphcp_{N}}(x)\;\text{ for all } x,\\
	\label{eq:cond2_f_N_plus_1}
	\norma{f_{N+1}}_2\geq 
	\eta_\delta\norma{G_N}_2\geq \frac{1}{\sqrt{2}}\eta_\delta\delta\norma{f}_2,
	\end{gather}
	where the second inequality in \eqref{eq:cond2_f_N_plus_1} follows as in \eqref{eq:lower-bound:9:15}.

	The algorithm terminates at $N\geq 1$ when either $\norma{f_{N}\sigma*f_{N}\sigma}_2\geq 
	\la_\delta\mathbf{S}^2\norma{f}_2^2$ or $\norma{G_N\sigma*G_N\sigma}_2< 
	\frac{1}{2}\delta^2\mathbf{S}^2\norma{f}_2^2$. In the former case we say the algorithm stops in 
	a 
	win and  set $g=f_{N}$, $h=G_N+f_0+\dotsb+f_{N-1}$, $\sphcp=\sphcp_{N}$ and the Lemma is proved.
	
	Let $N_\delta:=\lceil2\eta_\delta^{-2}\delta^{-2}\rceil$. We 
	claim 
	that the 
	algorithm stops in a win for some $N\leq N_\delta$. We first show that the algorithm can not 
	run 
	for more than $N_\delta$ steps, otherwise, using \eqref{eq:cond2_f_N_plus_1} we have
	\[\norma{f}_2\geq \biggl(\suma{k=1}{N_\delta+1}\norma{f_k}_2^2\biggr)^{1/2}\geq 
	\frac{1}{\sqrt{2}}(N_\delta+1)^{1/2}\eta_\delta\delta\norma{f}_2>\norma{f}_2,\]
	which is impossible.
	
	Second, we show that if the algorithm has not stopped in a win during the first $N$ 
	steps for some $N\leq 
	2N_\delta$, then we can perform the step $N+1$. More precisely, if 
	$\norma{f_k\sigma*f_k\sigma}_2<\la_\delta\mathbf{S}^2\norma{f}_2^2$ for all 
	$1\leq k\leq N$, for some $N\leq 2N_\delta$, then $\norma{G_N\sigma*G_N\sigma}_2\geq
	\frac{1}{2}\delta^2\mathbf{S}^2\norma{f}_2^2$. Indeed, using Plancherel's theorem and then the 
	triangle inequality we obtain
	\begin{align*}
	\norma{G_{N}\sigma*G_{N}\sigma}_2^{1/2}&\geq
	\norma{f\sigma*f\sigma}_2^{1/2}-\suma{k=1}{N}\norma{f_k\sigma*f_k\sigma}_2^{1/2}
	\geq \delta\mathbf{S}\norma{f}_2-N\la_\delta^{1/2}\mathbf{S}\norma{f}_2\\
	&\geq (\delta-2N_\delta\la_\delta^{1/2})\mathbf{S}\norma{f}_2\\
	&\geq \frac{1}{2}\delta\mathbf{S}\norma{f}_2.
	\end{align*}
	
	If follows that the algorithm stops in a win for some $N\leq N_\delta$. This finishes the proof.
\end{proof}

The next topic we review is that of "weak interaction between distant caps". For spherical caps $\sphcp,\,\sphcp'\subseteq\Sph^2$ there is a notion of distance. Let 
$(y,a),\,(y',a')\in\Sph^2\times(0,\infty)$ denote 
the centers and radii of the spherical caps $\sphcp,\,\sphcp'$,
\[ \sphcp=\{x\in\Sph^2\colon \ab{x-y}\leq a \},\quad\sphcp'=\{x\in\Sph^2\colon \ab{x-y'}\leq a' 
\}.\]
The distance between $\sphcp$ and $\sphcp'$ is defined by the expression
\begin{equation}\label{eq:def_distance}
\varrho(\sphcp,\sphcp')= \min(\operatorname{d}(\sphcp,\sphcp'),\operatorname{d}(\sphcp,-\sphcp')),
\end{equation}
where, as in \cite{OeS14}, we can take $\operatorname{d}$ to be the hyperbolic distance between $(y,a)$ and $(y',a')$ in the upper half space model, that is \footnote{We point out that for the two lemmas that follow we don't need $\operatorname{d}$ to be a distance. It would be perfectly fine to consider instead the expression
\[ \frac{(a-a')^2}{aa'}+\frac{\ab{y-y'}^2}{a^2}+\frac{\ab{y-y'}^2}{(a')^2}, \]
so that caps are far apart if either $a/a'$ or $a'/a$ is large or the distance from $y$ to $y'$ is much larger than either $a$ or $a'$.}
\[ \operatorname{d}(\sphcp,\sphcp')=\operatorname{arccosh}\Bigl(1+\frac{(a-a')^2+\ab{y-y'}^2}{2aa'}\Bigr). \]
The following lemma quantifies the notion of weak interaction between distant caps.

\begin{lemma}[\cite{CS}*{Lemma 7.6}]\label{lem:weak-interaction-caps}
	For any $\eps>0$ there exists $\rho<\infty$ such that 
	\[ 
	\norma{\one_\sphcp\sigma\ast\one_{\sphcp'}\sigma}_{L^2(\R^3)}<\eps\ab{\sphcp}^{1/2}\ab{\sphcp'}^{1/2},
	\quad\text{whenever}\quad\varrho(\sphcp,\sphcp')>\rho. \]
\end{lemma}

An inspection of the proof of the previous statement in \cite{CS}*{Lemma 7.6} shows that an analog 
result 
holds if we have caps $\sphcp\subseteq\Sph_r^2$ 
and 
$\sphcp'\subseteq\Sph_t^2$, with $r,\,t\in[1,2]$, that is, denoting $\frac{1}{r}\sphcp$ the rescale of 
$\sphcp$ to $\Sph^2$,
\[ \tfrac{1}{r}\sphcp=\{x\in\R^3\colon rx\in\sphcp \}, \]
we have the following lemma.

\begin{lemma}\label{lem:weak-interaction-mixed-radii}
	Let $r,t\in[1,2]$, $\sphcp\subseteq\Sph_r^2$ 
	and 
	$\sphcp'\subseteq\Sph_t^2$. Then for any $\eps>0$ there exists $\rho<\infty$ 
	such that 
	$\norma{\one_\sphcp\sigma_r\ast\one_{\sphcp'}\sigma_t}_{L^2(\R^3)}<\eps\ab{\sphcp}^{1/2}\ab{\sphcp'}^{1/2}$, whenever $\varrho(\frac{1}{r}\sphcp,\frac{1}{t}\sphcp')>\rho$.
\end{lemma}

\section{Lifting to the hyperboloid the inequality for the sphere}\label{sec:lifting_the_sphere}

The aim of this section is to use the Tomas--Stein inequality for the sphere $\Sph^2$ to obtain qualitative 
properties of $\delta$-quasi-extremals for the hyperboloid.
The connection here between the hyperboloid and the sphere is that the latter corresponds to horizontal traces of the former. This connection between the adjoint Fourier restriction operator on a hypersurface and on its traces appears, for instance, in the work of Nicola \cite{Ni}. An alternative approach to the methods in this section can be developed using refined bilinear estimates, but we choose to give a different argument. The main result of this section is the following lemma.

\begin{lemma}
	\label{lem:nearly-extremals-hyp}
	Let $0\leq s\leq \frac{1}{2}$. For any $\delta>0$ there exists $C_{\delta}>0$, $\eta_\delta>0$ and $\nu_\delta>0$ with the 
	following property. If $f(x,t)\in
	L^2(\hyp_s)$ supported where $1\leq \ab{x}\leq 2$ satisfies $\norma{f\mu_s\ast 
		f\mu_s}_{L^2(\R^4)}\geq \delta^2{\mathbf{H}}_4^2\norma{f}_{L^2}^2$ then there exist a decomposition $f=g+h$, a spherical cap $\sphcp\subseteq\Sph^2$
	and a cap $\cp=[1,2]\times\sphcp\subset \hyp_s$ satisfying
	\begin{align}
	&0\leq \ab{g},\ab{h}\leq \ab{f},\label{eq:hyp_cap_refinement_1}\\
	&g,h \text{ have disjoint supports},\label{eq:hyp_cap_refinement_2}\\
	&\supp(g)\subseteq \cp,\label{eq:hyp_cap_refinement_3}\\
	&\ab{g(x,t)}\leq C_\delta\norma{f}_{L^2}\mu_s({\cp})^{-1/2}\one_\cp(x,t)\;\text{ for all }(x,t),\label{eq:hyp_cap_refinement_4}\\
	&\norma{g}_{L^2}\geq \eta_\delta\norma{f}_{L^2},\label{eq:hyp_cap_refinement_5}\\
	&\norma{g}_{L^1}\geq \nu_\delta\mu_s({\cp})^{1/2}\norma{f}_{L^2}.\label{eq:hyp_cap_refinement_6}
	\end{align}
	The constants $C_\delta,\,\eta_\delta$ and $\nu_\delta$ are uniform in $s\leq \frac{1}{2}$.
\end{lemma}

\begin{remark}\label{rem:hyp_equiv_Hyp}
	The previous lemma is equivalent to the analog result for $\Hyp_s$. Indeed, that the result for $\Hyp_s$ implies a similar one for $\hyp_s$ is immediate. On the other direction, if $f\in L^2(\Hyp_s)$ is a $\delta$-quasi-extremal for \eqref{sharp_L4_double_hyp}, that is
	\[\norma{\overline{T}_sf}_{L^4(\R^4)}^4=(2\pi)^{4}\norma{f\bar{\mu}_s\ast f\bar{\mu}_s}_{L^2(\R^4)}^2\geq (2\pi)^{4}\delta^4\overline{\mathbf{H}}_4^4\norma{f}_{L^2(\Hyp_s)}^4,\]
	then, writing $f=f_++f_-$ so that $\overline{T}_sf=T_sf_++T_sf_-(\cdot,-\cdot )$ and $\norma{f}_{L^2(\Hyp_s)}^2=\norma{f_+}_{L^2(\hyp_s)}^2+\norma{f_-}_{L^2(\hyp_s)}^2$ we obtain that 
	\[\norma{f_\epsilon{\mu}_s\ast f_\epsilon{\mu}_s}_{L^2(\R^4)}^2\geq 2^{-4}\delta^4\overline{\mathbf{H}}_4^4\norma{f_\epsilon}_{L^2(\hyp_s)}^4\]
	for $\epsilon=+$ or for $\epsilon=-$, 
	so that if both $\norma{f_+}_{L^2(\hyp_s)}^2\geq \delta^2\norma{f}_{L^2(\Hyp_s)}^2$ and $\norma{f_-}_{L^2(\hyp_s)}^2\geq \delta^2\norma{f}_{L^2(\Hyp_s)}^2$, then we obtain the conclusions in Lemma \ref{lem:nearly-extremals-hyp} for $f$ from the ones for $f_+$ or $f_-$, as it corresponds. On the other hand, if say $\norma{f_-}_{L^2(\hyp_s)}^2< \delta^2\norma{f}_{L^2(\Hyp_s)}^2$, then $\norma{f_+}_{L^2(\hyp_s)}^2\geq (1-\delta^2)\norma{f}_{L^2(\Hyp_s)}^2$ and 
	\[ \norma{Tf_+}_{L^4}\geq \norma{\overline{T}f}_{L^4}-\norma{Tf_-}_{L^4}\geq 2\pi\delta(\overline{\mathbf{H}}_4-\mathbf{H}_4)\norma{f}_{L^2(\Hyp_s)} \geq c\delta\mathbf{H}_4\norma{f_+}_{L^2(\hyp_s)},\]
	so that Lemma \ref{lem:nearly-extremals-hyp} applied to $f_+$ yields the result for $f$.
	
	The support condition $1\leq \ab{x}\leq 2$ can be changed to $a\leq \ab{x}\leq b$ for any $a\geq s$ and 
	$b<\infty$, understanding that the implicit constants may depend on $a,b$. We can alternatively state the previous lemma for $f\in L^2(\hyp)$ supported where $2^{N}\leq \ab{x}\leq 2^{N+1}$, $N\geq 1$, the implicit constants independent of $N$, as can be easily checked by the use of scaling.
\end{remark}

Recall that we write $\psi_s(r)=\sqrt{r^2-s^2}\one_{\{r\geq s\}}$ and $\phi_s(t)=\psi_s^{-1}(t)=\sqrt{t^2+s^2}\one_{\{t\geq 0\}}$ and for $f\in 
\Sh(\R^3)$ and $r>0$ we denote by
$f\sigma_r$ the measure supported on $\Sph_r^2:=\{y\in\R^3:\ab{y}=r\}$ given by
\[\langle f\sigma_r,\varphi\rangle=\int_{\Sph^2}f(ry)\varphi(ry)r\d\sigma(y).\]
We denote $f_r$ the function $x\mapsto f(rx)$, which we consider as a function from $\Sph^2$ to $\mathbb{C}$.

In the next lemma we show that we can write the double convolution of functions on the hyperboloid $\hyp_s$ as an integral of convolutions 
of sliced spheres.

\begin{lemma}
	\label{lem:sliced-convolution}
	Let $s\geq 0$. For $f,g\in L^2(\hyp_s)$ we have the representation formula
	\begin{equation}
	\label{eq:representation-sliced-convolution}
	\bigl(f\mu_s*g\mu_s\bigr)(x,t)=\int_0^t \bigl(f\sigma_{\phi_s(t')}*g\sigma_{\phi_s(t-t')}\bigr)(x)\d t',
	\end{equation}
for a.e. $(x,t)\in\R^3\times\R_+$.
\end{lemma}

\begin{proof}
	Let $\varphi\in C_c^\infty(\R^4)$. Using spherical coordinates we have
	\begin{align*}
	&\langle f\mu_s*g\mu_s,\vphi\rangle=\int_{\ab{x},\ab{y}\geq s}
	\vphi(x+y,\psi_s(x)+\psi_s(y))f(x)g(y)\frac{\d x\d y}{\sqrt{\ab{x}^2-s^2}\sqrt{\ab{y}^2-s^2}}\\
	&=\int_s^\infty\int_s^\infty\int_{\Sph^2}\int_{\Sph^2}\vphi(r\omega+r'\omega',\psi_s(r)+\psi_s(r'))
	f(r\omega)g(r'\omega')\frac{r^2r'^2\d\omega \d\omega' \d r \d 
	r'}{\sqrt{r^2-s^2}\sqrt{r'^2-s^2}}.
	\end{align*}
	We change variables $(r,r')\mapsto (u,u')=(\psi_s(r),\psi_s(r'))=(\sqrt{r^2-s^2},\sqrt{r'^2-s^2})$ and 
	obtain
	\begin{multline*}
	\langle
	f\mu_s*g\mu_s,\vphi\rangle=\int_0^\infty\int_0^\infty\int_{\Sph^2}\int_{\Sph^2}\vphi(\phi_s(u)\omega+\phi_s(u')\omega',
	u+u')\\
	\cdot f(\phi_s(u)\omega)g(\phi_s(u')\omega')\phi_s(u)\phi_s(u')\d\omega \d\omega' \d u \d u'.
	\end{multline*}
	We change variables $(u,u')\mapsto(t,t')=(u+u',u)$ and obtain
	\begin{align*}
	\langle
	f\mu_s*g\mu_s,\vphi\rangle&=\int_0^\infty\int_0^t\int_{\Sph^2}\int_{\Sph^2}\vphi(\phi_s(t')\omega+\phi_s(t-t')\omega',
	t)\\
	&\qquad\cdot f(\phi_s(t')\omega)g(\phi_s(t-t')\omega')\phi_s(t')\phi_s(t-t')\d\omega \d\omega' \d t'
	\d 
	t\\
	&=\int_0^\infty \int_0^t \Bigl(\int_{\R^3} 
	\vphi(x,t)\bigl(f\sigma_{\phi_s(t')}*g\sigma_{\phi_s(t-t')}\bigr)(x)dx\Bigr) \d t' \d t\\
	&=\Bigl\langle \int_0^t\bigl(f\sigma_{\phi_s(t')}*g\sigma_{\phi_s(t-t')}\bigr)(x)\d t', \vphi\Bigr\rangle,
	\end{align*}
	where we used Fubini's Theorem and that for any $r,r'>0$,
	\begin{align*}
	\langle f\sigma_r*g\sigma_{r'},\vphi(\cdot,t)\rangle&=\int_{\R^3} 
	\vphi(x,t)\bigl(f\sigma_{r}*g\sigma_{r'}\bigr)(x)\d x\\
	&=\int_{\Sph^2_r\times 
		\Sph^2_s}\vphi(x+x',t)f(x)g(x')\d\sigma_r(x)\d\sigma_{r'}(x')\\
	&=\int_{\Sph^2\times 
		\Sph^2}\vphi(r\omega+r'\omega',t)f(r\omega)g(r'\omega')rr' \d\sigma(\omega)\d\sigma(\omega').
	\end{align*}
\end{proof}

Next, we record a formula for the $L^p(\hyp_s)$ norm in terms of the $L^p$ norm of the slices.

\begin{lemma}\label{lem:sliced-lp-norm}
	Let $f\in L^p(\mathcal H^3_s)$. Then
	\begin{equation}\label{eq:lp-norm-spherical}
	\norma{f}_{L^p(\hyp_s)}^p=\int_0^\infty\norma{f_{\phi_s(t)}}_{L^p(\Sph^2)}^p\phi_s(t)\d t.
	\end{equation}
\end{lemma}

\begin{proof} 
	Using spherical coordinates we have
	\begin{align*}
	\norma{f}_{L^p(\hyp_s)}^p&=\int_s^\infty\int_{\Sph^2}\ab{f(r\omega)}^p\frac{r^2}{\sqrt{r^2-s^2}}\d\omega
	\d r=\int_0^\infty\int_{\Sph^2}\ab{f(\phi_s(t)\omega)}^p\phi_s(t)\d\omega \d t\\
	&=\int_0^\infty\norma{f_{\phi_s(t)}}_{L^p(\Sph^2)}^p\phi_s(t)\d t.
	\end{align*}
\end{proof}

We now
analyze the dependence of $\norma{f\sigma_{r}*g\sigma_{r'}}_{L^2(\R^3)}$ in $(r,r')$. We start with
the
scaling property of $\widehat{f\sigma_r}$ as a function of $r$. We have 
\[ (\widehat{f\sigma_r})(x)=\int_{\Sph_r^2}e^{-ix\cdot y}f(y)\d\sigma_r(y)=\int_{\Sph^2}e^{-irx\cdot y}f(ry)r\d\sigma(y)=r(\widehat{f_r\sigma})(rx). \]
Thus
\[\norma{\widehat{f\sigma_r}}_{L^4(\R^3)}=r^{1/4}\norma{\widehat{f_r\sigma}}_{L^4(\R^3)}\leq
(2\pi)^{3/4}r^{1/4}\mathbf{S}\norma{f_r}_{L^2(\Sph^2)}.\]
Then, the Cauchy--Schwarz inequality implies that for any $r,r'>0$
\[\norma{\widehat{f\sigma_r}\,\widehat{g\sigma_{r'}}}_{L^2(\R^3)}\leq
\norma{\widehat{f\sigma_r}}_{L^4}\norma{\widehat{g\sigma_{r'}}}_{L^4}\leq
(2\pi)^{3/2}\mathbf{S}^2(rr')^{1/4}\norma{f_r}_{L^2(\Sph^2)}\norma{g_{r'}}_{L^2(\Sph^2)},\]
so that
\begin{equation}\label{eq:scaling_convolution_L2}
\norma{f\sigma_r\ast g\sigma_{r'}}_{L^2(\R^3)}\leq
\mathbf{S}^2(rr')^{1/4}\norma{f_r}_{L^2(\Sph^2)}\norma{g_{r'}}_{L^2(\Sph^2)}
\end{equation}
and in particular, when $r=r'$ we obtain
\begin{equation}\label{eq:scaling_convolution_same_radius}
\norma{f\sigma_r\ast g\sigma_{r}}_{L^2(\R^3)}=r^{1/2}\norma{f_r\sigma\ast g_r\sigma}_{L^2(\R^3)}\leq
\mathbf{S}^2r^{1/2}\norma{f_r}_{L^2(\Sph^2)}\norma{g_r}_{L^2(\Sph^2)}.
\end{equation}

\begin{define}
A quasi-cap of $\mathcal{H}^3_s$ 
is a measurable set 
$\mathcal G\subseteq \mathcal{H}^3_s$ 
for which there exist $E\subseteq \R$ and spherical caps $\sphcp_{t}\subseteq 
\Sph^2$, for 
$t\in E$, such that 
\begin{equation}\label{eq:Def_quasi_cap}
\mathcal G=\{(x,t)\in\R^4:\,t\in E,\,x\in\phi_s(t)\sphcp_{t}\}.
\end{equation}
\end{define}

We note that a cap is also a quasi-cap; the difference in a generic quasi-cap is that the spherical caps 
may not be the same as in the case of a cap, and the set $E$ may not be an interval.

In our main result of the section, Lemma \ref{lem:nearly-extremals-hyp}, we want to obtain an analog of Lemma \ref{lem:quasiextremals-sphere-CS} for a compact 
subset of the hyperboloid. The idea is to use the cap Lemma \ref{lem:quasiextremals-sphere-CS} for the 
sphere on horizontal slices of the hyperboloid via \eqref{eq:representation-sliced-convolution} in a measurable way (recall Remark \ref{rem:measurability}), and 
show that there are 
enough aligned sliced caps of similar size to obtain a cap for the hyperboloid. We do it for the 
upper sheet as the full one-sheeted hyperboloid follows from this as already noted in Remark \ref{rem:hyp_equiv_Hyp}. The proof of Lemma \ref{lem:nearly-extremals-hyp} is accomplished in the following way. First, we 
show that on a large subset of $t$'s in $[\psi_s(1),\psi_s(2)]$ we can apply Lemma 
\ref{lem:quasiextremals-sphere} to the function $x\in\Sph^2\mapsto f(\phi_s(t)x)$ in a measurable way. This will allow us 
to prove a version of Lemma \ref{lem:nearly-extremals-hyp} where instead of a cap we have a quasi-cap. 
Next, we show that a subset of the quasi-cap of large relative measure is comparable to a cap and satisfies the 
requirements of Lemma \ref{lem:nearly-extremals-hyp}, which then are shown to be satisfied by the cap itself. To prove this last point, we will make use of the 
quantitative version of the statement that "distant spherical caps interact weakly" as stated in Lemmas \ref{lem:weak-interaction-caps} and \ref{lem:weak-interaction-mixed-radii}.

\begin{proof}[Proof of Lemma \ref{lem:nearly-extremals-hyp}]
	In what follows, $c(\delta)$ denotes a constant that depends only on $\delta$ and is 
	allowed to change from line to line\footnote{Reviewing the argument one can see that such constants can be taken to depend only on powers, positive and negative, of $\delta$.}. Recall from Remark \ref{rem:lower_L1} that \eqref{eq:hyp_cap_refinement_6} can be obtained from \eqref{eq:hyp_cap_refinement_4} and \eqref{eq:hyp_cap_refinement_5} with $\nu_\delta=\eta_\delta^2/C_\delta$.
	
	We first argue that we can assume that the support of $f(\cdot,t)$ does not contain antipodal points for each $t\in[\psi_s(1),\psi_s(2)]$. We can cover $\Sph^2$ as the union of finitely many spherical caps 
	$\{\sphcp_k\}_{k=1,\dotsc,\kappa}$ 
	each of 
	radius $\frac{1}{4}$, whose 
	centers form a maximally $\frac{1}{4}$-separated set on $\Sph^2$, and induce a decomposition 
	of $\mathcal 
	H_s^3$ as 
	the union of the caps $\{[s,\infty)\times \sphcp_k \}_{k=1,\dotsc,\kappa}$. By the triangle 
	inequality 
	we can therefore assume that $f$ is supported on the cap $[s,\infty)\times \sphcp_k$, for some 
	$k$, at the expense of changing $\delta$ by $\delta/\kappa$. 
	The 
	reason for doing this is to ensure that there are no nearly antipodal spherical caps later on. 
	
	Let us start by noting that for $(x,t)$ in the support of $f$ and $s\in[0,\frac{1}{2}]$ we have $\ab{x}\in[1,2]$ and $t=\psi_s(x)\in[\psi_s(1),\psi_s(2)]=[\sqrt{1-s^2},\sqrt{4-s^2}]\subseteq[\frac{\sqrt{3}}{2},2]$, and that from Lemma \ref{lem:sliced-lp-norm}
	\[ \int_{\psi_s(1)}^{\psi_s(2)}\norma{f_{\phi_s(t)}}_{L^2(\Sph^2)}^2\d t\leq \norma{f}_{L^2(\hyp_s)}^2\leq 2 \int_{\psi_s(1)}^{\psi_s(2)}\norma{f_{\phi_s(t)}}_{L^2(\Sph^2)}^2\d t. \]
	On the other hand $(f\mu_s*f\mu_s)(x,t)$ is supported 
	where $2\psi_s(1)\leq t\leq	2\psi_s(2)$. From Lemma \ref{lem:sliced-convolution} for a.e. $(x,t)\in\R^4$ we have
	\begin{equation}\label{eq:equality_convolution}
	f\mu_s*f\mu_s(x,t)=\int_{\psi_s(1)}^{\psi_s(2)} (f\sigma_{\phi_s(t')}*f\sigma_{\phi_s(t-t')})(x)\d t',
	\end{equation}
	(recall that $\phi_s(\tau)=0$ for $\tau<0$). Let
	\begin{equation*}
	E_{\gamma}=\Biggl\{t\in[{\psi_s(1)},{\psi_s(2)}]\colon
	\begin{aligned}
	\norma{f\sigma_{\phi_s(t)}*f\sigma_{\phi_s(t)}}_{L^2(\R^3)}&\geq
	\gamma^2\delta^2\mathbf{H}_{4}^2\mathbf{S}^2\norma{f_{\phi_s(t)}}_2^2,\\
	\norma{f_{\phi_s(t)}}_2&\geq
	\gamma\delta\mathbf{H}_{4}\norma{f}_2
	\end{aligned}
	\Biggr\}
	\end{equation*}
	and
	\begin{equation*}
	E_{\gamma,\la}=\Biggl\{t\in[{\psi_s(1)},{\psi_s(2)}]\colon\,
	{\begin{aligned}
	\norma{f\sigma_{\phi_s(t)}*f\sigma_{\phi_s(t)}}_{L^2(\R^3)}&\geq
	\gamma^2\delta^2\mathbf{H}_{4}^2\mathbf{S}^2\norma{f_{\phi_s(t)}}_2^2,\\
	\la\delta\mathbf{H}_{4}\norma{f}_2\geq \norma{f_{\phi_s(t)}}_2&\geq
	\gamma\delta\mathbf{H}_{4}\norma{f}_2
	\end{aligned}}
	\Biggr\}.
	\end{equation*}
	Here, $\norma{f_{\phi_s(t)}}_2=\norma{f({\phi_s(t)}\,\cdot,t)}_{L^2(\mathbb S^3)}$, while $\norma{f}_2=\norma{f}_{L^2(\hyp_s)}$. We claim that $\ab{E_{\gamma}}\geq c(\delta)$ and $\ab{E_{\gamma,\la}}\geq c(\delta)$ if 
	$\gamma$ 
	and $\la$ are chosen small and large enough depending on $\delta$, respectively. 
	Let us first analyze $\ab{E_\gamma}$. From \eqref{eq:equality_convolution}, using Fubini's theorem and Minkowski's integral inequality we have
	\begin{align*}
	\delta^2\mathbf{H}_4^2\norma{f}_2^2&\leq \Norma{\int_{\psi_s(1)}^{\psi_s(2)} (f\sigma_{\phi_s(t')}*f\sigma_{\phi_s(t-t')})(x)\d t'}_{L^2_{t,x}}\\
	&\leq
	\Norma{\int_{\psi_s(1)}^{\psi_s(2)}\norma{f\sigma_{\phi_s(t')}*f\sigma_{\phi_s(t-t')}}_{L^2_x}\one_{E_{\gamma}^\complement}(t')\d 
		t'}_{L^2_t}\\
	&\quad+\Norma{\int_{\psi_s(1)}^{\psi_s(2)}
		(f\sigma_{\phi_s(t')}*f\sigma_{\phi_s(t-t')})(x)\one_{E_{\gamma}}(t')\d t'}_{L^2_{x,t}}.
	\end{align*}
	Plancherel's theorem and the Cauchy--Schwarz 
	inequality 
	give 
	\[ \norma{f\sigma_{\phi_s(t')}*f\sigma_{\phi_s(t-t')}}_{L^2_x}\leq
	\norma{f\sigma_{\phi_s(t')}*f\sigma_{\phi_s(t')}}_{L^2_x}^{1/2}\norma{f\sigma_{\phi_s(t-t')}*f\sigma_{\phi_s(t-t')}}_{L^2_x}^{1/2}, \]
	 so that using the sharp estimate for $\norma{f\sigma_{\phi_s(t-t')}*f\sigma_{\phi_s(t-t')}}_{L^2_x}$ as in \eqref{eq:scaling_convolution_same_radius}, recalling that $\phi_s(t'),\,\phi_s(t-t')\in[1,2]$, we obtain 
	\begin{align*}
	\Norma{\int_{\psi_s(1)}^{\psi_s(2)}\norma{f\sigma_{\phi_s(t')}*f\sigma_{\phi_s(t-t')}}_{L^2_x}&\one_{E_\gamma^\complement}(t')\d t'}_{L^2_t}\\
	&\leq
	2\gamma\delta\mathbf{H}_{4}\mathbf{S}^2
	\Norma{\int_{\psi_s(1)}^{\psi_s(2)}\norma{f_{\phi_s(t')}}_2\norma{f_{\phi_s(t-t')}}_2\d t'}_{L^2_t}\\
	&\qquad+2\gamma\delta\mathbf{H}_{4}\mathbf{S}^2\norma{f}_2\Norma{\int_{\psi_s(1)}^{\psi_s(2)}\norma{f_{\phi_s(t-t')}}_2\d t'}_{
		L^2_t}\\
	&\leq 8\gamma\delta\mathbf{H}_{4}\mathbf{S}^2\norma{f}_2^2.
	\end{align*}
	Therefore, choosing $\gamma=\delta\mathbf{H}_{4}/(16\mathbf{S}^2)$ we obtain
	\[\Norma{\int_{\psi_s(1)}^{\psi_s(2)} (f\sigma_{\phi_s(t')}*f\sigma_{\phi_s(t-t')})(x)\one_{E_\gamma}(t')\d t'}_{L^2_{x,t}}\geq 
	\frac{1}{2}\delta^2\mathbf{H}_{4}^2\norma{f}_2^2.\]
	For this choice of $\gamma$ we then obtain
	\begin{equation*}
%	\label{eq:lower-bound-E}
	\begin{split}
	\frac{1}{2}\delta^2\mathbf{H}_{4}^2\norma{f}_2^2&\leq \Norma{\int_{\psi_s(1)}^{\psi_s(2)} 
		(f\sigma_{\phi_s(t')}*f\sigma_{\phi_s(t-t')})(x)\one_{E_\gamma}(t')\d t'}_{L^2_{x,t}}\\
	&\leq \Norma{\int_{\psi_s(1)}^{\psi_s(2)}\norma{f\sigma_{\phi_s(t')}*f\sigma_{\phi_s(t-t')}}_
		{L^2_x}\one_{E_\gamma}(t')\d t'}_{L^2_t}\\
	&\leq 2\mathbf{S}^2\Norma{\int_{\psi_s(1)}^{\psi_s(2)}\norma{f_{\phi_s(t')}}_{L^2_x}\norma{f_{\phi_s(t-t')}}_{L^2_x}
		\one_{E_\gamma}(t')\d t'}_{L^2_t}\\
	&\leq 2\mathbf{S}^2\ab{E_\gamma}^{1/2}\int_{\psi_s(1)}^{\psi_s(2)}\norma{f_{\phi_s(t)}}_{L^2_x}^2\d t
	\leq 2\mathbf{S}^2\norma{f}_2^2\ab{E_\gamma}^{1/2},
	\end{split}
	\end{equation*}
	and therefore $\ab{E_\gamma}\geq \mathbf{H}_{4}^4\delta^4/(16\mathbf{S}^4)$.

	To analyze $\ab{E_{\gamma,\la}}$ we use
	\[ E_{\gamma,\la}=E_\gamma\cap \{t\in[{\psi_s(1)},{\psi_s(2)}]:\norma{f_{\phi_s(t)}}_2\leq 
	\la\delta\mathbf{H}_{4}\norma{f}_2\}. \]
	Chebyshev's and H\"older's inequalities imply
	\begin{align*}
	\ab{\{t\in[{\psi_s(1)},{\psi_s(2)}]:\norma{f_{\phi_s(t)}}_2>\la\delta\mathbf{H}_{4}\norma{f}_2\}}&\leq
	\frac{1}{\la\delta\mathbf{H}_{4}\norma{f}_2}\int_{\psi_s(1)}^{\psi_s(2)}\norma{f_{\phi_s(t)}}_2\d t\\
	&\leq\frac{2}{\la\delta\mathbf{H}_{4}}.
	\end{align*}
	Therefore, choosing $\la=64\mathbf{S}^4/(\mathbf{H}_4^5\delta^5)$ we obtain
	\begin{align*}
	\ab{E_{\gamma,\la}}\geq \ab{E_\gamma}- \ab{\{t\in[{\psi_s(1)},{\psi_s(2)}]:\norma{f_{\phi_s(t)}}_2>
		\la\delta\mathbf{H}_{4}\norma{f}_2\}}
	\geq \frac{\mathbf{H}_4^4}{32\mathbf{S}^4}\delta^4.
	\end{align*}
	From now on, let us fix such values of $\gamma$ and $\la$ and 
	let $E:=E_{\gamma,\la}$. From the definition of $E$ and \eqref{eq:scaling_convolution_same_radius}, we have that for $t\in E$
	\[ \norma{f_{\phi_s(t)}\sigma\ast f_{\phi_s(t)}\sigma}_{L^2(\R^3)}\geq c\phi_s(t)^{-1/2}\delta^{4}\norma{f_{\phi_s(t)}}_{L^2(\Sph^2)}^2, \]
 	so that Lemma \ref{lem:quasiextremals-sphere-CS} imply that for $t\in E$ there are caps $\sphcp_t\subseteq\Sph^2$ and a decomposition $f_{{\phi_s(t)}}=G_{{\phi_s(t)}}+H_{{\phi_s(t)}}$. 
	In this 
	way we obtain a decomposition $f=g+h$, where
	$g(\phi_s(t)x,t)=G_{\phi_s(t)}(x)\one_{E}(t)$, $x\in\Sph^2$, $t\in[{\psi_s(1)},{\psi_s(2)}]$. As argued in Remark \ref{rem:measurability} and recorded in Lemma  \ref{lem:measurable_Sel}, by using a measurable selection theorem we can perform this decomposition in such a way that $g$ and $h$ are measurable functions and $\mathcal G_0:=\{(x,t)\in\R^4:t\in E,\, x\in 
	\phi_s(t)\sphcp_t\}$ is a measurable subset of $\hyp_s$, so that $\mathcal{G}_0$ is a quasi-cap. According to Lemma \ref{lem:quasiextremals-sphere-CS}, $g$ and $h$ satisfy the following conditions: $f=g+h$, $0\leq \ab{g},\ab{h}\leq \ab{f}$, $g$ and $h$ have disjoint supports, $g(x,t)=0$ if $t\notin E$, 
	\begin{gather}
	\ab{g(\phi_s(t)x,t)}\leq C_\delta \norma{f_{\phi_s(t)}}_2\ab{\sphcp_t}^{-1/2}\one_{\sphcp_t}(x),\text{ for all } t\in E,\, x\in\Sph^2,\nonumber\\
	\label{eq:pointwise-L2}
	\norma{g_{\phi_s(t)}}_{2}\geq\eta_\delta\norma{f_{\phi_s(t)}}_2,\,\norma{g_{\phi_s(t)}}_{1}\geq \frac{\eta_\delta^2}{C_\delta}\ab{\sphcp_t}^{1/2}\norma{f_{\phi_s(t)}}_2, \text{ for all } t\in E.
	\end{gather}
	Note that Lemma \ref{lem:sliced-lp-norm} and \eqref{eq:pointwise-L2} imply
	\[ \norma{g}_{2}\geq\eta_\delta\norma{f}_2. \]
	Given that 
	for $t\in E$ we have 
	$\delta^2\mathbf{H}_4\norma{f}_2\lesssim\norma{f_{\phi_s(t)}}_2\lesssim \delta^{-4}\mathbf{H}_4\norma{f}_2$ we conclude, 
	possibly by changing the constants that depend on $\delta$, that the function $g$ satisfies 
	\begin{equation}\label{eq:pointwise_Bound_g}
	\ab{g(\phi_s(t)x,t)}\leq
	C_\delta\norma{f}_2\ab{\sphcp_t}^{-1/2}\one_{\sphcp_t}(x)\one_{E}(t),\text{ for all }t\in[{\psi_s(1)},{\psi_s(2)}],\,x\in\Sph^2 
	\end{equation}
	and
	\begin{equation}\label{eq:lower_Bound_gphis}
	\norma{g_{\phi_s(t)}}_{L^2(\Sph^2)}\geq \eta_\delta\norma{f}_2\,\text{ and }\,
	\norma{g_{\phi_s(t)}}_{L^1(\Sph^2)}\geq \frac{\eta_\delta^2}{C_\delta}\ab{\sphcp_t}^{1/2}\norma{f}_2,\text{  for each }t\in E.
	\end{equation}
	
	Summing up, we can restate what has been done so far in the following way: If $f\in L^2(\mathcal{H}_s^3)$ 
	satisfies $\norma{f\mu_s\ast 
		f\mu_s}_2\geq
	\delta^2\mathbf{H}_4^2\norma{f}_2^2$ and is supported where $1\leq \ab{x}\leq 2$ then there 
	exist a 
	decomposition $f=g+h$, a
	set $E\subseteq[{\psi_s(1)},{\psi_s(2)}]$ 
	satisfying $\ab{E}\gtrsim \delta^4$ and a quasi-cap $\mathcal G_0$ (associated to $E$ as in \eqref{eq:Def_quasi_cap})
	such that $g$ and $h$ have disjoint supports,
	\[\ab{g(x,t)}\leq
	C_\delta\norma{f}_2\ab{\sphcp_t}^{-1/2}\one_{\mathcal G_0}(x,t),\text{ for all }(x,t)\in\hyp_s\]
	and \eqref{eq:lower_Bound_gphis} holds. This is the analog of Lemma \ref{lem:nearly-extremals-hyp} with a quasi-cap instead of a cap.
		
	Using the quasi-cap analog of Lemma \ref{lem:nearly-extremals-hyp}, as described in the previous paragraph, we can argue exactly as in Lemma \ref{lem:quasiextremals-sphere} for the sphere to ensure, 
	possibly after changing the constants that depend on $\delta$, that there 
	exist a quasi-cap, which we continue to denote
	$\mathcal G_0$, associated to a set $E\subseteq[\psi_s(1),\psi_s(2)]$ with $\ab{E}\gtrsim \delta^4$, and functions $g$ and $h$ with the properties of the previous paragraph and additionally 
	\begin{equation}\label{eq:lower_bound_g_ast_g}
	\norma{g\mu_s*g\mu_s}_{L^2(\R^4)}\geq c_\delta\norma{f}_2^2.
	\end{equation}
	The next and final step is to show that the caps $\sphcp_t$, $t\in E$, which define $\mathcal G_0$ are aligned 
	for a large fraction of the $t$'s, and by this we mean that they
	have close radii and centers, up to powers of $\delta$.
		
	Recall that for
	caps $\sphcp,\,\sphcp'\subseteq\Sph^2$ there is a distance function $\varrho(\sphcp,\sphcp')$, defined in \eqref{eq:def_distance}, that is relevant in Lemmas \ref{lem:weak-interaction-caps} and  \ref{lem:weak-interaction-mixed-radii}. 
	For $\rho>0$ define
	\[A_{\rho}=\{(t,t')\in 
	E\times E\colon\varrho(\sphcp_{t},\sphcp_{t'})\leq\rho\}.\]
	Then, starting from \eqref{eq:lower_bound_g_ast_g} we have the estimate
	\begin{align*}
	&c_\delta\norma{f}_2^2\leq \norma{g\mu_s*g\mu_s}_2=\Norma{\int_{\psi_s(1)}^{\psi_s(2)}(g\sigma_{\phi_s(t')}*g\sigma_{\phi_s(t-t')})(x)\d t'}_{L^2_{x,t}}\\
	&\leq 
	\Norma{\int_{\psi_s(1)}^{\psi_s(2)}\norma{g\sigma_{\phi_s(t')}*g\sigma_{\phi_s(t-t')}}_{L_x^2}\one_{A_{\rho}}(t',t-t')\d 
	t'}_{L^2_t}\\
	&\qquad+\Norma{\int_{\psi_s(1)}^{\psi_s(2)}\norma{g\sigma_{\phi_s(t')}*g\sigma_{\phi_s(t-t')}}_{L_x^2}\one_{A_{\rho}^\complement}(t',t-t')\d
		t'}_{L^2_t}\\
	&\leq 
	\Norma{\int_{\psi_s(1)}^{\psi_s(2)}\norma{g\sigma_{\phi_s(t')}*g\sigma_{\phi_s(t-t')}}_{L_x^2}\one_{A_{\rho}}(t',t-t')\d 
		t'}_{L^2_t}
	+C_\delta^2\norma{f}_2^2\Norma{\int_{\psi_s(1)}^{\psi_s(2)} 
		\ab{\sphcp_{t'}}^{-1/2}\\
	&\qquad\cdot\ab{\sphcp_{t-t'}}^{-1/2}\norma{\one_{\phi_s(t')\sphcp_{t'}}\sigma_{\phi_s(t')}\ast\one_{\phi_s(t-t')\sphcp_{t-t'}}\sigma_{\phi_s(t-t')}}_{L^2_x}\one_{(E\times E)\cap A_{\rho}^\complement}(t',t-t')
		\d t'}_{L^2_t}\\
	&\leq 
	\Norma{\int_{\psi_s(1)}^{\psi_s(2)}\norma{g\sigma_{\phi_s(t')}*g\sigma_{\phi_s(t-t')}}_{L_x^2}\one_{A_{\rho}}(t',t-t')\d 
		t'}_{L^2_t}+\frac{c_\delta}{2}\norma{f}_2^2,
	\end{align*}
	where in the second to last line we used \eqref{eq:pointwise_Bound_g} and the last line holds if $\rho$ is large enough as a function of\footnote{From the proof of Lemma 7.6 in \cite{CS} one can see that $\cosh\rho$ can be taken to be a power of $\delta^{-1}$.} $\delta$, by the use of Lemma \ref{lem:weak-interaction-mixed-radii}. For such choice 
	of $\rho$ we can therefore ensure that
	\begin{equation}\label{eq:lower_bound_g_ast_g_A_rho}
	\Norma{\int_{\psi_s(1)}^{\psi_s(2)}\norma{g\sigma_{\phi_s(t')}*g\sigma_{\phi_s(t-t')}}_{L_x^2}\one_{A_{\rho}}(t',t-t')\d 
		t'}_{L^2_t}\geq \frac{c_\delta}{2}\norma{f}_2^2.
	\end{equation}
	Note that \eqref{eq:pointwise_Bound_g} implies $\norma{g_{\phi_s(t)}}_2\leq C_\delta\norma{f}_2$ for all $t\in E$. This and \eqref{eq:lower_bound_g_ast_g_A_rho} imply that 
	\begin{align*}
	\frac{c_\delta}{2}\norma{f}_2^2&\leq
	2\mathbf{S}^2\Norma{\int_{\psi_s(1)}^{\psi_s(2)}\norma{g_{\phi_s(t')}}_2\norma{g_{\phi_s(t-t')}}_2\one_{A_\rho}(t',t-t')\d t'}_{L^2_t} \\
	&\leq 2\mathbf{S}^2C_\delta^2\norma{f}_2^2\int_{\psi_s(1)}^{\psi_s(2)}\norma{\one_{A_\rho}(t',t-t')}_{L^2_t} \d t'\\
	&\leq 4\mathbf{S}^2C_\delta^2\norma{f}_2^2\,\ab{A_\rho}^{1/2},
	\end{align*}
	where 
	$\rho=\rho(\delta)$ is the already fixed function of $\delta$ and $\ab{A_\rho}$ denotes the Lebesgue measure of $A_\rho\subseteq\R^2$. As $\ab{A_\rho}\leq 2$ we conclude that $\ab{A_\rho}\asymp c(\delta)$. By Fubini's theorem, the fibers $A_\rho(t):=\{t'\in E:(t,t')\in A_\rho \}=\{t'\in E:\varrho(\sphcp_t,\sphcp_{t'})\leq\rho \}$ are a.e. measurable, the function $t\in E\mapsto\ab{A_\rho(t)}=\ab{\{t'\in E:\varrho(\sphcp_t,\sphcp_{t'})\leq\rho \}}$ is measurable and $\ab{A_\rho}\leq 2\esssup_{t\in E}\ab{A_\rho(t)}$. We then obtain the following estimate
	\[c(\delta)\asymp\esssup_{t\in E}\ab{\{t'\in E:\varrho(\sphcp_t,\sphcp_{t'})\leq\rho \}}\leq \sup_{(y,a)\in\Sph^2\times (0,\infty)} \ab{\{t'\in E:\varrho(\sphcp(y,a),\sphcp_{t'})\leq\rho \}}, \]
	from where we conclude the existence of a spherical cap $\sphcp(y_0, a_0)$ such that 
	\[\ab{\{t\in E:\varrho(\sphcp(y_0,a_0),\sphcp_{t})\leq\rho \}}\asymp c(\delta).\]
	Denote $\sphcp_0=\sphcp(y_0,a_0)$ and $B_\rho=\{t\in E:\varrho(\sphcp_0,\sphcp_{t})\leq\rho \}$.
	For $t\in B_\rho$, 
	the radii and the distance between the centers of the caps $\sphcp_0$ and $\sphcp_t$ are of the same order modulo 
	powers 
	of $\delta$. More precisely, if we let $(y,a)$ denote the center and radius of a cap $\sphcp_t$, $t\in B_\rho$, then the definition of the distance function $\varrho$ ensures that
	\begin{equation}\label{eq:close-caps}
	c(\delta)a_0\leq a\leq c'(\delta)a_0,\,\text{ and } \ab{y_0-y}\leq 
	c''(\delta)a_0.
	\end{equation}
	This is the only place where we used the assumption that $f$ is supported on a cap 
	$[1,2]\times 
	\sphcp$, were the radius of $\sphcp$ is $\frac{1}{4}$, because this implies that the centers of 
	the caps associated to $g_{\phi_s(t)}$, $t\in E$, can be chosen to be at distance at most $\frac{1}{2}$ from each other and therefore 
	any two caps $\sphcp_t,\,\sphcp_{t'}$ for $t,\,t'\in E$ are not nearly antipodal.
	
	From \eqref{eq:close-caps} we conclude that for $t\in B_\rho$ we have $\ab{\sphcp_{t}}\asymp_\delta\ab{\sphcp_0}$ and there exists $c(\delta)\geq 1$ such that the 
	$c(\delta)$-enlargement of 
	$\sphcp_{0}$, 
	denoted $\sphcp_{0}^\delta$ and defined by
	\[ \sphcp_{0}^\delta:=\{x\in \Sph^2\colon \ab{x-y_0}\leq c(\delta)a_0 \},\]
	contains $\sphcp_t$ for all $t\in B_\rho$, and hence the cap $\cp:=[1,2]\times 
	\sphcp_{0}^\delta\subseteq \hyp_s$ contains the quasi-cap $\mathcal{G}_1:=\{(x,t)\in\mathcal G_0\colon t\in 
	B_\rho\}$. Note also that $\ab{\sphcp_{t}}\asymp_\delta\ab{\sphcp_{0}^\delta}$, for all $t\in B_\rho$.
	
	Now, for each $t\in E$, $g_{\phi_s(t)}$ is supported on $\sphcp_t$ and $\int_{\sphcp_t} \ab{g_{\phi_s(t)}}\d\sigma\geq 
	c(\delta)\ab{\sphcp_t}^{1/2}\norma{f}_2$, as stated in \eqref{eq:lower_Bound_gphis}. If in addition $t\in
	B_\rho$, then 
	\[\int_{\sphcp_{0}^\delta} \ab{g_{\phi_s(t)}}\d\sigma=\int_{\sphcp_t} \ab{g_{\phi_s(t)}}\d\sigma\geq 
	c(\delta)\ab{\sphcp_t}^{1/2}\norma{f}_2\geq 
	c'(\delta)\ab{\sphcp_{0}^\delta}^{1/2}\norma{f}_2,\]
	and so integrating in $t\in B_\rho$ and using that $\phi_s(t)\geq 1$ if $t\geq \psi_s(1)$ gives 
	\[\int_{ B_\rho}\int_{\sphcp_{0}^\delta} \ab{g_{\phi_s(t)}} \phi_s(t)\d\sigma\d t\geq
	c(\delta)\ab{\sphcp_{0}^\delta}^{1/2}\ab{
		B_\rho}\norma{f}_2\geq c'(\delta)\ab{\sphcp_{0}^\delta}^{1/2}\norma{f}_2. \]
	Given that $\mu_s(\cp)=\mu_s([1,2]\times \sphcp_{0}^\delta)\asymp \ab{\sphcp_{0}^\delta}$
	we obtain
	\[\int_{\cp}\ab{g\one_{\mathcal G_1}} \d\mu_s=\int_{ B_\rho}\int_{\sphcp_{0}^\delta} \ab{g_{\phi_s(t)}}\phi_s(t)\d\sigma \d t\geq  c(\delta)\mu_s(\cp)^{1/2}\norma{f}_2.\]
	Then $g\one_{\mathcal G_1},\, f-g\one_{\mathcal G_1}$ and $\cp$ satisfy all of our 
	requirements, given that $\operatorname{supp}(g\one_{\mathcal G_1})\subseteq\overline{\mathcal G_1}\subseteq\cp$, $\mathcal G_1\subseteq \mathcal G_0$, $\ab{\sphcp_t}\asymp_\delta\mu_s(\cp)$ for all $t\in B_\rho$, and thus
	\begin{align*}
	&g\one_{\mathcal G_1}(x,t)\leq c(\delta)\norma{f}_{L^2(\hyp_s)}\mu_s(\cp)^{1/2}\one_{\cp}(x,t),\text{ for all }(x,t),\\
	&\norma{g\one_{\mathcal G_1}}_{L^2(\hyp_s)}\geq c(\delta)\norma{f}_{L^2(\hyp_s)},\\
	&\norma{g\one_{\mathcal G_1}}_{L^1(\hyp_s)}\geq c(\delta)\mu_s(\cp)^{1/2}\norma{f}_{L^2(\hyp_s)}.
	\end{align*}
\end{proof}

\section{A concentration-compactness lemma}\label{sec:concentration_compactness}

The result of this section is stated for $\Hyp_s$ but a similar statement and proof also hold for $\hyp_s$. 

\begin{lemma}
 \label{lem:concentration-compactness}
 Let $\{\rho_n\}_{n}$ be a sequence in $L^2(\Hyp_s)$ satisfying 
 \[\int_{\Hyp_s}\ab{\rho_n}^2\d\bar\mu_s=\la,\]
 where $\la>0$ is fixed. Then there exists a subsequence $\{\rho_{n_k}\}_{k}$ such that $\{\ab{\rho_{n_k}}^2\}_k$ satisfies one of the following three
possibilities:
\begin{enumerate}
 \item [(i)] (compactness) there exists $\ell_k\in\N$ such that
 \[\forall \eps>0,\, \exists R<\infty,\int_{s2^{\ell_k-R}\leq\ab{y}\leq s2^{\ell_k+R}}\ab{\rho_{n_k}}^2 
 \d\bar\mu_s\geq \la-\eps;\]
 \item [(ii)] (vanishing) 
 $\ds\lim\limits_{k\to\infty}\sup_{\ell\in\N}\int_{s2^{\ell-R}\leq\ab{y}\leq 
s2^{\ell+R}}\ab{\rho_{n_k}}^2\d\bar\mu_s=0$,
for all $R<\infty$;
 \item [(iii)] (dichotomy) There exists $\alpha\in(0,\la)$ such that for all $\eps>0$, there exist 
 $R\in\N$, $k_0\geq 1$ and
nonnegative functions $\rho_{k,1},\rho_{k,2}\in L^2(\Hyp_s)$ satisfying for $k\geq k_0$:
\begin{gather}\label{eq:conditions_dichotomy-1}
 \norma{\rho_{n_k}-(\rho_{k,1}+\rho_{k,2})}_{L^2(\Hyp_s)}\leq \eps,\\
  \abs{\int_{\Hyp_s}\ab{\rho_{k,1}}^2\d\bar\mu_s-\alpha}\leq
\eps,\,\abs{\int_{\Hyp_s}\ab{\rho_{k,2}}^2\d\bar\mu_s-(\la-\alpha)}\leq\eps,\\
\supp(\rho_{k,1})\subseteq \{y\in\R^3\colon s2^{\ell_k-R}\leq\ab{y}\leq s2^{\ell_k+R}\},\\
\label{eq:conditions_dichotomy-4}
\supp(\rho_{k,2})\subseteq \{y\in\R^3\colon\ab{y}\leq s2^{\ell_k-R_k}\}\cup 
\{y\in\R^3\colon \ab{y}\geq s2^{\ell_k+R_k}\},
\end{gather}
for certain sequences $\{\ell_k\}_k$ and $\{R_k \}_k$, where $R_k\to\infty$ as $k\to\infty$.
\end{enumerate}
\end{lemma}

\begin{proof}
 The proof is identical to the proof of Lemma I.1 in \cite{Li}, by defining the sequence of functions
 \[Q_n\colon [0,\infty)\to\R_+,\quad Q_n(t)=\sup\limits_{\ell\in\N}\int_{s2^{\ell-t}\leq \ab{y}\leq s2^{\ell+t}}\ab{\rho_n(y)}^2 
 \d\bar\mu_s(y).\]
We omit the details.
\end{proof}

In the forthcoming sections, we will be working with an $L^2$ normalized extremizing sequence $\{f_n\}_n$ and will apply the preceding lemma to with $\la=1$. We will slightly abuse notation and say that $\{f_n\}_n$ satisfies either \textit{concentration, vanishing} or \textit{dichotomy}, when the sequence $\{\ab{f_n}^2\}_n$ satisfies the respective alternative.

\section{Bilinear estimates and discarding dichotomy}\label{sec:no-dichotomy}

In this section we show that an extremizing sequence for $\overline{T}$ can not satisfy the dichotomy condition (iii) of Lemma \ref{lem:concentration-compactness}. This will be a consequence of bilinear estimates at dyadic scales.

\begin{prop}\label{prop:bilinear-estimate-hyp}
 There exists a constant $C<\infty$ with the following property. Let $s>0$, $k,k'\in\N$ and $f,g\in 
 L^2(\mathcal H_s^3)$ supported where
$2^{k}s\leq \ab{y}\leq 2^{k+1}s$ and $2^{k'}s\leq \ab{y}\leq 2^{k'+1}s$ respectively. Then
 \[\norma{T_sf\cdot T_sg}_{L^2(\R^4)}\leq C2^{-\frac{1}{4}\ab{k-k'}}\norma{f}_{L^2(\mathcal 
 H_s^3)}\norma{g}_{L^2(\mathcal H_s^3)}.\]
\end{prop}

\begin{proof}
 Without loss of generality we can assume $k'\geq k$. Using Lemma \ref{lem:sliced-convolution} we  
 write
 \[f\mu_s*g\mu_s(x,t)=\int_0^t (f\sigma_{\phi_s(t')}*g\sigma_{\phi_s(t-t')})(x)\d t',\]
 so that by Minkowski's integral inequality
 \begin{equation}\label{eq:upper_bound_for_conv}
 \norma{f\mu_s*g\mu_s}_{L^2_{x,t}}\leq \NOrma{\int_0^t
\norma{f\sigma_{\phi_s(t')}*g\sigma_{\phi_s(t-t')}}_{L^2_x}\d t'}_{L^2_t}.
\end{equation}
Recalling \eqref{eq:scaling_convolution_L2}, the right hand side of \eqref{eq:upper_bound_for_conv} satisfies
\begin{align*}
\Norma{\int_0^t\norma{f\sigma_{\phi_s(t')}*g&\sigma_{\phi_s(t-t')}}_{L^2_x}\d t'}_{L^4_t}
\\ 
&\leq C\Norma{\int_0^t
\phi_s(t')^{1/4}\norma{f_{\phi_s(t')}}_2\,\phi_s(t-t')^{1/4}\norma{g_{\phi_s(t-t')}}_2\d t'}_{L^2_t}\\
&\leq C\int_0^\infty \phi_s(t')^{1/4}\norma{f_{\phi_s(t')}}_2\Norma{\one_{\{t\geq
t'\}}(t')\phi_s(t-t')^{1/4}\norma{g_{\phi_s(t-t')}}_2}_{L^2_t}\d t'\\
&\leq
C\Norma{\phi_s(t)^{1/4}\norma{g_{\phi_s(t)}}_2}_{L^2_t}\int_{\psi_s(2^{k}s)}^{\psi_s(2^{k+1}s)}\phi_s(t')^{1/4}\norma{f_{
\phi_s(t') } }_2\d t',
\end{align*}
where in the last line we used the support condition for $f$. Recalling the support condition for $g$
\begin{align*}
\Norma{\phi_s(t)^{1/4}\norma{g_{\phi_s(t)}}_{L^2(\Sph^2)}}_{L^2_t}^2
&=\int_{\psi_s(2^{k'}s)}^{\psi_s(2^{k'+1}s)}\phi_s(t)^{1/2}
\norma{g_{\phi_s(t) }}_{L^2(\Sph^2)} ^2 \d t\\
&\leq 
\bigl(\phi_s(\psi_s(2^{k'}s))\bigr)^{-1/2}\int_0^\infty\phi_s(t)\norma{g_{\phi_s(t)}}_{L^2(\Sph^2)}^2
\d t\\
&=(2^{k'}s)^{-1/2}\norma{g}_{L^2(\mathcal{H}^3_s)}^2,
\end{align*}
where in the last line we used Lemma \ref{lem:sliced-lp-norm}. Similarly
\begin{align*}
\int_{\psi_s(2^{k}s)}^{\psi_s(2^{k+1}s)}\phi_s(t')^{1/4}\norma{f_{\phi_s(t') } }_2\d t'&\leq
\Bigl(\int_{\psi_s(2^{k}s)}^{\psi_s(2^{k+1}s)}\phi_s(t')^{1/2}\norma{f_{\phi_s(t')
	}}_2^2\d t'\Bigr)^{1/2}\\
&\quad\qquad\cdot\Bigl(\int_{\psi_s(2^{k}s)}^{\psi_s(2^{k+1}s)} 1\d t'\Bigr)^{1/2}\\
&\leq (2^ks)^{-1/4}(\psi_s(2^{k+1}s)-\psi_s(2^{k}s))^{1/2}\norma{f}_{L^2(\mathcal{H}^3_s)}\\
&\asymp (2^ks)^{-1/4}(2^{k}s)^{1/2}\norma{f}_{L^2(\mathcal{H}^3_s)}\\
&=(2^{k}s)^{1/4}\norma{f}_{L^2(\mathcal{H}^3_s)}.
\end{align*}
We conclude that
\begin{align*}
\norma{f\mu_s*g\mu_s}_{L^2_{x,t}}&\leq\Norma{\int_0^t\norma{f\sigma_{\phi_s(t')}*g\sigma_{\phi_s(t-t')}}_{L^2_x}\d
 t'}_{L^4_t}
\lesssim 2^{k/4}\norma{f}_{L^2(\mathcal{H}^3_s)}2^{-k'/4}\norma{g}_{L^2(\mathcal{H}^3_s)}\\
&=2^{-\frac{1}{4}\ab{k'-k}}\norma{f}_{L^2(\mathcal{H}^3_s)}\norma{g}_{L^2(\mathcal{H}^3_s)}.
\end{align*}
\end{proof}

\begin{prop}
 \label{prop:large-logdistance}
 Let $f,g\in L^2(\mathcal H^3)$ and suppose that their supports are separated in the sense that there 
 exist $k,k'\in\N$, $k\leq k'$, such
that $\supp(f)\subseteq\{\ab{y}\leq 2^{k}\}$ and $\supp(g)\subseteq\{\ab{y}\geq 2^{k'}\}$. Then
\[\norma{Tf\cdot Tg}_{L^2(\R^4)}\leq C2^{-\frac{1}{4}\ab{k-k'}}\norma{f}_{L^2(\mathcal 
H^3)}\norma{g}_{L^2(\mathcal H^3)}.\]
Similarly, if there exist $k,R,R'\in\N$, $R\leq R'$, such that $\supp(f)\subseteq\{2^{k-R}\leq\ab{y}\leq 2^{k+R}\}$ and $\supp(g)\subseteq\{\ab{y}\leq 2^{k-R'} \}\cup\{\ab{y}\geq 2^{k+R'}\}$, then
\[\norma{Tf\cdot Tg}_{L^2(\R^4)}\leq C2^{-\frac{1}{4}\ab{R-R'}}\norma{f}_{L^2(\mathcal 
	H^3)}\norma{g}_{L^2(\mathcal H^3)}.\]
\end{prop}

\begin{proof}
 We decompose $f=\sum_{m\in\N}f_m$ and $g=\sum_{m'\in\N}g_{m'}$ where $f_m,g_m$ are supported 
 where $2^{m}\leq \ab{y}\leq 2^{m+1}$, $m\geq 0$. Then
\begin{align*}
 \norma{Tf\cdot Tg}_{L^2(\R^4)}&= \Norma{\suma{m,m'}{}Tf_m\cdot Tg_{m'}}_{L^2}\leq \suma{m,m'}{}\norma{Tf_m\cdot Tg_{m'}}_{L^2}\\
 &\leq \suma{m,m'}{}2^{-\frac{1}{4}\ab{m-m'}}\norma{f_m}_{L^2}\norma{g_{m'}}_{L^2}\\
 &=2^{-\frac{1}{4}\ab{k'-k+1}}\suma{m\leq 0,m'\geq
0}{}2^{-\frac{1}{4}\ab{m-m'}}\norma{f_{m+k-1}}_{L^2}\norma{g_{m'+k'}}_{L^2}\\
&\leq C 2^{-\frac{1}{4}\ab{k'-k}}\norma{f}_{L^2(\hyp)}\norma{g}_{L^2(\hyp)}.
\end{align*}
The second part of the proposition follows from the first and the triangle inequality.
\end{proof}

Decomposing a function $f\in L^2(\Hyp)$ as the sum of a function $f_+\in L^2(\hyp)$ and $f_-\in L^2(-\hyp)$, $f=f_++f_-$, using that $\overline{T}f(\cdot,\cdot)=Tf_+(\cdot,\cdot)+Tf_-(\cdot,-\cdot)$ and the triangle inequality we can obtain a statement analog to the previous proposition for functions on the full one-sheeted hyperboloid $\Hyp$: if $f,g$ belong to $L^2(\Hyp)$ and satisfy for some $k,R,R'\in\N$, $R\leq R'$:
\[\supp(f)\subseteq\{2^{k-R}\leq\ab{y}\leq 2^{k+R}\},\,\supp(g)\subseteq\{\ab{y}\leq 2^{k-R'} \}\cup\{\ab{y}\geq 2^{k+R'}\},\]
then 
\begin{equation}\label{eq:Weak_interaction_bar_T}
\norma{\overline{T}f\cdot \overline{T}g}_{L^2(\R^4)}\leq C2^{-\frac{1}{4}\ab{R-R'}}\norma{f}_{L^2(\Hyp 
	)}\norma{g}_{L^2(\Hyp)}.
\end{equation}

\begin{prop}\label{prop:no-dichotomy}
An extremizing sequence for the adjoint Fourier restriction inequality \eqref{sharp_L4_double_hyp} on $\overline{\mathcal H}^{\,3}$ does not satisfy dichotomy.
\end{prop}

\begin{proof}
 Let us argue by contradiction. Let $\{f_n\}_n$ be an extremizing sequence such that 
 $\{\ab{f_n}^2\}_n$ satisfies
condition (iii), \textit{dichotomy}, in Lemma \ref{lem:concentration-compactness}. Let $\eps>0$ be given and $f_{n,1},f_{n,2}$, $n_0$ be as in
the conclusion of the dichotomy condition. Then, for $n\geq n_0$
\[\norma{\overline{T}f_n-\overline{T}f_{n,1}-\overline{T}f_{n,2}}_{L^4}\leq  \overline{\mathbf{H}}_4\norma{f_n-(f_{n,1}+f_{n,2})}_{L^2}\leq \overline{\mathbf{H}}_4\,\eps,\]
therefore
\begin{equation}\label{eq:dichotomy-bound-Tf}
\norma{\overline{T}f_n}_{L^4}\leq  
\overline{\mathbf{H}}_4\,\eps+\norma{\overline{T}(f_{n,1}+f_{n,2})}_{L^4},
\end{equation}
Expanding, using Proposition \ref{prop:large-logdistance} (or the comment thereafter) and the support condition for $f_{n,1}$ and $f_{n,2}$ as in \eqref{eq:conditions_dichotomy-1}--\eqref{eq:conditions_dichotomy-4}, there exists $C<\infty$ independent of $\eps$ such that 
for all $n$ large enough
\begin{align*}
 \norma{\overline{T}(f_{n,1}+f_{n,2})}_{L^4}^4&=\norma{(\overline{T}(f_{n,1}+f_{n,2}))^2}_{L^2}^2=\norma{(\overline 
 Tf_{n,1})^2+2\overline
Tf_{n,1}\cdot \overline{T}f_{n,2}+(\overline{T}f_{n,2})^2}_{L^2}^2\\
 &=\norma{\overline{T}f_{n,1}}_{L^4}^4+\norma{\overline{T}f_{n,2}}_{L^4}^4+2\langle (\overline{T}f_{n,1})^2,(\overline 
 Tf_{n,2})^2\rangle\\
 &\quad+4\langle (\overline
Tf_{n,1})^2,\overline{T}f_{n,1}\cdot \overline{T}f_{n,2}\rangle
+4\langle (\overline
Tf_{n,2})^2,\overline{T}f_{n,1}\cdot \overline{T}f_{n,2}\rangle\\
&\leq \norma{\overline{T}f_{n,1}}_{L^4}^4+\norma{\overline{T}f_{n,2}}_{L^4}^4+\eps.\\
&\leq \overline{\mathbf{H}}_4^4\norma{f_{n,1}}_2^4+\overline{\mathbf{H}}_4^4\norma{f_{n,2}}_2^4+\eps\\
&\leq\overline{\mathbf{H}}_4^4(\alpha^2+(1-\alpha)^2)+C\eps,
\end{align*}
so that using \eqref{eq:dichotomy-bound-Tf} and taking $n\to\infty$ we find that for any $\eps>0$
\[ \overline{\mathbf{H}}_4^4\leq\overline{\mathbf{H}}_4^4(\alpha^2+(1-\alpha)^2)+C\eps, \]
for some constant $C<\infty$ independent of $\eps$.

We conclude $1\leq\alpha^2+(1-\alpha)^2$. We reach a contradiction since $\alpha\in(0,1)$ and the numerical inequality $\alpha^2+(1-\alpha)^2<1$ holds. 
\end{proof}
 
The proof we just gave to discard \textit{dichotomy} can be seen in the context of the \textit{strict superaditivity condition} as proposed by Lions \cite{Li}*{Section I.2}; see for instance the comment at the end of Appendix A in \cite{OeSQ18}.

\section{Dyadic refinements and discarding vanishing}\label{sec:no-vanishing}
 
In this section we prove a dyadic improvement of the $L^2\to L^4$ inequality \eqref{adjoint-restriction-one-sheeted} that will imply 
that extremizing sequences for $\overline{T}$ do not satisfy the \textit{vanishing} condition (ii) of Lemma \ref{lem:concentration-compactness}. We start with the following 
proposition.

\begin{prop}\label{prop:dyadic-bound}
	There exists a constant $C<\infty$ with the following property. Let $f\in 
	L^2({\mathcal{H}}^3)$ and for $k\in\N$ let $f_k(y)=f(y)\one_{\{2^{k}\leq 
	\ab{y}<	2^{k+1}\}}$. Then
\begin{equation}\label{eq:improved-dyadic-bound}
\norma{Tf}_{L^4(\R^4)}\leq C\biggl(\sum_{k\geq 
0}\norma{f_k}_{L^2(\overline{\mathcal{H}}^3)}^3\biggr)^{1/3}.
\end{equation}
\end{prop}

\begin{proof}
	We follow \cite{RQ1}*{Proof of Proposition 3.4}. We have
	\[ \norma{Tf}_{L^4(\R^4)}^{3}=\norma{Tf\cdot Tf\cdot Tf}_{L^{4/3}}=\Norma{\sum_{k,l,m}Tf_k\cdot 
	Tf_l\cdot Tf_m}_{L^{4/3}}
	\leq \sum_{k,l,m}\norma{Tf_k\cdot Tf_l\cdot Tf_m}_{L^{4/3}}. \]
	Fix a triplet $(k,l,m)$. We can assume without loss of generality that 
	$\ab{k-l}=\max\{\ab{k-l},\ab{k-m},\ab{l-m}\}$ so that the use of H\"older's inequality and 
	Proposition \ref{prop:bilinear-estimate-hyp} give
	\begin{align*}
	\norma{Tf_k\cdot Tf_l\cdot Tf_m}_{L^{4/3}}&\leq \norma{Tf_k\cdot Tf_l}_{L^2}\norma{Tf_m}_{L^4}\\
	&\lesssim 2^{-\frac{1}{4}\ab{k-l}}\norma{f_k}_{L^2}\norma{f_l}_{L^2}\norma{f_m}_{L^2}\\
	&\leq 
	2^{-\ab{k-l}/12}2^{-\ab{k-m}/12}2^{-\ab{l-m}/12}\norma{f_k}_{L^2}\norma{f_l}_{L^2}\norma{f_m}_{L^2}.
	\end{align*}
	We conclude that
	\[ \norma{Tf}_{L^4(\R^4)}^{3}\lesssim 
	\sum_{k,l,m}2^{-\ab{k-l}/12}2^{-\ab{k-m}/12}2^{-\ab{l-m}/12}\norma{f_k}_{L^2}\norma{f_l}_{L^2}\norma{f_m}_{L^2}.
	 \]
	 Applying H\"older's inequality to the last estimate
	 we obtain
	 \[ \norma{Tf}_{L^4(\R^4)}^{3}\lesssim 
	 \sum_{k,l,m}2^{-\ab{k-l}/12}2^{-\ab{k-m}/12}2^{-\ab{l-m}/12}\norma{f_k}_{L^2}^3\lesssim 
	 \sum_{k}\norma{f_k}_{L^2}^3.\]
\end{proof}

As an application we have the following corollary.

\begin{cor}\label{cor:weak-cap-bound}
	There exists a constant $C<\infty$ with the following property. Let $f\in 
	L^2({\mathcal{H}}^3)$ and for $k\in\N$ let $f_k(y)=f(y)\one_{\{2^{k}\leq \ab{y}< 
		2^{k+1}\}}$. Then
	\begin{equation}\label{eq:improved-cap-bound}
	\norma{Tf}_{L^4(\R^4)}\leq 
	C\sup_{k\in\N}\norma{f_k}_{L^2({\mathcal{H}}^3)}^{1/3}\norma{f}_{L^2({\mathcal{H}}^3)}^{2/3}.
	\end{equation}
\end{cor}

\begin{proof}
	From Proposition \ref{prop:dyadic-bound} we obtain
	\begin{align*}
	\norma{Tf}_{L^4(\R^4)}&\leq 	C\Bigl(\sum_{k\geq 		
	0}\norma{f_k}_{L^2({\mathcal{H}}^3)}^3\Bigr)^{1/3}=C\Bigl(\sum_{k\geq 		
	0}\norma{f_k}_{L^2({\mathcal{H}}^3)}\cdot 
	\norma{f_k}_{L^2({\mathcal{H}}^3)}^2\Bigr)^{1/3}\\
	&\leq C\sup_{k\in\N}\norma{f_k}_{L^2({\mathcal{H}}^3)}^{1/3}
	\Bigl(\sum_{k\geq 0}\norma{f_k}_{L^2({\mathcal{H}}^3)}^2\Bigr)^{1/3}\\
	&=C\sup_{k\in\N}\norma{f_k}_{L^2({\mathcal{H}}^3)}^{1/3}
	\norma{f}_{L^2({\mathcal{H}}^3)}^{2/3}.\qedhere
	\end{align*}
\end{proof}

The same previous argument and \eqref{eq:Weak_interaction_bar_T} gives 
\begin{equation}\label{eq:improved-cap-bound_bar_T}
	\norma{\overline Tf}_{L^4(\R^4)}\lesssim 
	\sup_{k\in\N}\norma{f_k}_{L^2(\Hyp)}^{1/3}\norma{f}_{L^2(\Hyp)}^{2/3}.
\end{equation}
and thus it is immediate that for an extremizing sequence for $\overline{T}$ the vanishing alternative does not hold.

\begin{prop}\label{prop:no-vanishing}
	Extremizing sequences for the adjoint Fourier restriction inequality \eqref{sharp_L4_double_hyp} on $\overline{\mathcal H}^{\,3}$ 
	do not satisfy vanishing.
\end{prop}

\section{Convergence to the cone}\label{sec:conv-to-cone}

The content of this section is important in the study of the compactness alternative of Lemma \ref{lem:concentration-compactness}, in the case in which, in addition, the extremizing sequences concentrate at infinity.

Formally, we can write $\Gamma^3=\hyp_0$, $\sigma_c=\mu_0$ and $T_c=T_0$. It is natural then to study relationships between the adjoint Fourier restriction operator on cone $(\Gamma^3,\sigma_c)$ and on each member of the family $\{(\hyp_s,\mu_s)\}_{s>0}$, in the limit $s\to0^+$, and this is the content of this section (see also \cite{KSV}*{Lemma 2.9} for related results for the case of the two-sheeted hyperboloid). 

 Note that if $0\leq t\leq s$ and $\ab{y}\geq s$, then the inequality $\sqrt{\ab{y}^2-s^2}\leq \sqrt{\ab{y}^2-t^2}$ implies that for $f\in L^2(\mu_s)$
\[ \norma{f\one_{\{\ab{y}\geq s\}}}_{L^2(\sigma_c)}\leq \norma{f\one_{\{\ab{y}\geq s\}}}_{L^2(\mu_t)}\leq \norma{f}_{L^2(\mu_s)}, \]
and for $f\in L^2(\mu_s)$, extended to be zero in the region where $\ab{y}\leq s$,  
\[ \lim_{t\to 0^+}\norma{f}_{L^2(\mu_t)}=\norma{f}_{L^2(\sigma_c)}. \]

Throughout this section we will commonly encounter the situation of having $f\in L^2(\hyp_s)$ and regard it as a function in $L^2(\hyp_t)$, $0\leq t\leq s$, via the orthogonal projection to $\R^3\times\{0\}$. In this case, it will be understood that $f$ is extended by zero in the region where\footnote{Alternatively, we can think of $f$ as a function living in $L^2(\R^3,w\d x)$, for different weights $w$.} $\ab{y}\leq s$.

Let us consider the following situation. Let $a>0$, $\{s_n\}_n\subset \R$ satisfying 
$s_n\to 0$ 
as 
$n\to\infty$. Let $\{f_n\}_n$ be a family of functions with $f_n\in L^2(\mathcal H_{s_n}^3)$, supported where $\ab{y}\geq a$ and satisfying ${\sup_n\norma{f_n}_{L^2(\mu_{s_n})}<\infty}$. 
As already noted, $\norma{f_n}_{L^2(\mu_{s_n})}\geq 
\norma{f_n}_{L^2(\sigma_c)}$,
therefore $\{f_n\one_{\{\ab{y}\geq s_n\}}\}_n$ is a bounded
sequence in $L^2(\sigma_c)$. 
We can assume, 
possibly after passing to a subsequence, that $f_n\rightharpoonup f$ in
$L^2(\sigma_c)$. The aim of this section is to compare $
\norma{f\sigma_c*f\sigma_c}_2$ and the limiting behavior of
$\norma{f_n\mu_{s_n}*f_n\mu_{s_n}}_2$, as $n\to\infty$, in the case $f\neq 0$. We have some preliminary results. 

\begin{lemma}\label{lem:l4-convergence-to-cone}
	Let $a>0$ and $f\in L^2(\mathcal{H}_s^3)$ for all small $s> 0$ and supported where $\ab{y}\geq a$, then
	\begin{equation*}
%	\label{eq:convergence-norm-to-cone}
	\norma{T_{s}f-T_cf}_{L^4(\R^4)}\to 0 \text{ as } s\to0^+. 
	\end{equation*}
\end{lemma}
One possible way to prove Lemma \ref{lem:l4-convergence-to-cone} can be to follow the outline in the proof of 
\cite{KSV}*{Lemma 2.9 (d)} for which we would need some mixed norm Strichartz estimates, but	we try a different approach using that we are working with even integers. 

\begin{proof}[Proof of Lemma \ref{lem:l4-convergence-to-cone}]
	From the uniform in $s$ bound $\norma{T_{s}}=\norma{T}$ and density arguments, it suffices to consider the case 
	when $f\in C_c^\infty(\R^3)$. Let $b\in(a,\infty)$ be such that the support of $f$ is contained in the region where $a\leq\ab{y}\leq b$. 
	
	By Plancherel's 
	theorem, to show $T_sf\to Tf$ in $L^4(\R^4)$, as $s\to 0^+$, it suffices to show that $f\mu_{s}\ast f\mu_{s}\to f\sigma_c\ast f\sigma_c$ and 
	$f\mu_{s}\ast f\sigma_c\to f\sigma_c\ast f\sigma_c$ in 
	$L^2(\R^4)$, as $s\to 0^+$. 
	
	First, we claim that there is pointwise convergence $f\mu_{s}\ast f\mu_{s}(\xi,\tau)\to 
	f\sigma_c\ast f\sigma_c(\xi,\tau)$ and $f\mu_{s}\ast f\sigma_c(\xi,\tau)\to 
	f\sigma_c\ast f\sigma_c(\xi,\tau)$, a.e. $(\xi,\tau)\in\R^4$, as $s\to 0^+$. Indeed, as in the proof of the explicit 
	formula for 
	$\mu_s\ast\mu_s$ in Section \ref{sec:calculation_convolution}, we can write integral 
	formulas 
	for $f\mu_s\ast f\mu_s$, $f\mu_s\ast f\sigma_c$ and $f\sigma_c\ast f\sigma_c$ for any $s\geq 
	0$. Unwinding the change of variables used in the proof of Proposition 
	\ref{prop:formula-double-convolution}, for $\xi\neq 0$ we let
	\begin{align*}
	\alpha_s(u,v,\te)=\frac{\ab{\xi}^2+uv}{\ab{\xi}\sqrt{(u+v)^2+4s^2}},\quad
	\beta_s(u,v,\te)=\frac{\ab{\xi}^2+uv-s^2}{\ab{\xi}(u+v)},
	\end{align*}
	\begin{align*}
	\omega_{s}(u,v,\te)=\bigl(\sqrt{1-\alpha_s(u,v,\te)^2}\cos\te,\sqrt{1-\alpha_s(u,v,\te)^2}\sin\te,
	\alpha_s(u,v,\te)\bigr),\\
	\vartheta_s(u,v,\te)=\bigl(\sqrt{1-\beta_s(u,v,\te)^2}\cos\te,\sqrt{1-\beta_s(u,v,\te)^2}\sin\te,\beta_s(u,v,\te)\bigr),
	\end{align*}
	and
	\begin{align*}
	F_s(u,v)&=\int_0^{2\pi}f\Bigl(\sqrt{(\tfrac{u+v}{2})^2+s^2}\,\omega_s(u,v,\te)\Bigr)\, 
	f\Bigl(\sqrt{(\tfrac{u-v}{2})^2+s^2}\,\omega_s(u,v,\te)\Bigr)\d\te,\\
	G_s(u,v)&=\int_0^{2\pi}f\bigl(\tfrac{u+v}{2}\,\vartheta_s(u,v,\te)\bigr)\, 
	f\Bigl(\sqrt{(\tfrac{u-v}{2})^2+s^2}\,\vartheta_s(u,v,\te)\Bigr)\d\te,\\
	H_0(u,v)&=\int_0^{2\pi}f\bigl(\tfrac{v+v}{2}\,\omega_0(u,v,\te)\bigr)\, 
	f\bigl(\tfrac{v-u}{2}\,\omega_0(u,v,\te)\bigr)\d\te.
	\end{align*}
	Recalling the sets $\widetilde R_s(\xi)$ and $\widetilde{Q}_s(\xi)$ from  \eqref{eq:preliminar-formula-conv} and \eqref{eq:formula-pre-mixed-conv} we have
	\begin{align}\label{eq:convolution-functions}
	f\mu_s\ast f\mu_s(\xi,\tau)&=\frac{1}{2\ab{\xi}}\int_{\widetilde 
		R_s(\xi)}F_s(u,v)\ddirac{\tau-v}\d u\d v,\\
	\label{eq:convolution-functions-mixed}
	f\mu_s\ast f\sigma_c(\xi,\tau)&=\frac{1}{2\ab{\xi}}\int_{ \widetilde 
		Q_s(\xi)}G_s(u,v)\ddirac{\tau-v}\d u\d v,
	\intertext{and}
	\label{eq:convolution-functions-cone}
	f\sigma_c\ast f\sigma_c(\xi,\tau)&=\frac{1}{2\ab{\xi}}\int_{-\ab{\xi}}^{\ab{\xi}}
	H_0(u,\tau)\d u\,\one_{\{\tau\geq\ab{\xi} \}}(\xi,\tau).
	\end{align}
	Given that $\widetilde R_s(\xi)$ and $\widetilde Q_s(\xi)$ are explicit, we can spell out 
	\eqref{eq:convolution-functions} and \eqref{eq:convolution-functions-mixed} and integrate the Dirac 
	delta from where it becomes clear that there is a.e. pointwise convergence to 
	$f\sigma_c\ast 
	f\sigma_c$ as $s\to 0^+$. Alternatively, note that for each fixed $\xi\neq 0$, $\one_{\widetilde 
		R_s(\xi)}(u,v)\to \one_{\{\ab{u}\leq \ab{\xi}\leq v\}}(u,v)$ and $\one_{ \widetilde 
		Q_s(\xi)}(u,v)\to \one_{\{\ab{u}\leq \ab{\xi}\leq v\}}(u,v)$ a.e. pointwise as $s\to 0^+$. We 
	omit further details.
	
	By the Dominated Convergence Theorem, to finish it suffices to show that there 
	exists $F\in L^2(\R^4)$ such 
	that $\ab{f\mu_s\ast f\mu_s(\xi,\tau)}\leq F(\xi,\tau)$ and $\ab{f\mu_s\ast 
		f\sigma_c(\xi,\tau)}\leq F(\xi,\tau)$, 
	for a.e. $(\xi,\tau)\in\R^4$. We use the inequalities
	\begin{align*}
	\ab{f\mu_{s}\ast f\mu_{s}(\xi,\tau)}^2&\leq \norma{f}_{L^\infty}^4\,\bigl(\mu_{s}\ast 
	\mu_{s}\bigr)^2(\xi,\tau),\\
	\ab{f\mu_{s}\ast f\sigma_c(\xi,\tau)}^2&\leq \norma{f}_{L^\infty}^4\,\bigl(\mu_{s}\ast 
	\sigma_c\bigr)^2(\xi,\tau).
	\end{align*}
	On the supports of $f\mu_{s}\ast f\mu_{s}$ and $f\mu_{s}\ast f\sigma_c$, the functions 
	$\mu_{s}\ast \mu_{s}$ and $\mu_{s}\ast\sigma_c$ are 
	uniformly 
	bounded in $s\in(0,1)$, as can be seen from Lemma \ref{lem:behavior-infinity} and formula
	\eqref{eq:infinity-norm-mixed-conv}. It follows that we can take
	\[F(\xi,\tau)=4\pi\norma{f}_{L^\infty}^2\bigl(1+a^{-1}\bigr)\one_{\{a\leq \tau\leq 
		2b\}}\one_{\{\ab{\xi}\leq 2b\}}(\xi,\tau).\]
\end{proof}

Recall the Fourier multiplier notation and the $\dot H^{1/2}(\R^3)$ homogeneous Sobolev norm and inner product from \eqref{eq:multiplier_notation} and \eqref{eq:def_homogeneous_Sobolev}.
We have the following lemma.

\begin{lemma}\label{lem:H-convergence-cone}
	Let $a>0$, then for each fixed $t\in\R$
	\[\lim\limits_{s\to 0}\sup_{\substack{u\in \dot H^{1/2}(\R^3)\\ \supp(\hat
			u)\subseteq\{\xi\in\R^3:\ab{\xi}\geq a\}}}\frac{\norma{e^{it\sqrt{-\Delta-s^2}}u-e^{it\sqrt{-\Delta}}u}_{\dot 
			H^{1/2}(\R^3)}}{\norma{u}_{\dot
			H^{1/2}(\R^3)}}=0.\]
\end{lemma}

\begin{proof}
	For any $s\geq 0$ we have $\norma{e^{it\sqrt{-\Delta-s^2}}u}_{\dot H^{1/2}(\R^3)}=\norma{u}_{\dot 
		H^{1/2}}$. Now
	\[e^{it\sqrt{\ab{y}^2-s^2}}-e^{it\ab{y}}=\int_0^s \frac{\d}{\d r}e^{it\sqrt{\ab{y}^2-r^2}} 
	\d 
	r=-it\int_0^s
	e^{it\sqrt{\ab{y}^2-r^2}}\frac{r}{\sqrt{\ab{y}^2-r^2}}\d r.\]
	Then,
	\begin{align*}
	\norma{(e^{it\sqrt{-\Delta-s^2}}-e^{it\sqrt{-\Delta}})u}_{\dot H^{1/2}(\R^3)}&\leq\ab{t}\int_0^s
	\Norma{e^{it\sqrt{-\Delta-r^2}}\frac{r}{\sqrt{-\Delta-r^2}}u}_{\dot
		H^{1/2}(\R^4)}\d r\\
	&=\ab{t}\int_0^s\Norma{\frac{r}{\sqrt{-\Delta-r^2}}u}_{\dot H^{1/2}(\R^3)}\d r.
	\end{align*}
	If $0\leq s<a$ and $\supp(\hat u)\subseteq \{\ab{\xi}\geq a\}$, then 
	\[\Norma{\frac{r}{\sqrt{-\Delta-r^2}}u}_{\dot H^{1/2}(\R^3)}\leq
	\frac{r}{\sqrt{a^2-r^2}}\norma{u}_{\dot 
		H^{1/2}(\R^3)},\]
	so that
	\[ \norma{(e^{it\sqrt{-\Delta-s^2}}-e^{-it\sqrt{-\Delta}})u}_{\dot H^{1/2}(\R^3)}\leq 
	\ab{t}(a-\sqrt{a^2-s^2})\norma{u}_{\dot H^{1/2}(\R^3)}, \]
	and the conclusion follows.
\end{proof}

We now address the pointwise convergence of $T_{s_n}f_n$ to $T_cf$.

\begin{lemma}\label{lem:pointwise-convergence-to-cone}
	Let $a>0$ and $\{s_n\}_n$ be a sequence of positive real numbers converging to zero. Let $f\in 
	L^2(\Gamma^3)$ and $\{f_n\}_{n}$ be a sequence satisfying $f_n\in L^2(\mathcal{H}_{s_n}^3)$, 
	$\sup_n\norma{f_n}_{L^2(\mu_{s_n})}<\infty$ and supported where $\ab{y}\geq a$, for all 
	$n$. Suppose that $f_n\rightharpoonup f$ in 
	$L^2(\Gamma^3)$, as $n\to\infty$. Then, as $n\to\infty$
	\[ T_{s_n}f_{n}(x,t)\to T_cf(x,t) \text{ for a.e. } (x,t)\in\R^4. \]
\end{lemma}

\begin{proof}
	Following the argument in the proof of Proposition \ref{prop:pointwise-convergence}, we start by defining $v_n$ and $v$ by their Fourier transforms
	\[\hat v_n(y)=\frac{f_n(y)}{\sqrt{\ab{y}^2-s_n^2}},\quad \hat v(y)=\frac{f(y)}{\ab{y}}.\]
	Since $\sup_n\norma{f_n}_{L^2(\Gamma^3)}\leq \sup_n\norma{f_n}_{L^2(\mu_{s_n})}<\infty$ and the functions are supported where $\ab{y}\geq a>0$ we see 
	that 
	\[\sup_n\norma{v_n}_{\dot
		H^{1/2}(\R^3)}^2=\sup_n\int_{\R^3}\ab{\hat v_n(y)}^2\ab{y}\d y\leq \sup_n\frac{a}{\sqrt{a^2-s_n^2}}\norma{f_n}_{L^2(\mu_{s_n})}^2<\infty\]
	and
	\[\sup_n\norma{v_n}_{
		L^{2}(\R^3)}^2=(2\pi)^{-3}\sup_n\int_{\R^3}\ab{\hat v_n(y)}^2\d y\leq (2\pi)^{-3}\sup_n\frac{1}{\sqrt{a^2-s_n^2}}\norma{f_n}_{L^2(\mu_{s_n})}^2<\infty.\]
	If $\vphi\in \dot H^{1/2}(\R^3)$, then $\hat\vphi(\cdot)\ab{\cdot}\in L^2(\Gamma^3)$, from where we can deduce that $v_n\rightharpoonup v$ in $\dot H^{1/2}(\R^3)$, as $n\to\infty$. The operator $T_{s_n}$ applied to 
	$f_n$ equals
	$(2\pi)^3e^{it\sqrt{-\Delta-s_n^2}}v_n$. Fix
	$t\in \R$. From Lemma \ref{lem:H-convergence-cone} we know
	$\norma{(e^{it\sqrt{-\Delta-s_n^2}}-e^{it\sqrt{-\Delta}})\mathbbm{1}_{\{\sqrt{-\Delta}\geq 
			a\}}}\to 0$ as $n\to\infty$, the norm being	as 
	operators on $\dot H^{1/2}(\R^3)$. This, added to the continuity of $e^{it\sqrt{-\Delta}}$ 
	in $\dot H^{1/2}(\R^3)$ implies 
	\[e^{it\sqrt{-\Delta-s_n^2}}v_n\rightharpoonup e^{it\sqrt{-\Delta}}v\]
	weakly in $\dot H^{1/2}(\R^3)$, as $n\to\infty$. Then, by the Rellich-Kondrashov Theorem, for any $R>0$
	\[e^{it\sqrt{-\Delta-s_n^2}}v_n\to e^{it\sqrt{-\Delta}}v\]
	strongly in $L^2(B(0,R))$, as $n\to\infty$. Denote by
	\[F_n(t):=\int_{\ab{x}<R}\abs{e^{it\sqrt{-\Delta-s_n^2}}v_n-e^{it\sqrt{-\Delta}}v}^2\d 
	x=\norma{e^{it\sqrt{-\Delta-s_n^2}}v_n-e^{
			it\sqrt{-\Delta}}v}_{ L^2(B(0,R))}^2.\]
	Since we have $\norma{\hat v_n}_{L^2(\R^3)}\lesssim_a 
	\norma{f_n}_{L^2(\mu_{s_n})}$ and $\norma{\hat v}_{L^2(\R^3)}\lesssim_a 
	\norma{f}_{L^2(\sigma_c)}$, we obtain
	\begin{align*}
	F_n(t)&\leq\norma{e^{it\sqrt{-\Delta-s_n^2}}v_n-e^{
			it\sqrt{-\Delta}}v}_{L^2(\R^3)}^2
	\leq (\norma{v_n}_{L^2(\R^3)}+\norma{v}_{L^2(\R^3)})^2\\
	&\leq C(\norma{\hat v_n}_{L^2(\R^3)}^2+\norma{\hat v}_{L^2(\R^3)}^2)\\
	&\lesssim \norma{f_n}_{L^2(\mu_s)}^2+\norma{f}_{L^2(\sigma_c)}^2.
	\end{align*}
	We can now finish exactly as in the proof of Proposition \ref{prop:pointwise-convergence} and conclude that there exists a subsequence $\{n_k\}_k$ such that
	\[T_{s_{n_k}}f_{n_k}-T_cf\to 0\quad\text{a.e. in }\R^4.\]
\end{proof}

Finally, we prove that the existence of a nonzero weak limit of an extremizing sequence that concentrates at infinity implies that the operator norm of $T$ is upper bounded by that of $T_c$ (which in the end we will pair with Proposition \ref{prop:comparison-cone} to rule out this scenario).

\begin{lemma}\label{lem:scaling-convergence-cone}
 Let $\{s_n\}_n$ be a sequence of positive real numbers converging to zero. Let $f\in 
 L^2(\Gamma^3)$ be a nonzero function and $\{f_n\}_{n}$ be a sequence satisfying $f_n\in 
 L^2(\mathcal{H}_{s_n}^3)$, for all 
 $n$. Suppose that: 
 \begin{itemize}
 	\item[(i)] $\norma{f_n}_{L^2(\mu_{s_n})}=1$,
 	\item [(ii)] $\norma{T_{s_n}f_n}_{L^4}\to \norma{T} \,(=\norma{T_1})$,
 	\item[(iii)] $f_n\rightharpoonup f\neq 0$ in $L^2(\Gamma^3)$,
\end{itemize}
	If there exists $a>0$ such that
	\begin{itemize}
		\item [(iv)] $\supp(f),\supp(f_n)\subseteq\{y\in\R^3\colon\ab{y}\geq a\}$, for all $n$,
	\end{itemize}
 then 
 \[\norma{T}\leq\norma{T_c}.\]
 If condition {\rm(iv)} is relaxed to 
 \begin{itemize}
 	\item[(v)] $\sup_{n\in\N}\norma{f_n\one_{\{\ab{y}\leq a\}}}_{L^2(\mu_{s_n})}\leq\eps$, for some $\eps>0$,
 \end{itemize}
 then
 \[ \norma{T}^2\norma{f\one_{\{\ab{y}\geq a\}}}_{L^2(\sigma_c)}^2\leq \norma{T_c}^2\norma{f\one_{\{\ab{y}\geq a\}}}_{L^2(\sigma_c)}^2+C\eps,
  \] 
  for some universal constant $C$. In particular, if we have $\sup_{n\in\N}\norma{f_n\one_{\{\ab{y}\leq a\}}}_{L^2(\mu_{s_n})}\to 0$ as $a\to0^+$, then $\norma{T}\leq \norma{T_c}$.
\end{lemma}

An analog statement applies if we change $T$ and $T_c$ by $\overline{T}$ and $\overline{T}_c$,  respectively, the proof being identical.

\begin{proof}
 We argue as in \cite{FVV}. By the weak convergence condition (iii),
 \begin{equation}\label{eq:weak-conv-norm}
\norma{f_n-f}_{L^2(\sigma_c)}^2=\norma{f_n}_{L^2(\sigma_c)}^2-\norma{f}_{L^2(\sigma_c)}^2+o(1).
 \end{equation}
 Now consider that (iv) holds. By (i) and (iv), $\norma{f_n}_{L^2(\sigma_c)}^2 -\norma{f_n}_{L^2(\mu_{s_n})}^2\to 0$. Indeed,
\begin{equation}\label{eq:convergence_norm_f_n}
\begin{split}
0\leq \norma{f_n}_{L^2(\mu_{s_n})}^2-\norma{f_n}_{L^2(\sigma_c)}^2&=\int_{\ab{y}\geq a} 
 \ab{f_n(y)}^2\abs{\frac{1}{\sqrt{\ab{y}^2-s_n^2}}-\frac{1}{\ab{y}}}\d y\\
&\leq 
\norma{f_n}_{L^2(\mu_{s_n})}^2 \NOrma{\frac{\ab{y}-\sqrt{\ab{y}^2-s_n^2}}{\ab{y}}\one_{\{\ab{y}\geq 
		a\}}}_{L_y^\infty(\R^3)}\\
&=\Bigl(1-\sqrt{1-s_n^2a^{-2}}\Bigr)\to 0,
\end{split}
\end{equation}
as $n\to\infty$. Then, \eqref{eq:weak-conv-norm} implies
 \begin{equation}\label{eq:weak_conv_norms}
 \norma{f_n-f}_{L^2(\sigma_c)}^2=\norma{f_n}_{L^2(\mu_{s_n})}^2-\norma{f}_{L^2(\sigma_c)}^2+o(1).
 \end{equation}
Because of conditions (iii) and (iv) and Lemma \ref{lem:pointwise-convergence-to-cone}, $T_{s_n}f_n\to T_cf$ a.e. pointwise in $\R^4$, as $n\to\infty$, and we can apply the Br\'{e}zis-Lieb lemma to the sequence $\{T_{s_n}f_n\}_n\subset L^4(\R^4)$ to obtain
 \[\norma{T_{s_n}f_n-T_cf}_{L^4(\R^4)}^4=\norma{T_{s_n}f_n}_{L^4(\R^4)}^4-\norma{T_cf}_{L^4(\R^4)}^4+o(1).\]

 Because by scaling the norm of the operators $T_{s_n}$ is independent of $n$ (see Section \ref{sec:scaling}) and by (ii)
 $\norma{T_{s_n}f_n}_{L^4(\R^4)}\to\norma{T}$ as $n\to\infty$, we obtain
 \begin{align} 
 \label{eq:where-to-cut}
 \norma{T_{s_n}}^2=\norma{T}^2&=\frac{\norma{T_{s_n}f_n}_{L^4(\R^4)}^2}{\norma{f_n}_{L^2(\mu_{s_n})}^2}+o(1)\\
 \nonumber
&=\frac{(\norma{T_{s_n}f_n-T_cf}_{L^4}^4+\norma{T_cf}_{L^4}^4+o(1))^{1/2}}{\norma{f_n-f}_{L^2(\sigma_c)}^2+\norma
{f}_{L^2(\sigma_c)}^2+o(1)}+o(1)\\
\nonumber
&\leq \frac{\norma{T_{s_n}f_n-T_cf}_{L^4}^2+\norma{T_cf}_{L^4}^2+o(1)}{\norma{f_n-f}_{L^2(\sigma_c)}^2+\norma
{f}_{L^2(\sigma_c)}^2+o(1)}+o(1)\\
\nonumber
&\leq \frac{\norma{T_{s_n}f_n-T_{s_n}f}_{L^4}^2+\norma{T_cf}_{L^4}^2+o(1)}{\norma{f_n-f}_{L^2(\sigma_c)}^2+\norma
{f}_{L^2(\sigma_c)}^2+o(1)}+o(1),
 \end{align}
 where in the last inequality we used the triangle inequality and that $\norma{T_{s_n}f-T_cf}_{L^4}\to 0$ as $n\to\infty$, from Lemma 
 \ref{lem:l4-convergence-to-cone}.
Then
\[\norma{T_{s_n}}^2\leq
\frac{\norma{T_{s_n}}^2\norma{f_n-f}_{L^2(\mu_{s_n})}^2+\norma{T_cf}_{L^4}^2+o(1)}{\norma{f_n-f}_{L^2(\sigma_c)}^2+\norma
{f}_{L^2(\sigma_c)}^2+o(1)}+o(1)\]
and hence
\[\norma{T_{s_n}}^2(\norma{f_n-f}_{L^2(\sigma_c)}^2+\norma
{f}_{L^2(\sigma_c)}^2+o(1))\leq\norma{T_{s_n}}^2\norma{f_n-f}_{L^2(\mu_{s_n})}^2+\norma{T_cf}_{L^4}^2+o(1)\]
which is equivalent to
\[\norma{T_{s_n}}^2\norma{f}_{L^2(\sigma_c)}^2\leq
\norma{T_cf}_{L^4}^2+\norma{T_{s_n}}^2(\norma{f_n-f}_{L^2(\mu_{s_n})}^2-\norma{f_n-f}_{L^2(\sigma_c)}^2)+o(1).
\]
Arguing as in \eqref{eq:convergence_norm_f_n} we obtain $\norma{f_n-f}_{L^2(\mu_{s_n})}^2-\norma{f_n-f}_{L^2(\sigma_c)}^2\to 
0$,
and therefore,
\[\norma{T}=\norma{T_{s_n}}\leq \frac{\norma{T_c f}_{L^4}}{\norma{f}_{L^2(\sigma_c)}}\leq  \norma{T_c}.\]

Finally, if we relax the support condition (iv) to the $\eps$-small norm condition (v), it will be 
enough if in 
\eqref{eq:where-to-cut} we use
\[ \frac{\norma{T_{s_n}f_n}_{L^4(\R^4)}^2}{\norma{f_n}_{L^2(\mu_{s_n})}^2}\leq 
\frac{\norma{T_{s_n}(f_n\one_{\{\ab{y}\geq a\}})}_{L^4(\R^4)}^2}{\norma{f_n\one_{\{\ab{y}\geq 
a\}}}_{L^2(\mu_{s_n})}^2}+C\eps, \]
where $C<\infty$ is independent of $n$ and $a$, together with $f_n\one_{\{\ab{y}\geq a\}}\rightharpoonup f\one_{\{\ab{y}\geq a\}}$ in $L^2(\Gamma^3)$ and $T_{s_n}(f_n\one_{\{\ab{y}\geq a\}})\to T_c(f\one_{\{\ab{y}\geq a\}})$ a.e. in $\R^4$, as $n\to\infty$, the latter property being a consequence of the former and Lemma \ref{lem:pointwise-convergence-to-cone}.
\end{proof}

\section{Proof of Theorem \ref{thm:main-theorem-2}}\label{sec:concluding_compactness}

In previous Sections \ref{sec:no-dichotomy} and \ref{sec:no-vanishing}, we proved
that if $\{f_n\}_n$ is an extremizing sequence for $\overline{T}$, then subsequences of $\{\ab{f_n}^2\}_n$ can not satisfy vanishing nor dichotomy of Lemma \ref{lem:concentration-compactness}, which as we saw, were a consequence of bilinear estimates for $\overline{T}$. In this 
section we prove that, as a consequence of the compactness alternative and Lemma \ref{lem:scaling-convergence-cone} of the previous section, extremizing sequences posses convergent subsequences and, as a result, extremizers exist.

\begin{proof}[Proof of Theorem \ref{thm:main-theorem-2}]
Let $\{f_n\}_{n}\subset L^2(\overline{\mathcal{H}}^3)$ be an $L^2$ normalized complex valued extremizing sequence for $\overline{T}$. After passing to a subsequence if necessary we can assume that alternative (i) 
in Lemma \ref{lem:concentration-compactness} holds for $\{\ab{f_n}^2\}_{n}$, that is,
there exists $\{\ell_n\}_n\subset \N$ with the property that for all $\eps>0$ there exists $R_\eps<\infty$ such that for all $R\geq R_\eps$ and $n\in\N$
\begin{equation}\label{eq:large_dyadic_piece}
\int_{2^{\ell_n-R}\leq\ab{y}\leq 2^{\ell_n+R}}\ab{f_n(y)}^2\d\bar\mu(y)\geq 1-\eps.
\end{equation}
If
there exists a subsequence $\{n_k\}_{k}\subset \N$ such that $\{\ell_{n_k}\}_{k}$ is bounded above, 
then we 
can apply the same method provided in the proof of Proposition \ref{prop:convergence-not-concentration-infinity}
for the upper half of the hyperboloid, $\mathcal H^3$, to conclude that there exists a subsequence  $\{f_{n_k}\}_k$ that satisfies the conclusion of the theorem with all $L_{n_k}$'s equal to the identity matrix. Therefore, in what follows 
we will assume that 
$\ell_{n}\to\infty$ as $n\to\infty$.

Passing to a subsequence if necessary we can assume then that $\{f_n\}_{n}$
satisfies the following: $\norma{f_n}_{L^2}=1$,
$\norma{\overline{T}f_n}_{L^4}\to \overline{\mathbf{H}}_4$ and there exists a sequence $\{\ell_n\}_{n\in\N}\subset \N$ 
such that $\ell_n\to\infty$
as $n\to\infty$ and for any $\eps>0$ there exists $R_\eps<\infty$ such that for all $R\geq R_\eps$ and all $n\in\N$ equation \eqref{eq:large_dyadic_piece} holds.
Therefore, with $R_\eps$ as before, we have that for all $R\geq R_\eps$ there exists $N_n\in[\ell_n-R,\ell_n+R]\cap \N$ such that for all $n\in\N$
\begin{equation*}\label{eq:Nn_large_dyadic_piece}
\int_{2^{N_n}\leq\ab{y}\leq 2^{N_n+1}}\ab{f_n(y)}^2\d\bar\mu(y)\geq \frac{1-\eps}{2R}.
\end{equation*}
Denote $P_N$ the dyadic cut off at scale $2^N$, that is,
$P_Nf(y):=f(y)\mathbbm{1}_{\{2^N\leq \ab{y}< 2^{N+1}\}}$. Using the continuity of $\overline{T}$ and the triangle inequality we obtain 
\begin{align*}
\norma{\overline{T}(P_{N_n}f_n)}_{L^4}&\geq \norma{\overline{T}f_n}_{L^4}-\overline{\mathbf{H}}_4\norma{f_n-P_{N_n}f_n}_{L^2(\bar{\mu})}\geq  \norma{\overline{T}f_n}_{L^4}-\overline{\mathbf{H}}_4\Bigl(1-\frac{1-\eps}{2R}\Bigr)^{1/2}\\
&=\overline{\mathbf{H}}_4-\overline{\mathbf{H}}_4\Bigl(1-\frac{1-\eps}{2R}\Bigr)^{1/2}+o_n(1).
\end{align*}
Choosing $\eps=\eps_0$ close to $0$ and $R=R_{\eps_0}+1$, we obtain a 
sequence $\{N_n\}_{n}\subset \N$, with $\ab{N_n-\ell_n}\leq R_{\eps_0}+1$, so that $N_n\to\infty$ as $n\to\infty$, and a constant $c>0$ 
such that for all $n$ large enough \footnote{By redefining the sequence $\{f_n\}_n$ we will assume that the property holds for all $n\geq 1$.}
\[\norma{P_{N_n}f_n}_{L^2(\bar{\mu})}> c,\quad\norma{\overline{T}(P_{N_n}f_n)}_{L^4}> c.\]
We rescale $f_n$ defining $g_n$ by $g_n(y)=2^{N_n}f(2^{N_n}y)$. Letting $s_n=2^{-N_n}$ we have $s_n\to 0$ as $n\to\infty$, $g_n\in 
L^2(\overline{\mathcal H}^3_{s_n})$,
\begin{align}
&\norma{g_n}_{L^2(\bar\mu_{s_n})}=\norma{f_n}_{L^2(\bar\mu)}=1,\nonumber\\
&\norma{\overline{T}_{s_n}g_n}_{L^4}=\norma{\overline{T}f_n}_{L^4}\to\overline{\mathbf{H}}_4 \text{ as }n\to\infty,\nonumber\\
&\norma{P_1g_n}_{L^2(\bar\mu_{s_n})}=\norma{P_{N_n}f_n}_{L^2(\bar\mu)}>c\label{eq:lower_bound_P1g}\text{ and}\\
&\norma{\overline{T}_{s_n}(P_1g_n)}_{L^4}=\norma{\overline{T}(P_{N_n}f_n)}_{L^4}>c,\label{eq:lower_bound_TP1}
\end{align}
as obtained by simple rescaling (see Section
\ref{sec:scaling}). Moreover, from \eqref{eq:large_dyadic_piece} for any small $\eps>0$, $R>2R_\eps$ and $n\in\N$
\begin{equation}\label{eq:localization_gn}
\int_{2^{-R}\leq\ab{y}\leq 2^{R}}\ab{g_n(y)}^2 \d\bar\mu_{s_n}(y)\geq 1-\eps,
\end{equation}
so that the $g_n$'s are "localized at scale $1$".
Using Lemma \ref{lem:nearly-extremals-hyp} applied to 
$\overline{T}_{s_n}$ and $P_1g_n$, which is possible given \eqref{eq:lower_bound_P1g} and \eqref{eq:lower_bound_TP1}, we obtain that for 
all $n\in\N$ there 
exist caps $\cp_n\subset \Hyp_{s_n}$, which we may consider all to be contained in the upper half , $\hyp_{s_n}$,  possibly after passing to a subsequence\footnote{Otherwise we reflect the $f_n$'s and $g_n$'s with respect to the origin, as necessary, by considering the sequences $\{L^*f_n\}_n$ and $\{L^*g_n\}_n$ where $L\in\mathcal{L}$ is the reflection map $L(x,t)=(-x,-t)$}, $\cp_n =
[1,2]\times \sphcp_n\subset \hyp_{s_n}$, for some spherical caps $\sphcp_n\subseteq \Sph^2$, such that
\[ \int_{\cp_n}\ab{g_n(y)} \d\bar\mu_{s_n}(y)=\int_{\cp_n}\ab{P_1g_n(y)} \d\bar\mu_{s_n}(y)\geq c\bar\mu_{s_n}(\cp_n)^{1/2}\norma{P_1g_n}_{L^2(\bar\mu_{s_n})}\gtrsim\bar\mu_{s_n}(\cp_n)^{1/2}, \]
as a consequence of \eqref{eq:hyp_cap_refinement_6}. Equivalently
\begin{equation}\label{eq:lower_l1_fn}
\int_{2^{N_n}\cp_n}\ab{f_n(y)} \d\bar\mu(y)\gtrsim\bar\mu(2^{N_n}\cp_n)^{1/2}.
\end{equation}
Let $\alpha=\limsup_{n\to\infty} \bar\mu_{s_n}(\cp_n)$. Two 
cases arise.\\

\noindent\textbf{Case 1:} $\alpha>0$. Passing to a subsequence if necessary, we can assume that there 
exists a constant $c>0$ such that for all $n$
\[ \int_{\cp_n}\ab{g_n(y)} \d\bar\mu_{s_n}(y)\geq c>0. \]
We can view $g_n$ as a 
function on the double cone via the usual identification using the orthogonal projection onto 
$\R^3$, where we extend it to be zero in the region where $\ab{y}\leq s_n$. Since 
$\norma{g_n}_{L^2(\bar\sigma_c)}\leq \norma{g_n}_{L^2(\bar\mu_{s_n})}=1$ and
\begin{equation}\label{eq:lowerBoundL1Cone}
0<c<\int_{\cp_n}\ab{g_n(y)} \d\bar\mu_{s_n}(y)\lesssim\int_{\cp_n}\ab{g_n(y)} \d\bar\sigma_c(y),
\end{equation}
for all $n$ large enough (as a consequence of \eqref{eq:localization_gn}), there is weak 
convergence of $\{\ab{g_n}\}_n$ in 
$L^2(\bar\sigma_c)$ after the possible extraction of a subsequence,
$\ab{g_n}\rightharpoonup g$, for some $g\in L^2(\bar\sigma_c)$ which satisfies $g\neq 0$ by \eqref{eq:lowerBoundL1Cone}. Inequality \eqref{eq:localization_gn} implies that 
\[\lim_{a\to 0^+}\sup_{n\in\N}\norma{g_n\one_{\{\ab{y}\leq a\}}}_{L^2(\bar{\mu}_{s_n})}=0.\]
Because $\norma{\overline{T}_{s_n}(g_n)}_{L^4}\leq \norma{\overline{T}_{s_n}(\ab{g_n})}_{L^4}$, it is then also the case that $\norma{\overline{T}_{s_n}(\ab{g_n})}_{L^4}\to \overline{\mathbf{H}}_4$, so that we can use part (v) of Lemma \ref{lem:scaling-convergence-cone} applied to $\{\ab{g_n}\}_n$ to conclude
\[\norma{\overline{T}}\leq\norma{\overline{T}_c},\]
which is in contradiction with 
Proposition \ref{prop:comparison-double-cone}. Therefore, this case 
does not arise.\\

\noindent\textbf{Case 2:} $\alpha=0$. 
Let $\{\gamma_n\}_n\subset [0,\pi]$ and $\{R_n\}_n\subset SO(3)$ be such that
\begin{multline*}
R_n^{-1}(\cp_n)=\{(r\omega,\sqrt{r^2-s_n^2})\colon 1\leq r\leq 2,\\
\omega=(\cos\vphi,\cos\te\sin\vphi,\sin\te\sin\vphi),\te\in[0,2\pi],\,\vphi\in[0,\gamma_n] \}.
\end{multline*}
The condition $\alpha=0$ implies $\gamma_n\to 0$ as $n\to\infty$. Let $\beta=\limsup_{n\to\infty}\bar{\mu}(2^{N_n}\cp_n)=\limsup_{n\to\infty}2^{2N_n}\bar{\mu}_{s_n}(\cp_n)$. Two subcases arise.

\textbf{Subcase 2a:} $\beta<\infty$. This implies that the sequence $\{\bar{\mu}(2^{N_n}\cp_n)\}_n$ is bounded. We may assume that the angles $\gamma_n$ are less that $\pi/2$ as $\{\gamma_n\}_n$ tends to zero. Form Lemma \ref{lem:measure_bounded_cap} with $s=1$, there exists $\{t_n\}_n\subset [0,1)$ such that the caps $\{L^{-t_n}R_n^{-1}(2^{N_n}\cp_n):n\in\N\}$ are contained in a fixed bounded ball of $\R^4$.
It therefore follows from \eqref{eq:lower_l1_fn} and the Cauchy--Schwarz inequality that $\{(R_nL^{t_n})^*f_n\}_n\subset L^2(\Hyp)$  is an extremizing sequence with $L^2$ norm uniformly bounded below by a constant $c>0$ in a fixed ball. We can then conclude the precompactness modulo characters of the sequence $\{(R_nL^{t_n})^*f_n\}_n$ using the 
argument in the proof of Proposition \ref{prop:convergence-not-concentration-infinity}.

\textbf{Subcase 2b:} $\beta=\infty$. From \eqref{eq:measure_asymp} in Lemma \ref{lem:measure_bounded_cap} with $s=1$, after passing to a subsequence if necessary, $\lim_{n\to\infty}2^{2N_n}\sin^2(\gamma_n)=\infty$.
Set $t_n=\cos\gamma_n$, so that $t_n\to 1$ as 
$n\to\infty$. From Lemma \ref{scaling-cap-hyp} with $s=s_n$, the set $\widetilde \cp_n:=L_{t_n}^{-1}R_n^{-1}(\cp_n)\subset \Hyp_{s_n(1-t_n^2)^{-1/2}}$ 
satisfies, for all $n$ large enough for which $2^{2N_n}\sin^2(\gamma_n)\geq 8$ and $\gamma_n\leq\pi/3$,
\[ \bar\mu_{\frac{s_n}{\sqrt{1-t_n^2}}}(\widetilde\cp_n)\geq \frac{\pi}{2}\quad\text{ and }\quad \widetilde\cp_n\subseteq
[\tfrac{7}{16},\tfrac{33}{16}]\times \Sph^2.\]
Set $a_n=s_n(1-t_n^2)^{-1/2}=(2^{N_n}\sin\gamma_n)^{-1}\to 0$, as $n\to\infty$. Let $\tilde{f}_n=(R_nL^{t_n})^\ast f_n$ so that $\{\tilde f_n\}_n\subset L^2(\Hyp)$ is also an $L^2$-normalized extremizing sequence which satisfies
\[ \int_{a_n^{-1}\widetilde \cp_n}\ab{\tilde f_n(y)} \d\bar\mu(y)\geq c\bar\mu(a_n^{-1}\widetilde 
\cp_n)^{1/2},\quad \int_{a_n^{-1}\widetilde \cp_n}\ab{\tilde f_n(y)}^2 \d\bar\mu(y)\geq c^2. \]
and $a_n^{-1}\widetilde\cp_n\subseteq[\frac{7}{16a_n},\frac{33}{16a_n}]\times\Sph^2$. 

Define the rescale $\tilde g_n(\cdot):=a_n^{-1}\tilde f_n(a_n^{-1}\,\cdot)$, so that for each $n$ we have
$\tilde g_n\in L^2(\Hyp_{a_n})$, $\norma{\tilde g_n}_{L^2(\Hyp_{a_n})}=1$ and
\[ \int_{\widetilde \cp_n}\ab{\tilde g_n(y)} \d\bar\mu_{a_n}(y)\geq c\bar\mu_{a_n}(\widetilde 
\cp_n)^{1/2}>c'>0. 
\]
We are almost in the same situation as in Case 1, but we need the analog of \eqref{eq:localization_gn} for the sequence $\{\tilde g_n\}_n$. 
After passing to a subsequence if necessary, $\{\tilde f_n \}_n$ satisfies the \textit{compactness} alternative in Lemma \ref{lem:concentration-compactness}. Denoting $\{\tilde \ell_n\}_n$ the corresponding sequence associated to $\{\tilde f_n\}_n$ as in \eqref{eq:large_dyadic_piece} we then necessarily have that $\{{\tilde \ell_n}-\log_2(a_n^{-1})\}_n$ is bounded.
This implies the desired analog of \eqref{eq:localization_gn} for $\{\tilde g_n\}_n$. Therefore the argument in Case 1 applies showing that this subcase does not arise.

As a result, only Subcase 2a of Case 2 is possible, proving the theorem.
\end{proof}

\section{Scaling}\label{sec:scaling}

Here we record scaling properties of the family of operators $\{T_s\}_{s>0}$. Recall from Section \ref{sec:calculation_convolution} that for $s>0$, $\mathcal
H^3_s=\{(y,\sqrt{\ab{y}^2-s^2}):y\in\R^3\}$, equipped with the measure $\mu_s$ with density
$\d\mu_s(y,t)=\one_{\{\ab{y}>s\}}\delta(t-\sqrt{\ab{y}^2-s^2})\frac{\d y\d t}{\sqrt{\ab{y}^2-s^2}}$.

The operator $T_s$, defined on $\mathcal S(\R^3)$, is given by 
 \[T_sf(x,t)=\widehat{f\mu_s}(-x,-t)=\int_{y\in\R^3,\ab{y}\geq s} e^{ix\cdot
y}e^{it\sqrt{\ab{y}^2-s^2}}f(y)\frac{\d y}{\sqrt{\ab{y}^2-s^2}}.\]
We want to study the scaling of the quantity $\mathbf{H}_{p,s}$ defined by
\[ \mathbf{H}_{p,s}:=\sup_{0\neq f\in L^2(\mathcal H^3_s)}\frac{\norma{T_sf}_{L^p(\R^4)}}{\norma{f}_{L^2(\mathcal H^3_s)}}. \]
Changing variables 
$y\rightsquigarrow sy$ in the expression defining $Tf(x,t)=T_1f(x,t)$ we obtain
\begin{align*}
 Tf(x,t)&=\int_{y\in\R^3,\ab{y}\geq 1} e^{ix\cdot y}e^{it\sqrt{\ab{y}^2-1}}f(y)\frac{\d 
 y}{\sqrt{\ab{y}^2-1}}\\
 &=\int_{y\in\R^3,\ab{y}\geq s} e^{i s^{-1}x\cdot
y}e^{it\sqrt{s^{-2}\ab{y}^2-1}}f(s^{-1}y)\frac{s^{-3}\d y}{\sqrt{s^{-2}\ab{y}^2-1}}\\
 &=s^{-1}\int_{y\in\R^3,\ab{y}\geq s} e^{is^{-1}x\cdot
y}e^{is^{-1}t\sqrt{\ab{y}^2-s^2}}s^{-1}f(s^{-1}y)\frac{\d y}{\sqrt{\ab{y}^2-s^2}},
\end{align*}
from where $sTf(sx,st)=T_s(s^{-1}f(s^{-1}\cdot))(x,t)$ and it follows that
\[s^{1-4/p}\norma{Tf}_{L^p(\R^{4})}=\norma{T_ss^{-1}f(s^{-1}\cdot)}_{L^p(\R^{4})}.\]
On the other hand 
\begin{align*}
 \int_{y\in\R^3,\,\ab{y}\geq 1} \ab{f(y)}^q\frac{\d 
 y}{\sqrt{\ab{y}^2-1}}&=\int_{y\in\R^3,\,\ab{y}\geq s} 
 \ab{f(s^{-1}y)}^q 
 \frac{s^{-3}\d 
 y}{\sqrt{s^{-2}\ab{y}^2-1}}\\
 &=\int_{y\in\R^3,\,\ab{y}\geq s} \ab{s^{-2/q}f(s^{-1}y)}^q\frac{\d y}{\sqrt{\ab{y}^2-s^2}},
\end{align*}
that is $\norma{f}_{L^q(\mu)}=\norma{s^{-2/q}f(s^{-1}\cdot)}_{L^q(\mu_s)}$. Thus
\begin{align*}
 s^{1-4/p}\norma{Tf}_{L^p(\R^{4})}\norma{f}_{L^2(\mu)}^{-1}&=\norma{T_ss^{-1}f(s^{-1}\cdot)}_{L^p(\R^{4})}
\norma{s^{-1}f(s^{-1}\cdot)}_{L^2(\mu_s)}^{-1},
\end{align*}
and it follows that for all $s>0$
\begin{equation*}
% \label{eq:equality-of-constants}
 \mathbf{H}_{p,s}=s^{1-4/p}\mathbf{H}_{p}.
\end{equation*}
In particular, if $p=4$, 
\[ \mathbf{H}_{4,s}=\mathbf{H}_{4}, \]
for all $s>0$.

\appendix
\section{Computation of a limit}\label{app:appendix}

Let
\begin{align*}
I(a)&=16\pi^3 
\int_0^{\infty}e^{-a\tau}\Bigl(\tau^2\sqrt{\tau^2+4}-\frac{2}{3}(\tau^2+4)\sqrt{\tau^2+1}
+\frac{8}{3}\\
&\qquad\qquad\qquad+2\tau \log(\tau+\sqrt{\tau^2+1})\Bigr)\d\tau\\
\end{align*}
and
\[ II(a)=16\pi^2\biggl(\int_0^\infty e^{-a\tau}\sqrt{\tau^2+1}\d \tau\biggr)^2. \]

The ratio $I(a)/II(a)$ appeared in the proof of Proposition \ref{prop:comparison-cone} while 
establishing that the best 
constant for the hyperboloid $\hyp$ is 
strictly greater than the best constant for the cone $\Gamma^3$ in their respective $L^2\to L^4(\R^4)$ adjoint Fourier restriction
inequalities. The 
purpose of this appendix is to prove the 
following lemma.

\begin{lemma}\label{lem:asymptotics-hyp}
	\begin{align*}
	&\lim_{a\to 0^+}\frac{I(a)}{II(a)}=2\pi,\quad\lim_{a\to 0^+}\frac{\d}{\d 
	a}\frac{I(a)}{II(a)}=0,\quad\lim_{a\to 0^+}\frac{\d^2}{\d a^2}\frac{I(a)}{II(a)}=0
\intertext{and}
	&\lim_{a\to 0^+}\frac{\d^3}{\d a^3}\frac{I(a)}{II(a)}=8\pi.
	\end{align*}
	 Therefore there exists $a_0>0$ such that
	 \[ \frac{I(a)}{II(a)}>2\pi, \]
	  for all $0<a<a_0$.
\end{lemma}

\begin{figure}
	\centering
	\includegraphics[width=12cm, height=7cm]{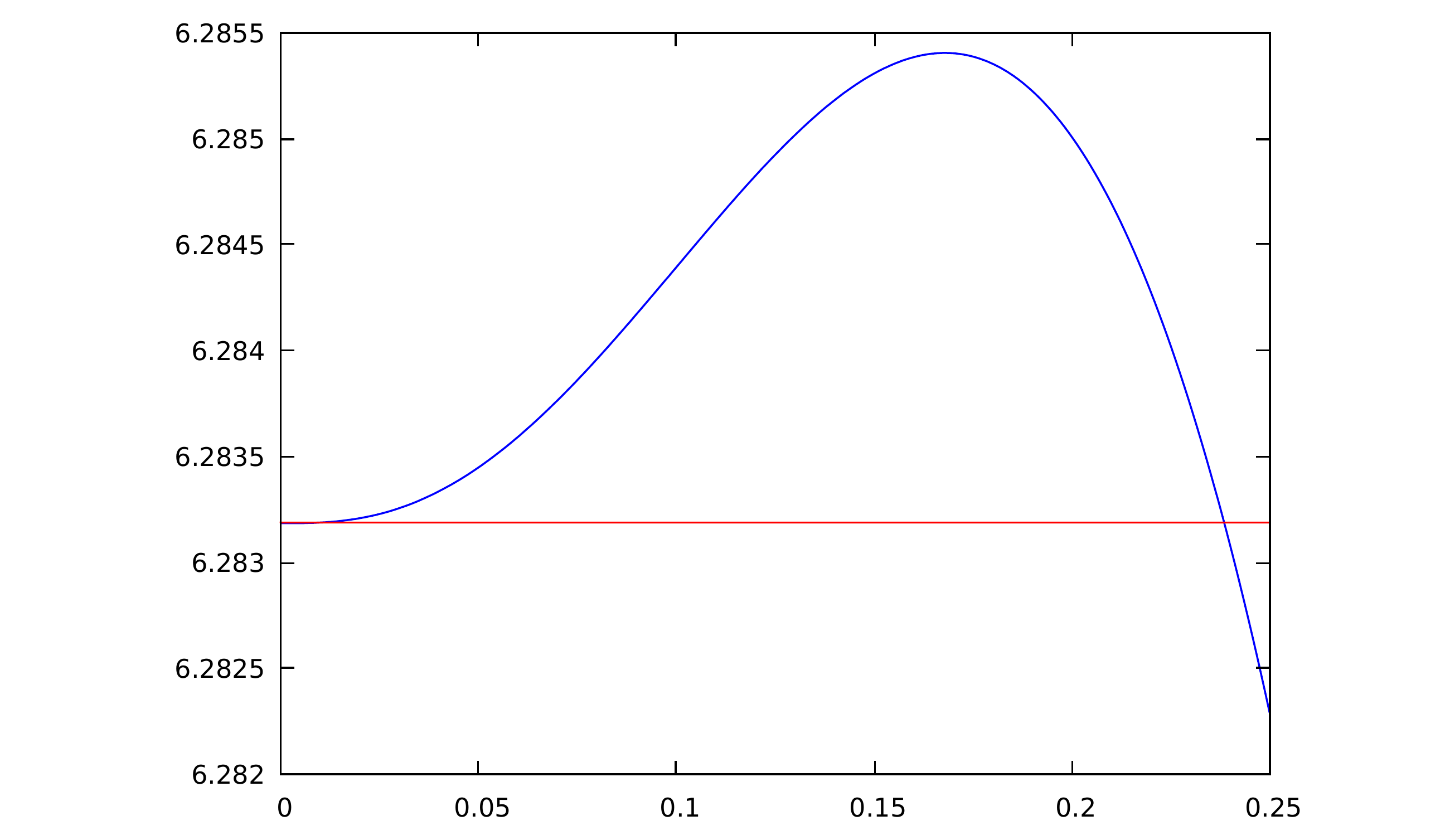}
	\caption{Graph of the ratio $I(a)/II(a)$ and the constant $2\pi$ for $0<a<0.25$, illustrating the content of Lemma \ref{lem:asymptotics-hyp}.}
	\label{fig:comparison_cone}
\end{figure}

Throughout this section we use the asymptotic notation $o_a(f(a))$ and $O_a(f(a))$ as $a\to 0^+$ in 
the usual 
way, namely $g(a)=o_a(f(a))$ if $g(a)/f(a)\to 0$ as $a\to 0^+$, and $g(a)=O_a(f(a))$ if there 
exists a constant $C$, independent of $a$, such that $\ab{g(a)}\leq C\ab{f(a)}$ for all $a>0$ small 
enough.

Changing variable $u=a\tau$ we obtain
\begin{align*}
I(a)&=\frac{16\pi^3}{a^4} 
\int_0^{\infty}e^{-u}\Bigl(u^2\sqrt{u^2+4a^2}-\frac{2}{3}(u^2+4a^2)\sqrt{u^2+a^2}
+\frac{8a^3}{3}\\
&\qquad\qquad+ 2a^2u\,\bigl( \log(u+\sqrt{u^2+a^2})-\log(a)\bigr)\Bigr)\d\tau\\
&=\frac{16\pi^3}{a^4}\biggl( 
\int_0^{\infty}e^{-u}\Bigl(u^2\sqrt{u^2+4a^2}-\frac{2}{3}(u^2+4a^2)\sqrt{u^2+a^2}\\
&\qquad\qquad+ 2a^2u\, 
\log(u+\sqrt{u^2+a^2})\Bigr)\d\tau+\frac{8a^3}{3}-2a^2\log(a)\biggr)
\intertext{and} 
II(a)
&=\frac{16\pi^2}{a^4}\biggl(\int_0^\infty 
e^{-u}\sqrt{u^2+a^2}\d u\biggr)^2.
\end{align*}
Using the Dominated Convergence Theorem it is direct to check that
\[ \lim_{a\to 0^+}a^4I(a)=32\pi^3\,\text{ and }\,\lim_{a\to 0^+}a^4II(a)=16\pi^2, \]
so that
\[ \lim_{a\to 0^+}\frac{I(a)}{II(a)}=2\pi. \]
To address the limit of the derivatives of the ratio $I(a)/II(a)$ it will be convenient to 
introduce a rescaling. Let
\begin{align*}
&N(a):=a^{4/3}I(a^{1/3})=16\pi^3\biggl( 
\int_0^{\infty}e^{-u}\Bigl(u^2\sqrt{u^2+4a^{2/3}}-\frac{2}{3}(u^2+4a^{2/3})\sqrt{u^2+a^{2/3}}\\
&\qquad\qquad\qquad\qquad\qquad\qquad+ 2a^{2/3}u\, 
\log(u+\sqrt{u^2+a^{2/3}})\Bigr)\d\tau+\frac{8a}{3}-\frac{2}{3}a^{2/3}\log(a)\biggr)
\intertext{and}
& D(a):=a^{4/3}II(a^{1/3})=16\pi^2\biggl(\int_0^\infty 
e^{-u}\sqrt{u^2+a^{2/3}}\d u\biggr)^2. 
\end{align*}
As we already know, and can readily check, $N(a)\to 32\pi^3,\, D(a)\to 16\pi^2$ and $N(a)/D(a)\to 2\pi$ as $a\to 0^+$. The remaining properties of the 
derivatives of $I(a)/II(a)$ in Lemma \ref{lem:asymptotics-hyp} will follow if we show that 
$\frac{\d}{\d a}(N(a)/D(a))\to \frac{4\pi}{3}$ as $a\to 0^+$.

In what follows we write $(\cdot)'$ as a short for the derivative with respect to $a$. Given that 
both $N'(a)$ and $D'(a)$ diverge to $+\infty$ as $a\to 0^+$ it will be convenient to write 
the 
derivative of $N(a)/D(a)$ in the following way
\begin{align}\label{eq:derivative-zero-split}
\begin{split}
\frac{\d}{\d a}\frac{N(a)}{D(a)}
&=\frac{16\pi^2N'(a)-32\pi^3 D'(a)}{D(a)^2}\\
&\qquad+\frac{(D(a)-16\pi^2)\,N'(a)-(N(a)-32\pi^3)\,D'(a)}{D(a)^2}.
\end{split}
\end{align}
We have the following lemma.

\begin{lemma}
	\mbox{}
	\begin{itemize}
		\item[(i)] $\ds\lim_{a\to 0^+}\frac{\d}{\d a}\frac{N(a)}{D(a)}=\frac{4\pi}{3}$.
		\item[(ii)] As $a\to 0^+$,
		\[ N'(a)=O_a\Bigl(\frac{\log a}{a^{1/3}}\Bigr)\,\text{ and }\, 
		D'(a)=O_a\Bigl(\frac{\log a}{a^{1/3}}\Bigr). \]
		\item[(iii)] $\ds\lim_{a\to 0^+}(N(a)-32\pi^3)\, D'(a)=0\,$  and  $\,\ds\lim_{a\to 
		0^+}(D(a)-16\pi^2)\,N'(a)=0.$
	\end{itemize}
\end{lemma}

\begin{proof}
In the course of the proof of this lemma we will make repeated use of the asymptotic behavior of some integrals as contained in Lemma \ref{lem:asymptotic_integrals} below. We start with property (ii). For $a>0$ the derivative of $N$ is as follows,
\begin{align}
\nonumber
N'(a)&=16\pi^3\biggl( 
\int_0^{\infty}e^{-u}\Bigl(u^2\frac{4}{3a^{1/3}\sqrt{u^2+4a^{2/3}}}-\frac{16}{9a^{1/3}}\sqrt{u^2+a^{2/3}}\\
\nonumber
&\qquad-\frac{2}{9}(u^2+4a^{2/3})\frac{1}{a^{1/3}\sqrt{u^2+a^{2/3}}}+ \frac{4}{3a^{1/3}}u\, 
\log(u+\sqrt{u^2+a^{2/3}})\\
\nonumber
&\qquad+\frac{2}{3}a^{1/3}u\,\frac{1}{(u+\sqrt{u^2+a^{2/3}})\sqrt{u^2+a^{2/3}}}\Bigr)
\d u+\frac{8}{3}-\frac{4}{9a^{1/3}}\log(a)-\frac{2}{3a^{1/3}}\biggr)\\
\label{eq:formula-N-prima}
&=16\pi^3\biggl( \frac{8}{3}-\frac{4}{3a^{1/3}}-\frac{4}{9a^{1/3}}\log(a)+
\frac{4}{3a^{1/3}}\int_0^\infty e^{-u}u\log(u+\sqrt{u^2+a^{2/3}})\d u\biggr)\\
\nonumber
&\qquad+o_a(1)\\
\nonumber
&=O_a\Bigl(\frac{\log a}{a^{1/3}}\Bigr),
\end{align}
where we used \eqref{eq:asymp_2},\eqref{eq:asymp_3},\eqref{eq:asymp_1},\eqref{eq:asymp_4} and \eqref{eq:asymp_5}.
The derivative of the function $D$ is as follows
\begin{equation*}
D'(a)=\frac{32\pi^2}{3}\int_{0}^{\infty} e^{-u}\sqrt{u^2+a^{2/3}}\d u\cdot 
\int_{0}^{\infty}e^{-u}\frac{1}{a^{1/3}\sqrt{u^2+a^{2/3}}}\d u,
\end{equation*}	
so that \eqref{eq:asymp_7} and \eqref{eq:asymp_2} imply
\begin{equation*}
D'(a)=O_a\Bigl(\frac{1}{a^{1/3}}\Bigr)\,O_a(\log a)=O_a\Bigl(\frac{\log a}{a^{1/3}}\Bigr),
\end{equation*}
and more explicitly using \eqref{eq:explicit_asymp_7}, as we will need later, 
\begin{equation}\label{eq:formula-D-prima}
D'(a)=\frac{32\pi^2}{3a^{1/3}}\biggl(\int_{0}^{\infty}e^{-u}\log(u+\sqrt{u^2+a^{2/3}})\d 
u-\frac{1}{3}\log a\biggr)
+o_a(1).
\end{equation}
We now turn to the proof of part (iii). Using that $\int_0^\infty e^{-u}u^3\d u=6$ we can write 
\begin{align*}
N(a)-32\pi^3&=16\pi^3\biggl( 
\int_0^{\infty}e^{-u}\Bigl(u^2(\sqrt{u^2+4a^{2/3}}-u)-\frac{2}{3}u^2(\sqrt{u^2+a^{2/3}}-u)\\
&\qquad\qquad-\frac{8}{3}a^{2/3}\sqrt{u^2+a^{2/3}}+ 2a^{2/3}u\, 
\log(u+\sqrt{u^2+a^{2/3}})\Bigr)\d u\\
&\qquad\qquad+\frac{8a}{3}-\frac{2}{3}a^{2/3}\log(a)\biggr)\\
&=16\pi^3a^{1/3}\biggl( 
\int_0^{\infty}e^{-u}\Bigl(u^2\frac{4a^{1/3}}{\sqrt{u^2+4a^{2/3}}+u}
-\frac{2}{3}u^2\frac{a^{1/3}}{\sqrt{u^2+a^{2/3}}+u}\\
&\qquad\qquad-\frac{8}{3}a^{1/3}\sqrt{u^2+a^{2/3}}+ 2a^{2/3}u\, 
\log(u+\sqrt{u^2+a^{2/3}})\Bigr)\d u\\
&\qquad\qquad+\frac{8a^{2/3}}{3}-\frac{2}{3}a^{1/3}\log(a)\biggr)\\
&=O_a(a^{2/3}\log a).
\end{align*}
Then
\begin{equation*}
(N(a)-32\pi^3)\cdot D'(a)=O_a(a^{2/3}\log a)\,O_a\Bigl(\frac{\log a}{a^{1/3}}\Bigr) 
=O_a(a^{1/3}\log^2 a)=o_a(1).
\end{equation*}
On the other hand
\begin{align*}
D(a)-16\pi^2&=16\pi^2\biggl(\int_0^\infty 
e^{-u}\sqrt{u^2+a^{2/3}}\d u+1\biggr)\biggl(\int_0^\infty 
e^{-u}\sqrt{u^2+a^{2/3}}\d u-1\biggr)\\
&=O_a(1)\biggl(\int_0^\infty 
e^{-u}(\sqrt{u^2+a^{2/3}}-u)\d u\biggr)\\
&=O_a(1)\biggl(\int_0^\infty 
e^{-u}\frac{a^{2/3}}{\sqrt{u^2+a^{2/3}}+u}\d u\biggr)\\
&=O_a(a^{2/3}\log a),
\end{align*}
where in the last line we used \eqref{eq:asymp_8}.
Then
\begin{equation*}
(D(a)-16\pi^2)\cdot N'(a)=O_a(a^{2/3}\log a)\, O_a\Bigl(\frac{\log 
a}{a^{1/3}}\Bigr)=O_a(a^{1/3}\log 
^2 a)=o_a(1).
\end{equation*}	
We now turn to the proof of (i). By (iii), the limit as $a\to 0^+$ of the second summand on the right hand side of \eqref{eq:derivative-zero-split} equals zero. We proceed to calculate the limit of the first summand. 
Combining \eqref{eq:formula-N-prima} and 
\eqref{eq:formula-D-prima} we obtain 
\begin{align*}
 16\pi^2 N'(a)&-32\pi^3D'(a)\\
&=\frac{8(16)^2\pi^5}{3}-\frac{4(16)^2\pi^5}{3a^{1/3}}\\
 &\qquad+\frac{(32)^2\pi^5}{3a^{1/3}}
\int_{0}^{\infty}e^{-u}(u-1)\log(u+\sqrt{u^2+a^{2/3}})\d u+o_a(1)\\
&=\frac{2(32)^2\pi^5}{3}+\frac{(32)^2\pi^5}{3a^{1/3}}
\int_{0}^{\infty}e^{-u}\Bigl((u-1)\log(u+\sqrt{u^2+a^{2/3}})-1\Bigr)\d u\\
&\qquad+o_a(1).
\end{align*}
Using \eqref{eq:asymp_9} to treat the integral in the previous expression we obtain
\begin{align*}
\lim_{a\to 0^+}(16\pi^2N'(a)-32\pi^3D'(a))=\frac{(32)^2\pi^5}{3},
\end{align*}
therefore
\[ \lim_{a\to 0^+}\frac{\d}{\d 
a}\frac{N(a)}{D(a)}=\frac{(32)^2\pi^5}{3(16\pi^2)^2}=\frac{4\pi}{3}. \]
\end{proof}

Finally, we state the asymptotic behavior of the many integrals used during the proof of the previous lemma.

\begin{lemma}\label{lem:asymptotic_integrals}
	We have the following identities as $a\to0^+$ 
	\begin{align}
	\label{eq:asymp_7}
	\int_{0}^{\infty}e^{-u}\frac{1}{\sqrt{u^2+a^{2/3}}}\d u&=O_a(\log a),\\
		\label{eq:asymp_2}
	\int_0^{\infty}e^{-u}\frac{\sqrt{u^2+a^{2/3}}}{a^{1/3}}\d u
	&=\frac{1}{a^{1/3}}+O_a(a^{1/3}\log a),\\
	\label{eq:asymp_6}
	\int_0^{\infty}e^{-u}\frac{u\sqrt{u^2+a^{2/3}}}{a^{1/3}}\d u
	&=\frac{2}{a^{1/3}}+O_a(a^{1/3}),\\
	\label{eq:asymp_1}
	\int_0^{\infty}e^{-u}\frac{u^2}{a^{1/3}\sqrt{u^2+4a^{2/3}}}\d u
	&=\frac{1}{a^{1/3}}+O_a(a^{1/3}\log a),\\
	\label{eq:asymp_3}
	\int_0^{\infty}e^{-u}\frac{u^2+4a^{2/3}}{a^{1/3}\sqrt{u^2+a^{2/3}}}\d u
	&=\frac{1}{a^{1/3}}+O_a(a^{1/3}\log a),\\
	\label{eq:asymp_8}
	\int_0^\infty e^{-u}\frac{a^{2/3}}{u+\sqrt{u^2+a^{2/3}}}\d u
	&=O_a(a^{2/3}\log a),\\
	\label{eq:asymp_4}
	\int_0^{\infty}e^{-u}\,\frac{a^{1/3}u}{(u+\sqrt{u^2+4a^{2/3}})\sqrt{u^2+4a^{2/3}}}\d u &=O_a(a^{1/3}\log a),\\
	\label{eq:asymp_5}
	\int_0^{\infty}e^{-u}\frac{u}{a^{1/3}}\, \log(u+\sqrt{u^2+a^{2/3}})\d u &=O_a\Bigl(\frac{1}{a^{1/3}}\Bigr),\\
	\label{eq:asymp_9}
	\frac{1}{a^{1/3}}\int_{0}^{\infty}e^{-u}\big((u-1)\log (u+\sqrt{u^2+a^{2/3}})-1\big)\d u&=-1+o_a(1).
	\end{align}
\end{lemma}

\begin{proof}
	The identities are elementary but we choose to give details for the sake of completeness. 
	
	\underline{Verification of \eqref{eq:asymp_7} and \eqref{eq:asymp_2}:} Integration by parts shows that
	\begin{equation}\label{eq:explicit_asymp_7}
	\int_{0}^{\infty}e^{-u}\frac{1}{\sqrt{u^2+a^{2/3}}}\d u
	=\int_{0}^{\infty}e^{-u}\log(u+\sqrt{u^2+a^{2/3}})\d u-\frac{1}{3}\log a=O_a(1)+O_a(\log a),
	\end{equation} 
	and
	\begin{align*}
	\frac{1}{a^{1/3}}\int_{0}^{\infty} &e^{-u}\sqrt{u^2+a^{2/3}}\d u
	=\frac{1}{2a^{1/3}}\int_{0}^{\infty} e^{-u}(a^{2/3}\log(u+\sqrt{u^2+a^{2/3}})
	\\
	&\qquad\qquad\qquad\qquad\qquad\qquad+u\sqrt{u^2+a^{2/3}}-\frac{1}{3}a^{2/3}\log a)\d u\\
	&=\frac{1}{2}\int_{0}^{\infty} e^{-u}(a^{1/3}\log(u+\sqrt{u^2+a^{2/3}})
	+\frac{u}{a^{1/3}}\sqrt{u^2+a^{2/3}}\\
	&\qquad\qquad-\frac{1}{3}a^{1/3}\log a)\d u\\
	&=O_a(a^{1/3})+O_a(a^{1/3}\log a)+\frac{1}{a^{1/3}}+\frac{1}{2}\int_{0}^{\infty} e^{-u}\frac{u}{a^{1/3}}(\sqrt{u^2+a^{2/3}}-u)\d u\\
	&=\frac{1}{a^{1/3}}+O_a(a^{1/3}\log a)+\frac{a^{1/3}}{2}\int_{0}^{\infty} e^{-u}\frac{u}{\sqrt{u^2+a^{2/3}}+u}\d u\\
	&=\frac{1}{a^{1/3}}+O_a(a^{1/3}\log a)+O_a(a^{1/3}).
	\end{align*}

	\noindent\underline{Verification of \eqref{eq:asymp_6}:} Using that $\int_0^\infty e^{-u}u^2\d u=2$ we have
	\begin{align*}
	\int_0^{\infty}e^{-u}\frac{u\sqrt{u^2+a^{2/3}}}{a^{1/3}}\d u&=\frac{2}{a^{1/3}}+\frac{1}{a^{1/3}}\int_0^{\infty}e^{-u}u(\sqrt{u^2+a^{2/3}}-u)\d u\\
	&=\frac{2}{a^{1/3}}+a^{1/3}\int_0^{\infty}e^{-u}\frac{u}{\sqrt{u^2+a^{2/3}}+u}\d u\\
	&=\frac{2}{a^{1/3}}+O_a(a^{1/3}).
	\end{align*}

	\noindent\underline{Verification of \eqref{eq:asymp_1}:}
	\begin{align*}
	\int_0^{\infty}e^{-u}\frac{u^2}{a^{1/3}\sqrt{u^2+4a^{2/3}}}\d u&=\int_0^{\infty}e^{-u}\frac{\sqrt{u^2+4a^{2/3}}}{a^{1/3}}\d u-4a^{1/3}\int_0^{\infty}e^{-u}\frac{1}{\sqrt{u^2+4a^{2/3}}}\d u\\
	&=\frac{1}{a^{1/3}}+O_a(a^{1/3}\log a)+a^{1/3}O_a(\log a),
	\end{align*}
	where we used \eqref{eq:asymp_7} and \eqref{eq:asymp_2}.

	\noindent\underline{Verification of \eqref{eq:asymp_3}:}
	\begin{align*}
	\int_0^{\infty}e^{-u}\frac{u^2+4a^{2/3}}{a^{1/3}\sqrt{u^2+a^{2/3}}}\d u&=\frac{1}{a^{1/3}}\int_0^{\infty}e^{-u}\sqrt{u^2+a^{2/3}}\d u+3a^{1/3}\int_0^{\infty}e^{-u}\frac{1}{\sqrt{u^2+a^{2/3}}}\d u\\
	&=\frac{1}{a^{1/3}}+O_a(a^{1/3}\log a)+O_a(a^{1/3}),
	\end{align*}
	where in the last line we used \eqref{eq:asymp_7} and \eqref{eq:asymp_2}. 
	
	\noindent\underline{Verification of \eqref{eq:asymp_8}:}
	\begin{align*}
	\int_0^\infty e^{-u}\frac{a^{2/3}}{u+\sqrt{u^2+a^{2/3}}}\d u&= \int_0^\infty 
	e^{-u}(\sqrt{u^2+a^{2/3}}-u)\d u\\
	&=1+a^{1/3}O_a(a^{1/3}\log a)-1\\
	&=O_a(a^{2/3}\log a),
	\end{align*}	
	where we used \eqref{eq:asymp_2}.

	\noindent\underline{Verification of \eqref{eq:asymp_4}:} 
	\begin{align*}
	\int_0^{\infty}e^{-u}\,\frac{a^{1/3}u}{(u+\sqrt{u^2+4a^{2/3}})\sqrt{u^2+4a^{2/3}}}\d u 
	&=\int_0^{\infty}e^{-u}\,\frac{a^{1/3}}{\sqrt{u^2+4a^{2/3}}}\d u\\
	&\qquad-\int_0^{\infty}e^{-u}\,\frac{a^{1/3}}{u+\sqrt{u^2+4a^{2/3}}}\d u\\
	&=O_a(a^{1/3}\log a),
	\end{align*}
	where we used \eqref{eq:asymp_7} and \eqref{eq:asymp_8}.

	\noindent\underline{Verification of \eqref{eq:asymp_5}:} The identity is immediate since $e^{-u}u\log(u)\in L^p([0,\infty))$ for all $p\in [1,\infty]$. 

	\noindent\underline{Verification of \eqref{eq:asymp_9}:} For $a>0$, integration by parts shows
	\[ \int_{0}^{\infty}e^{-u}(u-1)\log (u+\sqrt{u^2+a^{2/3}})\d u=\int_0^\infty  
	e^{-u}\frac{u}{\sqrt{u^2+a^{2/3}}}\d u, \]
	so that to prove the last identity we need to show
	\[ \lim_{a\to 0^+}\frac{1}{a}\int_{0}^{\infty}e^{-u}\biggl(1-\frac{u}{\sqrt{u^2+a^2}}\biggr)\d 
	u=1. \]
	Changing variable $u\rightsquigarrow au$ gives
	\begin{align*}
	\frac{1}{a}\int_{0}^{\infty}e^{-u}\biggl(1-\frac{u}{\sqrt{u^2+a^2}}\biggr)\d u&=
	\int_{0}^{\infty}e^{-au}\biggl(1-\frac{u}{\sqrt{u^2+1}}\biggr)\d u\\
	&=\int_{0}^{\infty}e^{-au}\frac{1}{(u+\sqrt{u^2+1})\sqrt{u^2+1}}\d u,
	\end{align*}
	hence
	\[ 	\lim_{a\to 0^+}\frac{1}{a}\int_{0}^{\infty}e^{-u}\biggl(1-\frac{u}{\sqrt{u^2+a^2}}\biggr)\d 
	u=\int_{0}^{\infty}\frac{1}{(u+\sqrt{u^2+1})\sqrt{u^2+1}}\d u.\]
	Changing variable $u=\sinh t$ we obtain
	\begin{align*}
	\int_{0}^{\infty}\frac{1}{(u+\sqrt{u^2+1})\sqrt{u^2+1}}\d u=\int_0^\infty \frac{1}{\sinh 
		t+\cosh t}\d t=\int_0^\infty e^{-t} \d t=1
	\end{align*} 	
\end{proof}

\section*{Acknowledgments}

We thank Michael Christ for comments and suggestions during the initial stage of this project (2012), and 
Diogo Oliveira e Silva for comments on a preliminary version of this manuscript. Part of this work was carried out at Universidad de los Lagos (Osorno, Chile).
%%%%%%%%%%%%%%%%%%%%%%%%%%%%%%%%%%%%%%%%%%%%%%%%%%%%%%%%%%%%%%%%%%%%%%%

\end{document}